\documentclass{amsart}
\usepackage{epsfig}
\usepackage{amssymb,mathrsfs}

\usepackage[frame,ps,matrix,arrow,curve,rotate,all,2cell,tips]{xy}
\input xy
\xyoption{all}

\usepackage[bookmarksopen,pagebackref, colorlinks, citecolor=blue, linkcolor=blue]{hyperref}

\usepackage[usenames,dvipsnames]{pstricks}

\usepackage{pst-grad} % For gradients
\usepackage{pst-plot} % For axes
\newtheorem{theorem}{Theorem}[subsection]
\newtheorem{proposition}[theorem]{Proposition}

\newtheorem{lemma}[theorem]{Lemma}
\newtheorem{corollary}[theorem]{Corollary}
\newtheorem{pro}[theorem]{Proposition}
\newtheorem{lem}[theorem]{Lemma}
\newtheorem{corol}[theorem]{Corollary}

\theoremstyle{definition}
\newtheorem{definition}[theorem]{Definition}
\newtheorem{defn}[theorem]{Definition}

\newtheorem{example}[theorem]{Example}

\newtheorem{defin}[theorem]{Definition}
\newtheorem{agreement}[theorem]{Agreement}

\theoremstyle{remark}
\newtheorem{remark}[theorem]{Remark}

\numberwithin{equation}{subsection}

\newcommand{\cat}[1]{\mathcal{#1}}

\DeclareMathOperator{\inj}{inj}

\newcommand{\sF}{\mathscr{F}}
\newcommand{\sW}{\mathscr{W}}
\newcommand{\sQ}{\mathscr{Q}}

\DeclareMathOperator{\colim}{colim}

\newcommand{\gint}{\mbox{$\int$}}

\newcommand{\NN}{\mathbb{N}}
\newcommand{\M}{\mathbb{M}}
\newcommand{\bbR}{\mathbb{R}}
\newcommand{\bbL}{\mathbb{L}}

\newcommand{\Sp}{\Sigma}

\newcommand{\MM}{\mathsf{M}}
\newcommand{\Sm}{\mathsf{S}}

\newcommand{\Br}{\mathrm{Br}}

\newcommand{\h}{{\mathcal P}^{{\mathcal P}_{f,g}}}
\newcommand{\ho}{\mathcal{P}^{\mathcal{P}+2\mathcal{A}}}
\newcommand{\hos}{\mathcal{P}^{ 3\mathcal{A}}}

\newcommand{\cop}{\mathcal{P}^{\mathcal{P}+\mathcal{A}}}
\newcommand{\ts}{\mbox{$\mathbf{p}$}}

\newcommand{\td}{\mbox{$\mathbf{t}$}}

\newcommand{\wa}{\mbox{${\bf w}$}}

\newcommand{\qa}{\mbox{${\bf q}$}}
\newcommand{\la}{\mbox{${\bf l}$}}

\newcommand{\ua}{\mbox{${\bf u}$}}

\newcommand{\LX}{\mbox{${\bf L}^{}$}}

\newcommand{\CC}{\mbox{$\mathcal C$}}

\newcommand{\E}{\mbox{$\mathcal {E}^{}$}}

\newcommand{\ON}{\mbox{${O}^{(n)}$}}
\newcommand{\NE}{\mbox{${\mathcal NO}^{(n)}$}}

\newcommand{\boxx}{\hspace{0.5mm}{\unitlength=0.2mm
\begin{picture}(11,2)(440,0)
\put(440,10){\line(0,-1){10}}
\put(440,0){\line(1,0){10}}
\put(450,0){\line(0,1){10}}
\put(450,10){\line(-1,0){10}}
\end{picture}\hspace{0.5mm} }}

\def\Int{\mathrm{Int}}
\def\Ww{{\mathtt{W}}}
\def\VV{\mathbb{V}}
\def\Cc{\mathcal{C}}
\def\NN{\mathbb{N}}

\def\Alg{\mathrm{Alg}}

\def\Pp{\mathcal{P}}

\def\Tt{\mathcal{T}}

\def\Ss{\mathcal{S}}

\def\colim{\mathrm{colim}}
\def\CC{\mathbb{C}}
\def\DD{\mathbb{D}}
\def\Set{\mathbb{S}\mathrm{et}}
\def\FinSet{\mathrm{FinSet}}

\def\Cat{\mathbb{C}\mathrm{at}}
\def\Int{\mathrm{Int}}

\def\inc{\hookrightarrow}

\def\Aa{\mathcal{A}}

\def\SO{\mathcal{SO}}

\def\On{\mathcal{O}_n}

\def\Qn{\mathcal{Q}^{op}_n}

\def\QQ{\mathrm{Q}}
\def\JJ{\mathrm{J}}

\def\SM{\mathcal{S}^{op}}

\def\Dd{\mathcal{D}}
\def\Bb{\mathcal{B}}
\def\Uu{\mathcal{U}}

\def\AA{\mathbb{A}}
\def\BB{\mathbb{B}}

\def\Ho{\mathbb{H}\mathrm{o}}

\title{Homotopy theory of algebras of substitudes and their localisation}
\author{Michael Batanin}
\address{Institute of Mathematics of Czech Academy of Sciences, Zitna 25, Prague 1, Czech Republic}
\email{bataninmichael@gmail.com}

\author{David White}
\address{Department of Mathematics and Computer Science \\ Denison University
\\ Granville, OH 43023}
\email{david.white@denison.edu}

\begin{document}

\begin{abstract}

We study the category of algebras of substitudes (also known to be equivalent to  the regular patterns of Getzler and operads coloured by a category) equipped with a (semi)model structure lifted from the model structure on the underlying presheaves. 
We are especially interested in the case when the model structure on presheaves is a Cisinski style localisation with respect to a proper Grothendieck fundamental localiser. For example, for  $\Ww= \Ww_{\infty}$  the minimal fundamental localiser, the local objects in such a localisation are locally constant presheaves, and local algebras of substitudes are exactly algebras whose underlying presheaves are locally constant.     

We investigate when this localisation has nice properties. We single out a class of such substitudes which we call left localisable and show that the substitudes for $n$-operads, symmetric, and braided operads are in this class. As an application we develop a homotopy theory of higher braided operads and prove a stabilisation theorem for their $\Ww_k$-localisations. This theorem implies, in particular, a generalisation of the  Baez-Dolan stabilisation hypothesis for higher categories.

%1991 Math. Subj. Class.  18D20 , 18D50, 55P48     
\end{abstract}


\maketitle \setcounter{tocdepth}{1} 
 \setcounter{tocdepth}{2}
     {\small \tableofcontents}

\

\section{Introduction}

In the paper \cite{batanin-berger} of Batanin and Berger a general approach to the homotopy theory of algebras over polynomial monads was proposed. 
The category of polynomial monads is equivalent to the category of symmetric $\Sigma$-free operads in $\Set$ \cite{GK}.    Algebras of such monads   include covariant presheaves, monoids, and many types of generalised operads (nonsymmetric, symmetric, cyclic, modular, $n$-operads, etc.). It makes sense to consider the category of algebras $\Alg_\Tt(\VV)$ of a polynomial monad $\Tt$  with values in an arbitrary symmetric monoidal category  $\VV.$ If $\VV$ is equipped with a model  structure  one can try to invert those morphisms in $\Alg_\Tt(\VV)$ which are weak equivalences on underlying collections. This way we obtain the homotopy category  of algebras $\Ho\Alg_\Tt(\VV).$ 

An important problem is whether we can realise $\Ho\Alg_\Tt(\VV)$ as the homotopy category of a model category. %An affirmative answer has real applications, as discussed in Section \ref{sec:Baez-Dolan}.
To solve the realisation problem, we can try to transfer the product model structures along the forgetful functor
$$\eta^*_0:\Alg_\Tt(\VV)  \to [T_0,\VV] \cong \prod_{Ob(T_0)} \VV$$
where $T_0$ is the discrete category on the objects set (colours) of $\Tt$ and  $[A,\VV]$ stands for the category of covariant presheaves on a small category $A$ with values in $\VV.$ In this transfer process we define a morphism $f\in \Alg_\Tt(\VV)$ to be a weak equivalence or fibration provided $\eta^*_0(f)$ is a weak equivalence or fibration in $[T_0,\VV].$ In the  paper \cite{batanin-berger} some very general model theoretical conditions on $\VV$ and combinatorial conditions on $\Tt$ were found which guarantee that the transfer process, indeed, leads to a model structure on algebras. Furthermore, this model structure on algebras satisfies several additional good properties \cite{batanin-berger} that we will use.

The transfer approach above, though a very useful method, is not, however, the only desirable way to get a model category structure on algebras of a polynomial monads. Any polynomial monad has an underlying category of unary operations $U(\Tt).$ Therefore, we can choose a small subcategory $A$ of this category (which we can assume to have the same set of objects $T_0$)  and consider a  morphism of polynomial monads $\eta:\Aa\to \mathcal{U}(\Tt) \to \Tt$  and, hence, a restriction functor
$$\eta^*:\Alg_\Tt(\VV)  \to [A,\VV].$$    
The category of presheaves $[A,V]$ can have a model structure of its own and we can try to transfer this model structure to the category of algebras. For example, we can first transfer a model structure from $[T_0,\VV]$ to $[A,\VV].$ The resulting model category $[A,V]_{proj}$  is known as the \textit{projective} model structure. If we transfer the projective model structure along $\eta^*$ we  get exactly the same structure as  transferred along $\eta^*_0.$  

The projective model structure is not the only model structure we can seek to transfer. For example, $[A,\VV]$ can have a Reedy model structure. This was considered by Benoit Fresse  \cite[Part 2, \S 8.3]{Fresse} in the context of transferred model structures on operads. 
 
In this paper we explore yet another possibility. One can take $[A,\VV]_{proj}$ and  take a left Bousfield localisation of  it  with respect to an appropriate set of morphisms. One can ask then, if this local model structure is transferrable to $\Alg_{\Tt}(\VV)$. If it is, the result of the transfer is itself a Bousfield localisation of the category of algebras \cite{batanin-white-eilenberg-moore}.

This question in general does not, probably, admit a satisfactory answer. Though, if we restrict the class of localisations we get some reasonable conditions on the monad $\Tt$ with a distinguished subcategory of unary operations $A.$ We consider what we call \textit{Cisinski localisations} of the category of presheaves. For this we fix a class $\Ww$, of functors between small categories, that satisfies certain conditions. Namely $\Ww$ must be a proper  fundamental localiser in the sense of Grothendieck \cite{cis06,Maltsiniotis}. In this case there is a localisation $[A,\VV]^\Ww_{proj}$ of the category $[A,\VV]_{proj}$ whose local objects are so-called \textit{$\Ww$-locally constant presheaves}. For example, if $\Ww =\Ww_\infty$ is the minimal fundamental localiser \cite{cis06} then the $\Ww_\infty$ locally constant presheaves are exactly those presheaves $X:A\to \VV$ for which $X(f)$ is a weak equivalence  for any morphism $f$ in $A$ \cite{cis}. The homotopy  category  $\Ho[A,\VV]^{\Ww_\infty}_{proj}$ is therefore the homotopy category of $\infty$-category of representations of the  groupoid $\Pi_\infty(A)$ in $\VV.$ If a transfer of this localisation to $\Alg_\Tt(\VV)$ exists, the corresponding homotopy category  of algebras can be thought as algebras of $\Tt$ in which  unary operations from $A$ are invertible up to all higher homotopies. 

\begin{remark} A different approach to the question of weak inversion of unary operations in case of algebras of one coloured operad was implemented in \cite{tillmann1}. The authors construct an explicit operadic version of Dwyer-Kan hammock localisation  to settle this question.
\end{remark} 

Here is the plan of the paper. 
In Part \ref{part:prelim} we provide some general homotopy theoretical facts  which we will need in the rest of the paper.  
To handle Bousfield localisation in sufficient level of generality we need to leave the environment of  full model structures and allow ourselves to work  in what is known as semimodel category structures. We recall what these are in Section \ref{sec:semi}. 
Transferred semimodel  structures often exist even when full model structures do not.    We recall the main ingredients of the transfer procedure in this section.   Even more striking is the fact that semimodel left Bousfield localisation exists without left properness. This is the content of our paper  \cite{bous-loc-semi} but in Section \ref{sec:loc} we briefly outline this theory.  Section \ref{sec:beck} is devoted to  the theory of homotopy Beck-Chevalley squares which will be our main tool for the comparison of various localised categories of algebras. To our surprise we did not find in the literature some sufficiently elementary results on how the homotopy Beck-Chevalley property can be used to lift Quillen equivalences. So, we include this discussion in this section.

In Part \ref{part:transfer} we come to our main object of study: $\Sigma$-free substitudes. Substitudes in general were defined by Day and Street as a common generalisation of lax-monoidal structures and operads \cite{DS1,DS2}. They are also known to be equivalent to the regular patterns of Getzler \cite{G,BKW} or category-coloured operads \cite{BD3,Pet}. Basically they are coloured operads together with the structure of a small category $A$ on the set of colours. In particular, the category of algebras of a substitude $P$ admits a natural forgetful functor to the category of covariant presheaves on $A.$

A $\Sigma$-free substitude is a substitude in $\Set$ in which the symmetric group action is free.  This is exactly the data for a polynomial monad together with a distinguished set of unary operations as we discussed at the beginning of the Introduction.  In Section \ref{sec:convolution} we recall the definition of the Day-Street convolution operation for substitudes, which will play a primary technical role in our paper. Section \ref{sec:classifiers} reminds the reader about the theory of internal algebra classifiers developed through many years in \cite{SymBat,EHBat,batanin-berger,batanin de leger}. And in Section \ref{sec:unary} we generalise to $\Sigma$-free substitudes the theory of tame polynomial monads developed in \cite{batanin-berger}. Tameness is the main combinatorial criteria used in \cite{batanin-berger} to handle  the question of transfer. Here, we generalise it to the notion of {\it unary tame substitude} to take into consideration the action of unary operations. 

Our main observation (Proposition \ref{retract}) is that for a unary tame substitude the semifree coproduct (a coproduct of an algebra $X$ with a free algebra on an $A$-presheaf $K$) is a natural retract of a certain convolution of the underlying presheaf $\eta^*(X)$ and $K.$ This allows us to deduce good homotopy theoretical properties of the semifree coproduct (and as a consequence good homotopy properties of free algebra extensions) from good properties of the convolution. In Section \ref{sec:transfer} we use this fact to get a general transfer theorem (Theorem \ref{semitransfer}) from model structures on the category of presheaves to semimodel structures on the category of algebras, provided the convolution operation of the substitude is a left Quillen functor of many variables.
The proof of Theorem \ref{semitransfer} follows the same basic plan as the proof of the transfer theorem for tame polynomial monads from \cite{batanin-berger}, but also can be considered as a lax-monoidal generalisation of the Schwede-Shipley proof of the transfer theorem for monoids \cite{SS}.

In Part \ref{part:localization} we apply the Transfer Theorem to the localised category of presheaves. In Section \ref{sec:localisers} we first construct the model categories of $\Ww$-locally constant presheaves  following the ideas of Cisinski from \cite{cis} where such a localisation was constructed for the minimal fundamental localiser $\Ww_\infty.$ 

We come back to the localisation of categories of algebras in Section \ref{sec:loc-alg}. A ${P}$-algebra $X$ with value in a symmetric monoidal model category $\VV$ is called $\Ww$-locally constant if its underlying presheaf $\eta^*(X):A\to \VV$ is a $\Ww$-locally constant presheaf. In general it is a difficult question whether there exists a left Bousfield localisation of the category of $P$-algebras with good properties, whose local objects are exactly  $\Ww$-locally constant algebras.  We call a substitude \textit{left localisable} (Definition \ref{defleftloc}) if such a Bousfield localisation exists for any symmetric monoidal combinatorial model category $\VV$ with cofibrant unit and any proper fundamental localiser $\Ww.$ The significance of this notion for us (Corollary \ref{BC_for_W}) stems from the fact that the local semimodel category of algebras over left localisable substitudes behave especially well with respect to the Beck-Chevalley morphisms of substitudes defined in Subsection \ref{BCsub}.     

We then show using our Transfer Theorem that a substitude is left localisable  provided  it is unary tame and satisfies one more combinatorial condition.  At the end of  Section \ref{sec:loc-alg} we consider our first examples of left localisable substitudes. In particular, we show that Cisinski's model structure for $\Ww$-locally constant presheaves on a category $C$ can be constructed relatively to any subcategory $A$ of $C$ and, therefore, one can ``weakly" invert only a subset of morphisms of $C$ in full analogy with the classical categorical localisation. We also consider a monoidal version of Cisinski's localisation as an example of the application of our methods.

Part \ref{part:n-operads} is our main application. After reviewing preliminary definitions in Section \ref{sec:n-operads}, we solve the relevant realisability problem in Section \ref{sec:locally constant n operads}, producing a model structure for locally constant $n$-operads. Using the criterion from Section \ref{sec:loc-alg} we demonstrate in Section \ref{sec:examples} that substitudes whose algebras are symmetric, braided, or $n$-operads in the sense of the first author \cite{EHBat}  are left localisable. Using this, we describe the category of higher braided operads for every $n\ge 0$ as a localisation of the category of $n$-operads. For the minimal localiser $\Ww_\infty$ the local objects of this category are exactly the $n$-operads on which quasibijections of $n$-ordinals act as weak equivalences. Since the homotopy type of the category of quasibijections is the same as the homotopy type of unordered configuration spaces of points in $\mathbb{R}^n$ these locally constant $n$-operads play the role of higher braided operads. In Section \ref{sec:stablization}, we use these results to prove a stabilisation theorem for these operads from which the Baez-Dolan stabilisation hypothesis for higher categories is a consequence. This completes the promise we made in \cite{batanin-white-CRM} and \cite{white-oberwolfach}. Our methods are very general, and so we are able to prove that the Baez-Dolan stabilization hypothesis is true in Rezk's model of $\Theta_n$-spaces, in Ara's model of $n$-quasi-categories, in models of Bergner and Rezk for Segal and complete Segal objects in $\Theta_{n-1}$-spaces, and in Simpson's models (which include Tamsamani weak $n$-categories and higher Segal categories).

\part{Some general homotopy theory} \label{part:prelim}

In this part, we recall a few definitions and results from abstract homotopy theory, that we will require. We assume the reader is familiar with model categories at the level discussed in \cite{hirschhorn, hovey-book}.

\section{Semimodel categories} \label{sec:semi}

In this section we set up our definition of semimodel category and recall the techniques of transferring of semimodel structures along a right adjoint.
%For further details on these topics, we refer the reader to \cite{barwickSemi, fresse-book, goerss-hopkins, hovey-monoidal, spitzweck-thesis, white-localization, white-yau} and to \cite{hirschhorn}. 

\subsection{Definition of semimodel category} Our definition of a semimodel category (sometimes written `semi-model category'), is taken from \cite{Barwieck} (and slightly generalising Spitzweck's notion of a $J$-semi model category \cite{spitzweck}). We follow the definition with an explanation of what semimodel categories are good for, and how semimodel categories arise. Recall that, for a set of morphisms $S$, $\inj S$ refers to the class of morphisms having the right lifting property with respect to $S$.

\begin{defn} \label{defn:semi}
A \textit{semimodel structure} on a category $\M$ consists of classes of weak equivalences $\sW$, fibrations $\sF$, and cofibrations $\sQ$ satisfying the following axioms:

\begin{enumerate}
\item[M1] Fibrations are closed under pullback.
\item[M2] The class $\sW$ is closed under the two out of three property.
\item[M3] $\sW,\sF,\sQ$ are all closed under retracts.
\item[M4] 
\begin{enumerate}
\item[i] Cofibrations have the left lifting property with respect to trivial fibrations.
\item[ii] Trivial cofibrations whose domain is cofibrant have the left lifting property with respect to fibrations.
\end{enumerate}
\item[M5] 
\begin{enumerate}
\item[i] Every map in $\M$ can be functorially factored into a cofibration followed by a trivial fibration. 
\item[ii] Every map whose domain is cofibrant can be functorially factored into a trivial cofibration followed by a fibration.
\end{enumerate} 
\end{enumerate}

If, in addition, $\M$ is bicomplete, then we call $\M$ a \textit{semimodel category}.
$\M$ is said to be \textit{cofibrantly generated} if there are sets of morphisms $I$ and $J$ in $\M$ such that $\inj I$ is the class of trivial fibrations, $\inj J$ is the class of fibrations in $\M$, the domains of $I$ are small relative to $I$-cell, and the domains of $J$ are small relative to maps in $J$-cell whose domain is cofibrant. We will say $\M$ is \textit{combinatorial} if it is cofibrantly generated and locally presentable.
\end{defn}

Note that, in a semimodel category $\M$, the axioms of a full model structure are satisfied on the subcategory of cofibrant objects. Furthermore, $\M$ has a cofibrant replacement functor defined on every object. Consequently, every result about model categories has a semimodel categorical analogue, usually obtained by cofibrantly replacing as needed. This includes the Fundamental Theorem of Model Categories (characterising morphisms in the homotopy category), left and right Quillen functors, Ken Brown's lemma, path and cylinder objects, the retract argument, the characterisation of cofibrations (resp. trivial cofibrations with cofibrant domains) as morphisms with the left lifting property with respect to trivial fibrations (resp. fibrations), the cube lemma, simplicial mapping spaces, hammock localisation, projective/injective/Reedy semimodel structures, latching and matching objects, cosimplicial and simplicial resolutions, computations of homotopy limits and colimits, and more. Many examples are detailed in \cite{bous-loc-semi}. We remark that a pair of adjoint functors between semimodel categories is called a Quillen pair if the right adjoint preserves fibrations and trivial fibrations. In practice, a semimodel structure is just as useful as a full model structure. 

Definition \ref{defn:semi} departs slightly from the definitions given in \cite{Barwieck} and \cite{spitzweck}, where the authors include an axiom that the initial object is cofibrant. This axiom is redundant: the statement can easily be deduced from a factorisation, lifting, and retract argument applied to the identity morphism on the initial object.

\subsection{Transfer of semimodel structures} 

Suppose $\M$ is a semimodel category, $\Tt$ is a monad on $\M$, $U$ is the forgetful functor from $\Tt$-algebras to $\M$ and $F$ is its left adjoint. For the \textit{transferred (semi)model structure} on $\Tt$-algebras, we define a morphism $f$ to be a weak equivalence (resp. fibration) if and only if $U(f)$ is a weak equivalence (resp. fibration) in $\M$.  The class of cofibrations is defined by the left lifting property with respect to trivial fibrations. An algebra $X$ is cofibrant provided the unique morphism $F(0)\to X$ is a cofibration, where $0$ is the initial object in $\M.$ A simple adjunction argument shows $F(0)$ is itself cofibrant. 

We now recall a condition is required for producing functorial factorisations in the category of $\Tt$-algebras. Let $I$ (resp. $J$) denote the set of generating (trivial) cofibrations of $\M$. 
%Recall that $J$ is contained in the cofibrations, and it is easy to see that $FJ$ is contained in $FI$-cof (see Theorem \ref{thm:semi-transfer-general} below). 
Following \cite[Section 12.1.3]{fresse-book}, we say $FI$ (resp. $FJ$) \textit{permits the small object argument} if the domains of morphisms in $FI$ (resp. $FJ$) are small relative to the class of relative $FI$-cell complexes (resp. relative $FJ$-cell complexes with cofibrant domain). 

Recall that, for a set of cofibrations $\cat K$, the class of \textit{relative $\cat K$-cell complexes} is the class of morphisms obtained as transfinite compositions of pushouts of morphisms in $\cat K$. When we speak of a \textit{$\cat FJ$-cell complex with cofibrant domain}, this means we are considering pushouts of $\Tt$-algebras of the form 
\begin{align} \label{free-first}
\xymatrix{
F(K) \ar[r]^{F(f)} \ar[d]_{g} & F(L) \ar[d]^{} \\
X \ar[r]^{p} & B
}
\end{align}
where $X$ is a \textit{cofibrant} $\Tt$-algebra, and $f:K\to L$ is in $J$. When we speak of a \textit{relative $\cat FJ$-cell complex with cofibrant domain}, we mean a transfinite composition of morphisms $p$ as above. In particular, each object in the transfinite composition must be cofibrant (because the morphisms $F(f)$ are cofibrations, as we will see in Theorem \ref{thm:semi-transfer-general}). 

There are several semimodel category existence theorems already in the literature \cite[Theorem 12.1.4]{fresse-book}, \cite[Theorem 2]{spitzweck}, \cite[Theorem 2.2.2]{white-yau1}, each producing a slightly different variant of a semimodel structure. Each of these theorems starts with a base model category $\M$. The following existence theorem starts with a base semimodel category $\M$, but its proof is exactly like the previous existence theorems.

\begin{theorem} \label{thm:semi-transfer-general}
Let $\M$ be a cofibrantly generated semimodel category, with sets of generating (trivial) cofibrations $I$ and $J$, such that $U$ preserves filtered colimits. Assume that the sets $FI$ and $FJ$ permit the small object argument, and that relative $FJ$-cell complexes with cofibrant domain are weak equivalences.

Then the category of algebras $\Alg_\Tt$ admits a transferred cofibrantly generated semimodel structure with generating (trivial) cofibrations $FI$ and $FJ$.
\end{theorem} 

\begin{proof}
The proof proceeds just like the transfer theorem for model categories \cite[Theorem 11.1.13]{fresse-book}. The class of weak equivalences is defined as the class of morphisms $f$ such that $Uf$ is a weak equivalence in $\M$. The class of fibrations is defined as $FJ$-inj, i.e., morphisms with the right lifting property with respect to $FJ$. The class of cofibrations is defined as $FI$-cof, i.e., morphisms with the left lifting property with respect to $FI$-inj. An adjunction argument tells us $FI$-inj is exactly the class of trivial fibrations (hence, axiom (M4i) is automatic), and that $f$ is a fibration if and only if $U(f)$ is a fibration. Since $J$ is contained in the cofibrations of $\M$, a lifting and adjunction argument shows that $FJ$ is contained in the cofibrations of $\Alg_{\Tt}$, using \cite[Lemma 1.7]{Barwieck} (it follows that $FJ$-cell is contained in the cofibrations). Axioms (M1), (M2), and (M3) are inherited from $\M$. Axiom (M5i) follows from the small object argument and \cite[Lemma 2.1.10]{hovey-book}. Axiom (M4ii) follows from (M4i) and (M5ii) just as in \cite[Theorem 11.1.13]{fresse-book} (adding the requirement that the domain be cofibrant). Lastly, Axiom (M5ii) follows just as in \cite[Theorem 11.1.13]{fresse-book}, via the small object argument in $\Alg_\Tt$. The factorisation is only required for morphisms with cofibrant domain, so the left leg is a relative $FJ$-cell complex with cofibrant domain, hence a weak equivalence by our assumption (and a cofibration, as we have already observed). The right leg is in $FJ$-inj, so is a fibration by definition.
\end{proof}

The conditions of Theorem \ref{thm:semi-transfer-general} have been verified for many categories of algebras over coloured operads \cite{white-yau1}. For general monads, it can be hard to verify that $FJ$-cell complexes with cofibrant domain are weak equivalences. Proposition \ref{prop:semi-helper} provides one technique for doing so. The smallness conditions in Theorem \ref{thm:semi-transfer-general} (i.e., permitting the small object argument) are automatically satisfied in locally-presentable settings. However, the result below provides a way to side-step the smallness conditions, and also to learn about when the forgetful functor $U$ preserves cofibrant objects. This result goes back to \cite[Theorem 2]{spitzweck} (it is the reason for the assumption on $\mathbb{T}I$-cell therein) and \cite[Corollary 5.2]{BergerMoerdijk}. It is also included in \cite[Theorem 12.1.4]{fresse-book} and \cite[Proposition 6.2.5]{white-yau1}.

\begin{theorem} \label{prop:semi-helper}
Let $\M$ be a cofibrantly generated semimodel category, with sets of generating (trivial) cofibrations $I$ and $J$, such that $U$ preserves filtered colimits. Assume that, for every (trivial) cofibration $i: K\to L$ between cofibrant objects in $\M$, and every pushout
\begin{align} \label{free cofibration}
\xymatrix{
F(K) \ar[r]^{F(i)} \ar[d]_{g} & F(L) \ar[d]^{} \\
A \ar[r]^{p} & B
}
\end{align}
where $A \in \Alg_\Tt$ is an $FI$-cell complex such that $U(A)$ is cofibrant in $\M$, the morphism $U(p)$ is a (trivial) cofibration in $\M$. Then 
\begin{enumerate}
\item $\Alg_\Tt$ admits a transferred cofibrantly generated semimodel structure with generating (trivial) cofibrations $FI$ and $FJ$;
\item The functor $U$ maps cofibrations with a cofibrant domain to cofibrations;
\item If $\Tt(0)$ is a cofibrant object in $\M$ then $U$ preserves cofibrant objects.
\end{enumerate} 
\end{theorem}

\begin{proof}
The proof proceeds exactly as in \cite{fresse-book}. The categorical argument in \cite[Lemma 12.1.5]{fresse-book} demonstrates that the assumption in (\ref{free cofibration}) holds for every (trivial) cofibration, even if $K$ and $L$ are not cofibrant. That $FI$ (resp. $FJ$) permits the small object argument follows from a cellular extension argument in $\M$, the fact that $I$ (resp. $J$) permits the small object argument in $\M$, and an adjunction argument. For the case of $FJ$, it is easy to check that all the objects $U(B_i)$ and $L_i$ in \cite[Proposition 11.1.14]{fresse-book} are cofibrant in $\M$, under our hypothesis on $A$, and hence that the required factorisations exist in $\M$. The key point for the existence of the semimodel structure on $\Alg_\Tt$ is that the $J$ part of the assumption in (\ref{free cofibration}) implies the hypothesis that relative $FJ$-cell complexes with cofibrant domain are weak equivalences, so we can use Theorem \ref{thm:semi-transfer-general}. As for the verification of (2) and (3), this proceeds exactly as in \cite[Proposition 11.1.14]{fresse-book}, and only relies on the $I$ part of the assumption in (\ref{free cofibration}). 
\end{proof}

The verification of the hypotheses of Proposition \ref{prop:semi-helper} is easier when the domains of $I$ and $J$ are cofibrant, as we can then reduce to analyzing (\ref{free cofibration}) for maps $i$ in $I$ and $J$. 
Procedures for creating cofibrantly generated semimodel structures, where the domains of the generating trivial cofibrations $J$ are cofibrant, can be found in \cite[Theorem B]{bous-loc-semi} and in \cite[Theorem 6.3.1]{white-yau1}.

\section{Left Bousfield localisations and their liftings} \label{sec:loc}

Left Bousfield localisation is a fundamental tool that allows us to study the homotopy theory of a model category $\M$ after a chosen class of morphisms $\cat C$ is homotopically inverted \cite{hirschhorn}. Formally, a left localisation of a model category $\M$ with respect to a chosen class of morphisms $\cat C$ is a new model category $L_{\cat{C}} \M$ and a left Quillen functor $j: \M \to L_{\cat C} \M$ that is the universal left Quillen functor out of $\M$ that takes morphisms in $\cat C$ to weak equivalences. We call the weak equivalences of $L_{\cat C} \M$ the\textit{ $\cat C$-local equivalences}. A particular construction of $L_{\cat C} \M$ is given by left Bousfield localisation, where $j$ is the identity on $\M$, the cofibrations of $L_{\cat C}\M$ are the same as those of $\M$, where the $\cat C$-local equivalences are defined from $\cat C$ using simplicial mapping spaces \cite[Chapter 3]{hirschhorn}, and where the $\cat C$-local fibrations are defined via the right lifting property. 

\begin{defn} \label{defn:admits a localisation and its lifting}
We will say that a model category $\M$ {\em admits a localisation} with respect to a chosen class of morphisms $\cat C$ if the classes of $\cat C$-local equivalences, cofibrations, and $\cat C$-local fibrations, satisfy the axioms of a semimodel category. 
\end{defn}

Local equivalences and local objects are defined with respect to simplicial mapping spaces (which exist even if $\M$ is only a semimodel category). Recall from \cite[Chapter 3]{hirschhorn} that an object $W$ is called \textit{$\cat C$-local} if $map(f,W)$ is a weak equivalence of simplicial sets for all $f\in \cat C$. And, a morphism $g$ in $\M$ is a \textit{$\cat C$-local equivalence} if $map(g,W)$ is a weak equivalence for all $\cat C$-local objects $W$. We will always assume we are localising a set $\cat C$ of cofibrations, but this is no loss of generality, thanks to cofibrant replacement.

As is the case for transferred model structures, $L_{\cat C} \M$ admits a semimodel structure much more frequently than a model structure. In our setting, the model categories of $n$-operads are not known to be left proper, but they still admit a localisation in the semimodel categorical sense \cite{bous-loc-semi}, and this is enough for our purposes. For the reader's convenience, we restate the main theorem from \cite{bous-loc-semi}, which we will need in Section \ref{sec:examples}.

\begin{theorem} \label{thm:bous-loc-semi}
Suppose that $\M$ is a combinatorial semimodel category whose generating cofibrations have cofibrant domain, and $\cat C$ is a set of morphisms of $\M$. Then there is a semimodel structure $L_{\cat C}(\M)$ on $\M$, whose weak equivalences are the $\cat C$-local equivalences, whose cofibrations are the same as $\M$, and whose fibrant objects are the $\cat C$-local objects. Furthermore, $L_{\cat C}(\M)$ satisfies the universal property that, for any left Quillen functor of semimodel categories $F:\M\to \NN$ taking $\cat C$ into the weak equivalences of $\NN$, then $F$ is a left Quillen functor when viewed as $F:L_{\cat C}(\M)\to \NN$.
\end{theorem} 

\begin{remark} The following notations for a pair of  adjoint functors will be often use in our paper: $\alpha^*\vdash\alpha_! .$ Here $\alpha^*$ is a right adjoint and $\alpha_!$ is its left adjoint.  \end{remark}

We also have the following useful generalisation of  \cite[Theorem 3.3.20]{hirschhorn}:
\begin{theorem}\label{loclifting} Let $\M$ be a semimodel category and $\cat C$ be a set of morphisms between cofibrant objects in $\M$   such that the left Bousfield localisation $\M\to L_{\cat{C}}(\M)$ exists. Let $\NN$ be a combinatorial semimodel categories whose  generating cofibrations have cofibrant domains, and let
 $$\xymatrix@C=3em{
\NN \ar@<-1pt>[r]_{\alpha^*} & \ar@<-3.5pt>[l]_{\alpha_!} \M  }$$ 
 be a Quillen adjunction. Then \begin{enumerate} \item There is a localisation $\NN\to L_{\alpha_!\cat{C}}(\NN)$ called the lifting of $\M\to L_{\cat{C}}(\M)$ along $\alpha^*$    which makes the following commutative diagram a diagram of Quillen adjunctions:
\begin{equation*}% \label{H:lifting}
\xymatrix@R=3em@C=3em{
\NN \hspace{1.5mm} \ar@<2.5pt>[r]^{id\hspace{2mm}} \ar@<-0.5pt>[d]^{\alpha^*}
& L_{\alpha_!\cat{C}}(\NN)  \ar@<2.5pt>[l]^{id\hspace{2mm}} \ar@<-2.5pt>[d]^{\beta^*} \\
\M \hspace{1.5mm} \ar@<2.5pt>[r]^{id}\ar@<4.5pt>[u]^{\alpha_!}  
&L_{\cat C}(\M); \ar@<2.5pt>[l]^{id} \ar@<6.5pt>[u]^{\beta_!}
}
\end{equation*}
\item The functor $\alpha^*$ reflects and preserves local fibrant objects. That is, a fibrant object $X\in \NN$ is local  if and only if $\alpha^*(X)$ is local in $\M;$

\item If $\alpha^*\vdash\alpha_! $ is a pair of Quillen equivalences then $\beta^*\vdash\beta_! $ is also a pair of Quillen equivalences. \end{enumerate}
      \end{theorem} 
      
\begin{proof} The proof of this theorem follows exactly the same patterns as the proof of  \cite[Theorem 3.3.20]{hirschhorn}. The existence of the lifted  localisation is guaranteed by Theorem \ref{thm:bous-loc-semi}.  The proof of the statement about local objects is identical to the proof of Lemma 3.3 from \cite{batanin-white-eilenberg-moore}, where it was proved for the forgetful functor from the category of algebras of a monad, but the argument relies  only on the adjunction. 
%The local objects in this semimodel categories are exactly      as observed in \cite{batanin-white-eilenberg-moore}.
\end{proof}

 \section{Beck-Chevalley squares} \label{sec:beck}
 
  \subsection{Beck-Chevalley squares and homotopy Beck-Chevalley squares}

Recall \cite{Malt} that a square of right adjoints and a natural transformation as displayed below 
 
 \begin{equation}\label{BScondition}\xymatrix@R=0.7em@C=0.7em{
\AA %\ar@<2.5pt>[r]^{\psi_!} 
%\ar@2{->}[rd]^b
%{\ar@{}[dr]|(.7){\Searrow}}
\ar@<0pt>[ddd]_{\beta^*}
& & &
\BB \ar@<0pt>[lll]_{\psi^*}
 \ar@<0pt>[ddd]^{\alpha^*} \\
  &\ar@2{->}[rd]^b &  & \\ 
  & & & \\
\CC %\ar@<2.5pt>[r]^{\phi_!}
% \ar@<2.5pt>[u]^{\beta_!}  
& & &
\DD \ar@<0pt>[lll]_{\phi^*} 
%\ar@<2.5pt>[u]^{\alpha_!}
}
\end{equation}

\noindent is called Beck-Chevalley if the natural transformation
\begin{equation}\label{BC transformation} \mathbf{bc}:  \phi_!  \beta^* \to  \alpha^*  \psi_! \end{equation}
is an isomorphism. 
The natural transformation $\mathbf{bc}$ is defined as vertical pasting  
 \begin{equation*}
\xymatrix@R=0.5em@C=1em{
\AA %\ar@<2.5pt>[r]^{\psi_!} 
\ar@<0pt>[ddd]_{id}
%\ar@2{->}@<1pt>[rd]^{\epsilon}
& & &
\AA \ar@<0pt>[lll]_{id}
 \ar@<0pt>[ddd]^{\psi_!} \\
  &\ar@2{->}[rd]^\epsilon &  & \\ & & & \\
 \AA 
%\ar@2{->}@<1pt>[rd]^b
\ar@<0pt>[ddd]_{\beta^*} 
& & &
\BB \ar@<0pt>[lll]_{\psi^*}
 \ar@<0pt>[ddd]^{\alpha^*}                      
                           \\
 &\ar@2{->}[rd]^b &  & \\ & & & \\                                                    
\CC  \ar@<0pt>[ddd]_{\phi_!} 
%\ar@2{->}@<1pt>[rd]^{\eta}
& & &
\DD \ar@<0pt>[lll]_{\phi^*} 
  \ar@<0pt>[ddd]^{id}         \\
 &\ar@2{->}[rd]^\eta &  & \\ & & & \\   
\CC
& & &
\DD \ar@<0pt>[lll]_{id} 
}
\end{equation*} 
where $\eta$ is the counit and $\epsilon$ is the unit of corresponding adjunctions.

\begin{agreement} In all our applications the natural transformation $b$ is the identity, so we will assume in the rest of the paper that the square of right adjoints (\ref{BScondition}) strictly commutes. 
\end{agreement} 

The following property of Beck-Chevalley squares allows us to lift equivalences along right adjoints. 

\begin{proposition} \label{lemma:Beck-Chevalley implies E}
Let  (\ref{BScondition}) 
be a  Beck-Chevalley square. Assume that $(\phi_!,\phi^*)$ is a pair of adjoint  equivalences, and that the right adjoint functors $\beta^*$ and $\alpha^*$ reflect isomorphisms. Then $(\psi_!,\psi^*)$ is a pair of adjoint equivalences.
\end{proposition}

\begin{proof}
We must only prove that the unit and counit for $(\psi_!,\psi^*)$ are isomorphisms. 

We begin with the counit. The Beck-Chevalley condition tells us that, in $\DD$, $ \alpha^*  \psi_! Y \cong  \phi_! \beta^* Y$ for every $Y$ in $\AA$. It follows that $ \alpha^* \psi_!  \psi^* X \cong \phi_! \beta^*  \psi^* X \cong  \phi_!  \phi^*  \alpha^* X$ for any $X \in \BB$ and this is isomorphic to $ \alpha^* X$ because $(\phi_!,\phi^*)$ is an  equivalence. Since $\alpha^*$ reflects isomorphisms, the counit $ \psi_!  \psi^* X \to X$ is an isomorphism.

We turn to the unit. We first apply $\beta^*$, and then observe that $\beta^* (\psi^* \psi_! X)\cong \phi^* \alpha^* \psi_! X$. The Beck-Chevalley condition tells us that the latter is isomorphic to $\phi^* \phi_!  \beta^* X$. Again using that $(\phi_!,\phi^*)$ is an adjoint equivalence, we see that $\beta^*X\to \phi^*  \phi_!   X \cong \beta^*  \psi^*  \psi_! X$ is an isomorphism. Since $\beta^*$ reflects isomorphisms, the composite isomorphism above shows that the unit $X\to  \psi^*  \psi_! X$ is an isomorphism, as required.
\end{proof}

Suppose now (\ref{BScondition}) is a commutative square 
 of Quillen adjunctions between semimodel categories.

\begin{definition}We say that the square above is a homotopy Beck-Chevalley square if it generates a Beck-Chevalley square of homotopy categories, that is, the morphisms between the derived functors
$$   \bbL \phi_! \bbR \beta^* (-)\to \bbR \alpha^* \bbL \psi_! (-)$$
is an isomorphism.

\end{definition}

The next proposition gives a  practical criteria for recognition of homotopy Beck-Chevalley squares.

\begin{proposition}\label{BC_implies_hBC} Let (\ref{BScondition})   be a Beck-Chevalley square  and let the  functor $\alpha^*$ preserve weak equivalences and $\beta^*$ preserve cofibrant objects. Then  (\ref{BScondition})  is also  a homotopy Beck-Chevalley square.
\end{proposition}
\begin{proof} % We are reduced to proving that the derived Beck-Chevalley { natural transformation} $$\bbL g_! \bbR U_{\mathcal{E}}(X) \cong \bbR U_{\mathcal{F}} \bbL f_!(X)$$ is an isomorphism.
 Let $X$ be a fibrant and cofibrant object from $\AA.$ Because $\beta^*$ preserves cofibrancy, we see that $\bbL \phi_! \bbR \beta^*(X)  \cong \phi_! \beta^*(X)$. Also, $\bbR \alpha^*  \bbL \psi_! (X) \cong \bbR \alpha^*  \psi_! (X)$ because $X$ is cofibrant. The object $\psi_! (X)$ need not be fibrant, so we must compute $\bbR \alpha^*  \psi_! (X)$ via the fibrant replacement $\psi_! (X)\to R \psi_! (X)$, which induces a morphism $\alpha^*  \psi_!  X \to \alpha^*  R \psi_! X$. Because $\alpha^*$ preserves weak equivalences, this morphism is a weak equivalence. Hence  the composite $\phi_! \beta^*(X) \to \alpha^*  \psi_! (X) \to \alpha^* R \psi_! (X)$ is a weak equivalence  because the first morphism  is an isomorphism. This finishes the proof.  

\end{proof}

\subsection{Lifting of Quillen equivalences} Immediately from Proposition \ref{lemma:Beck-Chevalley implies E} we have the following property of homotopy Beck-Chevalley squares.

\begin{proposition} \label{lemma:Beck-Chevalley implies QE}
Let a square (\ref{BScondition})
be a homotopy Beck-Chevalley square. Assume that $(\phi_!,\phi^*)$ is a pair of Quillen equivalences, and that the right adjoints $\beta^*$ and $\alpha^*$ reflect weak equivalences between fibrant objects. Then $(\psi_!,\psi^*)$ is a pair of Quillen equivalences.
\end{proposition}

\begin{proof}

From the condition on $\beta^*$ and $\alpha^*$ we see that the right derived functors $\bbR\beta^*$ and $\bbR\alpha^*$ reflect isomorphisms.
\end{proof}

Finally, we combine everything together in the following Theorem.

\begin{theorem}\label{lifting  QE}  Let (\ref{BScondition})  be a Beck-Chevalley square of  Quillen adjunctions in which  $(\phi_!,\phi^*)$ is a pair of  Quillen equivalences, $\alpha^*$ preserves weak equivalences, $\beta^*$ preserves cofibrant objects  and both $\alpha^*$ and $\beta^*$ reflect weak equivalences between fibrant objects. Then the pair $(\psi_!,\psi^*)$ is a pair of Quillen equivalences.    \end{theorem}

\begin{corollary}  \label{lifting QE for transfer} Let (\ref{BScondition})  be a Beck-Chevalley square of  Quillen adjunctions in which   semimodel structures on $\AA$ and $\BB$ are obtained as transfers along  $\alpha^*$ and $\beta^*$ correspondingly and $\beta^*$ preserves cofibrant objects. Then  if  $(\phi_!,\phi^*)$ is a pair of  Quillen equivalences,   the pair $(\psi_!,\psi^*)$ is also a pair of Quillen equivalences. 
\end{corollary} 

%\begin{proof} This follows from the Proposition \ref{lemma:Beck-Chevalley implies QE} below.

%\end{proof}

\part{Semimodel structures for algebras of substitudes} \label{part:transfer}

In this part, we introduce a series of definitions that allow us to encode the type of algebraic structure we are interested in, and we prove a general Transfer Theorem (Theorem \ref{semitransfer}) to endow categories of algebras with transferred semimodel structures.

\section{Substitudes, convolution and polynomial monads} \label{sec:convolution}

 \subsection{Substitudes}
Let $\VV= (\VV,\otimes,I)$ be a symmetric monoidal  category.  We will assume by default that it is cocomplete and closed. For a small $\VV$-category $A$  let $[A,\VV]$ denote the $\VV$-category of  $\VV$-functors.

\begin{definition}[Day-Street \cite{DS2}] 
A {\em $\VV$-substitude\/}  $(P,A)$ is a small 
$\VV$-category  $A$ together
with a sequence of $\VV$-functors:
$$
P^n: \underbrace{A^{op}\otimes\cdots\otimes A^{op}}_{n-times}\otimes  
A \rightarrow \VV, \  n\ge 0,\ $$ %$$P_n(a_1,\ldots,a_n; a) =  
%P_{X_1,\ldots,X_n}^ X
%$$
equipped with
\begin{enumerate}
\item 
a $\VV$-natural family of substitution operations
$$\mu: P^n(a_1,\ldots,a_n; a) \otimes P^{m_1}(a_{11},\cdots,a_{1m_1};  
a_1)\otimes\cdots\otimes P^{m_n}(a_{n1},\ldots,a_{nm_n};a_n)\rightarrow$$ $$\to  
P^{m_1+\ldots+m_n}(a_{11},\ldots,a_{nm_n};a)$$
\item 
a $\VV$-natural family of morphisms (unit of substitude)
$$\eta: A(a_1,a_2)\rightarrow P^1(a_1;a_2);$$
\item 
for each permutation $\sigma\in \Sm_n$ a $\VV$-natural family of  
isomorphisms
$$\gamma_{\sigma}:P^n(a_1,\ldots,a_n;a) \rightarrow P^n(a_{\sigma(1)}, 
\ldots,a_{\sigma(n)};a),$$
\end{enumerate}
satisfying the obvious associativity, unitality and equivariance conditions.  
%\cite{DS2}.
\end{definition}

Notice that $P^1$ is a $\VV$-monad on $A$ in the bicategory of 
$\VV$-bimodules (also known as $\VV$-profunctors or  $\VV$-distributors). The Kleisli category  
of this monad is called {\it the underlying category of $P.$}

\begin{defin} \label{defn:substitute-map}
A morphism of substitudes $(P,A)\to (Q,B)$ is a pair
$(f,g)$ where $g:A\to B$ is a $\VV$-functor and $f$ is a sequence of $\VV$-natural transformations
$$f^n:P^n\to g^*(Q^n)$$
and $g^*$ is the restriction functor along $(g^{op})^n\otimes g$ which respects substitution and unit operations in an obvious sense. 

\end{defin}

\begin{remark}To simplify the notations we will often omit the superscript  in the notation for the functor $P^n(a_1,\ldots,a_n; a).$\end{remark}

The concept of substitude generalises operads and
symmetric lax-monoidal categories. Indeed, any coloured
operad $\E$ in $\VV$  can be naturally considered as a substitude in several different ways \cite{DS1,BKW}. 

One possibility is to consider a substitude $(P(\E), A)$ with $A$ equal to the $\VV$-category of all unary operations in $\E$ and $P(\E)(a_1,\ldots,a_n,a) = \E(a_1,\ldots,a_n,a).$ % with the set
%of objects $Q_0$ and the object of morphisms $Q(X,Y)= \E(X;Y)\in \VV $.
The substitution operation in the coloured operad $\E$ makes this
assignment %$P(\E)(a_1,\ldots,a_n; a) $ 
a $\VV$-functor
$$ 
P(\E): \underbrace{A^{op}\otimes\cdots\otimes
A^{op}}_{n-times}\otimes\ A \rightarrow \VV, \ n\ge 0,\ 
.$$ 
 The category $A = U(\E)$ is
also called the underlying
category of the coloured operad $\E.$ 

In fact, a substitude in general
is a coloured operad $\E$ together with a small $\VV $-category
$A$ and an identity-on-objects $\VV$-functor $\eta: A \rightarrow U(\E)$
\cite[Prop. 6.3]{DS1}. 

Yet another possibility is to consider the full subcategory of substitudes $(P,A)$ for which $A$ is a discrete category. This subcategory  is isomorphic to the category of operads.
This full inclusion functor  has a right adjoint.
Namely, given a  substitude $(P,A)$ one can form a substitude  
$(P_0,A_0),$ where $A_0$ is the  maximal discrete subcategory of $A$ and $P_0$ is the restriction of $P$ to $A_0.$

We refer the reader to \cite{BKW} for a detailed treatment of the relationships between substitudes, regular patterns of Getzler \cite{G}, operads and Feynman categories. 

\begin{defin} Let $(P,A)$ be a $\VV$-substitude. An algebra of $(P,A)$  is a $\VV$-presheaf $X\in [A,\VV]$  equipped with the sequence of 
natural transformations:
$$P(a_1,\ldots,a_n;a)\otimes X(a_1)\otimes\ldots \otimes X(a_n) \to X(a)$$
satisfying natural commutativity, unitarity and equivariance conditions.

A morphism of $(P,A)$-algebras is a morphism of presheaves over $A$ which commutes with all structure maps. $(P,A)$-algebras and their morphisms form a category $\Alg_{(P,A)}(\VV).$    

\end{defin}

Following \cite{DS1,DS2} let us define the 
 convolution  operation. For the presheaves $X_1,\ldots , X_k \in [A,\VV]$ the convolution $k$-th 
 %tensor 
 product is given by the formula
$$\otimes^k_{P}(X_1,\ldots,X_k)(-) = \int^{a_1,\ldots,a_k} P(a_1,\ldots,a_k;-)\otimes X_1(a_1)\otimes\ldots\otimes X_k(a_k) .$$ 

{
 \begin{remark}\label{0convolution} The convolution formula for $k=0$ amounts to $\otimes_{{P}}^0 = P(\emptyset;-)$, that is, the underlying object of the algebra  of nullary operations of $P$, which is the initial object in $\Alg_{P}(\VV).$ For $k=1$ the convolution $\otimes_{{P}}^1(X) = \eta_!(X)$ is the left Kan extension of the presheaf $X$ along the unit of the substitude.
 \end{remark}
  }

Convolution determines a symmetric lax monoidal structure on the category $[A,\VV]$  \cite{DS2}. 
The free $(P,A)$-algebra monad ${\Tt_P}$ on $[A,\VV]$ can be expressed in terms of convolution:
 \begin{equation} \label{geometric series}{\Tt_P}(X) = \coprod_{k\ge 0} \otimes^k_{P}(X,\ldots,X)/\Sm_k.\end{equation}
 It is useful to remember  
 \begin{proposition}\label{3cat} 
 For a closed symmetric monoidal category $\VV$ the following $\VV$-categories are isomorphic:
 \begin{enumerate} 
 \item The category of algebras of the substitude $(P,A);$   
 \item The category of algebras $\Alg_{\Tt_P}(\VV)$ of the monad ${\Tt_P};$  
 \item The category of commutative monoids in $[A,\VV]$ with the lax-monoidal structure determined by the convolution;
 \item The category of algebras of the operad $(P_0,A_0).$
 
 \end{enumerate}  
 \end{proposition}
 
 \begin{proof} The isomorphisms between (1),(2) and(3) can be obtained by simply unpacking the meaning of the structure morphism $ \Tt_P(X)\to X$  for a $\Tt_P$-algebra $X$ (see  \cite[Proposition 1.8]{bbl} for details). % \cite{mcclure-smith,DS1,DS2,batanin-berger}).  
 
The isomorphism between (4) and (2) is obtained by an application of the strong Beck Monadicity Theorem. Indeed, consider the following composite of forgetful functors:
$$ G: Alg_{\Tt_P}(\VV)\to [A,\VV] \stackrel{i^*}{\to} [A_0,\VV],$$ 
where $i_0^*$ is the restriction along the inclusion $i:A_0\to A.$  Since this inclusion is the identity on objects, the functor $i^*$ is monadic. Hence, the composite $G$ is monadic and by Beck's Theorem the category $Alg_{\Tt_P}(\VV)$ is isomorphic to the category of the monad $\Tt$ induced on $[A_0,\VV].$ This last monad is obtained as the composite $\Tt(X)=i^*\Tt_Pi_!(X),$ where $i_!(X)$ is the left Kan extension of $X$ along $i.$ It is now not too difficult to demonstrate that such a composite monad is isomorphic to the monad 
$\Tt_{P_0}(X)$ because of the following formula connecting  convolutions (see \cite[Proposition A6]{bbm}):$$     
i^*(\otimes^k_P(i_!(X_1),\ldots , i_!(X_k))) \simeq  \otimes^k_{P_0}(X,\ldots,X).$$ 
Finally, observe that $$\Tt_{P_0}(X)(a) = \coprod_{k\ge 0} \coprod_{a_1,\ldots a_k} P(a_1,\ldots,a_k;a)\otimes X(a_1)\otimes \ldots \otimes X(a_k)/\Sm_k$$   
 is the classical formula for the monad induced by the coloured operad $(P_0,A_0).$ 
 
 \end{proof}  
 
\begin{corollary} 
The category  of algebras of an operad $\E$ is isomorphic to the category of algebras of the substitude $(P(\E), U(\E)).$  
\end{corollary}

In view of Proposition \ref{3cat}, we will call $(P,A)$-algebras simply $P$-algebras if it does not lead to confusion, and the category of $(P,A)$-algebras will be identified with $\Alg_{\Tt_P}(\VV)$ and denoted $\Alg_P(\VV).$ If $\VV=\Set$ we will write simply $\Alg_P.$

\subsection{Convolution through the Grothendieck construction} \label{subsec:convolution}

From now on we assume that the substitudes we use are $\Set$-based substitudes. Given such a substitude $({P},A)$  one can construct a $\VV$-enriched substitude $$({P}\otimes I,A\otimes I)$$ by 
$$A\otimes I(a,b) = \coprod_{A(a,b)} I \ ; \ P\otimes I(a_1,\ldots,a_k;a) = \coprod_{P(a_1,\ldots,a_k;a)} I .$$
We can then speak about algebras of ${P}$ in $\VV$ (being algebras of $({P}\otimes I,A\otimes I)$). To shorten notations we will denote the enriched substitude by the same letters $(P,A)$ believing that this does not lead to confusion.

%We first investigate what can be said about convolution operation for  $\Set$ based substitudes. Let $(\mathcal{P},A)$ be such a substitude.
For a $\Set$-based substitude $(P,A)$ there is a  way to describe the convolution operation using the two sided Grothendieck  construction, which goes back to the Max Kelly's notion of a club \cite{club, species}.
 
 { Recall that {\it a bimodule}  from $B$ to $A$ is a functor }
\begin{equation} F: B^{op}\times A\rightarrow \Set .\end{equation} 
Every bimodule $F$ generates a two-sided categorical fibration 
\begin{equation}\label{twosided} B\stackrel{p}{\longleftarrow} \gint F\stackrel{\pi}{\longrightarrow} A \end{equation}
 as follows. The objects of $\gint F$ are triples $(b,a,f)$ where $(b,a) \in Ob(B^{op} \times A)$ and $f \in F(b,a)$. A morphism in $\gint F$ is a commutative square

\begin{align} \label{diagram:mor-Groth}
\xymatrix{
b \ar[r]^f \ar[d]_{\phi} & a \ar[d]^{\psi} \\
b' \ar[r]_{f'} & a'
}
\end{align}
in the sense that the assignments $F(b,a)\times A(a,a') \to F(b,a')$ and 
$B(b,b')\times F(b',a') \to F(b,a')$ produce the same element $g \in F(b,a')$, from $(f,\phi)$ and $(\psi,f')$. The category $\gint F$ is called the two-sided Grothendieck construction of $F$.

Let $\MM A$ be the free strict monoidal category on $A$ and $\Sm A$ be the free strict symmetric monoidal category on $A.$ From the universal property we have a canonical monoidal functor $\epsilon: \MM A\rightarrow \Sm A.$ 
% identity on objects (strings), but codomain has more morphisms from permutations.

The data for a $\Set$-substitude  $({P},A)$ (without substitution and unit) is the same  as a bimodule
\begin{equation}\label{petersen}  P:(\Sm A)^{op}\times A \rightarrow \Set. \end{equation} 
\begin{remark} Petersen in \cite{Pet} defines an $A\int \Sm$-module in $\VV$ as a sequence of  functors  $A^{op}\times A\int \Sm_n \to \VV, n\ge 0.$ It is not hard to see that Petersen's wreath-product $A\int \Sm_n$ is just a Grothendieck construction for a functor on the symmetric group  $\Sm_n$, considered as a one-object category, to $\Cat$ that sends a small category $A$ to $A^n.$ Hence, $A\int \Sm^n = (\Sm A)_n$ and a Petersen $A\int \mathbb{S}$-module in $\Set$ is exactly the data for a bimodule (\ref{petersen}) up to the permutation of variables and variance of functors. It is not hard to prove then that Petersen's notion of an $A$-coloured operad  coincides with the notion of substitude  based on $A^{op}.$ 

\end{remark}

By precomposing with $\epsilon\times 1$ we also have another bimodule
$$(\epsilon\times 1)^*P:(\MM A)^{op}\times A \rightarrow \Set.$$
The bimodule $P$ induces a two-sided categorical fibration 
\vspace{12pt}
\begin{equation}\label{bimodule} 
\Sm A \stackrel{p}{\longleftarrow} \gint P\stackrel{\pi}{\longrightarrow} A
\end{equation}
Explicitly, the category $\gint P$ has objects the {multimorphisms}
$e\in P(a_1,\ldots,a_n;a)$ which we denote $a_1\ldots a_n\stackrel{e}{\rightarrow} a.$ A morphism from 
 $a_1\ldots a_n\stackrel{e}{\rightarrow} a$ to $b_1\ldots b_n\stackrel{h}{\rightarrow} b$ is given by a permutation $\sigma\in \Sm_n,$ 
  an $n$-tuple $f_i: a_{\sigma(i)}\rightarrow b_{i}$, and $f:a\rightarrow b$ such that the square

\begin{align} \label{diagram:mor-Groth1}
\xymatrix{
a_{\sigma(1)}\ldots a_{\sigma(n)} \ar[r]^{\scriptstyle \hspace{5mm} \sigma(e)} \ar[d]_{f_1\ldots f_n} & a \ar[d]^{f} \\
 b_1\ldots b_n \ar[r]_{\hspace{3mm}\scriptstyle h} & b
}
\end{align}

\noindent commutes in  the same sense as in (\ref{diagram:mor-Groth}). 
The projection $p$  { sends a multimorphism to its domain  and $\pi$ sends it to the codomain. }

Similarly, the underlying bimodule $(\epsilon\times 1)^*P$ induces a two-sided fibration
\begin{equation}\label{ubimodule} 
\MM A \stackrel{up}{\longleftarrow} \gint (\epsilon\times 1)^*P \stackrel{u\pi}{\longrightarrow} A.\end{equation}

Now, let  $X:A\rightarrow \VV$ be  a presheaf. By the universal property of $\Sm$, $X$ induces a symmetric monoidal functor
$$\Sm(X): \Sm A\rightarrow \VV, \ \ \Sm(X)(a_1,\ldots,a_k) = X(a_1)\otimes\ldots\otimes X(a_k).$$ So, we have a functor:
$$\widetilde{(-)}: [A,\VV]\rightarrow SymMon(\Sm A,\VV) \rightarrow [\Sm A,\VV].$$

Analogously, we have a functor
$$\widetilde{(-)'}: [A,\VV]\rightarrow [\MM A,\VV].$$
This functor factors as the following composite 
$$[A,\VV]\stackrel{\Delta}{\longrightarrow}\prod_{k\ge 0}( [A,\VV]^k) \stackrel{\prod_k\tilde{\otimes}^k}{\longrightarrow} \prod_k  [A^k ,\VV] \simeq [\MM A,\VV],$$
where $\Delta$ for the factor $[A,\VV]^k$ is the diagonal functor, i.e., it is a restriction functor induced by the canonical functor $\coprod_{1}^k A \rightarrow A$ (the folding map) and 
$$\tilde{\otimes}^k(X_1,\ldots,X_k)(a_1,\ldots,a_n) = X(a_1)\otimes\ldots\otimes X(a_n).$$

 \begin{lem} \label{lemma:convolution}
 For a  substitude  $({P},A)$,  the $k$-th convolution tensor product $\otimes^k_{P}$ is isomorphic to the composite
 \begin{equation}\label{convolution} [A,\VV]^k \stackrel{\tilde{\otimes}^k}{\longrightarrow} [A^k,\VV]\stackrel{up^*}{\longrightarrow} [\gint (\epsilon\times 1)^*P,\VV]\stackrel{(u\pi)_!}{\longrightarrow}[A,\VV]. \end{equation}
  
 The functor part of the monad ${\Tt_P}$   on $[A,\VV]$ is isomorphic to the composite
 \begin{equation}\label{TE}
[A,\VV] \stackrel{\tilde{(-)}}{\longrightarrow} [\Sm A,\VV]\stackrel{p^*}{\longrightarrow} [\gint P,\VV]\stackrel{\pi_!}{\longrightarrow}[A,\VV].
 \end{equation} 
 \end{lem}
 \begin{proof}
 The first claim is a result in \cite[Section 7]{DS1}. The second result is immediate from the definition of ${ \Tt_P}$ and $\widetilde{(-)}$.

\end{proof}

\begin{defin} \label{defn:free action} A $\Set$-substitude $({P},A)$ is called a $\Sigma$-free substitude if
 there exists a bimodule 
\begin{equation}\label{E'} {d}(P):(\MM A)^{op} \times A \rightarrow \Set,\end{equation}
such that $P$ is the left Kan extension of $d(P)$ along 
   $$\epsilon\times 1:(\MM A)^{op} \times A \rightarrow (\Sm A)^{op}\times A .$$ 
   \end{defin}
  \begin{remark}\label{d=e} It is not true in general that $d(P) \cong (\epsilon\times 1)^*P$ for a $\Sigma$-free substitude $(P,A).$ This is true, however, if $(P,A)$ as a symmetric substitude is obtained as a symmetrisation of a nonsymmetric substitude (see \cite[Section 4]{DS2} for the definition of the symmetrisation). The proof of this fact left to the reader as an exercise.
  
  \end{remark} 
   
If $({P},A)$ is $\Sigma$-free then the bimodule $d(P)$ provides us with  yet another two-sided fibration:

\begin{equation}\label{bimodule'} 
\MM A \stackrel{p'}{\longleftarrow} \gint d(P)\stackrel{\pi'}{\longrightarrow} A.
\end{equation}

\begin{lem} \label{lemma:convolution2}
Let $({P},A)$ be a $\Sigma$-free $\Set$-substitude. Then the $k$-th convolution $\otimes^k_{P}$ is isomorphic to the symmetrisation of the following composite functor
 \begin{equation}\label{convolution2} \odot^k_{P}: [A,\VV]^k \stackrel{\tilde{\otimes}^k}{\longrightarrow} [A^k,\VV]\stackrel{(p')^*}{\longrightarrow} [\gint d(P),\VV]\stackrel{{\pi'_!}}{\longrightarrow}[A,\VV]. \end{equation} Furthermore, 
 \begin{equation}\label{TE2}{\Tt_P}(X)=\pi_!'((p')^*(\tilde{X'})) = \coprod_{k\ge 0} \odot^k_{P}(X,\ldots,X).\end{equation} 
\end{lem}

\begin{proof}
Using a classical formula for left Kan extensions, we compute the value of $(\epsilon\times 1)^*P(\bar{a},a) := (\epsilon\times 1)^*(\epsilon\times 1)_!d(P)(a_1,\ldots,a_k,a)$ as 
$$(\epsilon\times 1)^*P(\bar{a},a) = \colim_{(\epsilon\times 1) / (\bar{a},a)} \tilde{d(P)}.$$
It is not hard to see that the comma-category $(\epsilon\times 1) / (\bar{a},a)$ is just a coproduct of copies $(\MM A)^{op}\times A/(\bar{a},a)$ indexed by elements $\sigma\in \Sm_k.$

Hence, $\gint (\epsilon\times 1)^*P$ decomposes as a coproduct of $k!$ copies of $\gint d(P)$, and the result now follows from Lemma \ref{lemma:convolution} and the definition of symmetrisation \cite[Section 4]{DS2}. The second part of the lemma is immediate from the formula \ref{geometric series} applied to the symmetrisation of $ \odot^{(-)}_{P}.$ 

\end{proof}
\begin{defin} For a $\Sigma$-free substitude $(P,A)$, we  call the sequence of multifunctors  $\odot^k_P, k\ge 0$ { the nonsymmetric convolution} of $(P,A).$ \end{defin}
\begin{remark}\label{nonsym} If $(P,A)$ is itself obtained as symmetrisation of a nonsymmetric substitude $(Q,A)$, then the nonsymmetric convolution of $(P,A)$ coincides with the convolution of $(Q,A)$ and, in particular, this is a lax-monoidal (nonsymmetric) structure on $[A,\VV].$ But in general, the nonsymmetric convolution of a symmetric $\Sigma$-free substitude does not provide any lax-monoidal structure.  

\end{remark}

\section{Substitudes and internal algebra classifiers}  \label{sec:classifiers}

\subsection{$\Sigma$-free substitudes and polynomial monads}

Let  $({P},A)$ be a $\Sigma$-free $\Set$-substitude. Recall that $A_0$ means the maximal discrete subcategory of $A.$  Let also $i:A_0\to A$ be the inclusion. We then have 
a composite of forgetful functors $${\eta^*_0}:\Alg_P(\Set)\stackrel{\eta^*}{\to} [A,\Set]\stackrel{i^*}{\to} [A_0,\Set]$$ which is monadic
with left adjoint
$$(\eta_0)_!:  [A_0,\Set]\stackrel{i_!}{\to} [A,\Set]  \stackrel{\eta_!}{\to}    \Alg_P(\Set).$$
 Since $(P,A)$ is a $\Sigma$-free  substitude, the operad $(P_0,A_0)$ is a $\Sigma$-free $A_0$-coloured operad and the  monad  $\Pp = \Tt_{P_0}$ on $[A_0,\Set]$ generated by this adjunction is   a finitary polynomial monad by the results of Kock \cite{Kock} and Szawiel-Zawadowski \cite{Zav}. The reader who doesn't want to read these long and technical papers    is referred  to  \cite[Remark 6.4]{batanin-berger} for a brief discussion of the Kock-Szawiel-Zawadowski results.

 This is an important observation that allows us  to use the full force of the theory of internal algebra classifiers developed in \cite{SymBat,EHBat,batanin-berger, batanin de leger}. 

We briefly recall the main definitions and facts about finitary polynomial monads which we will need in the next section.  

Let $T_0$ be a  set.  A polynomial monad $\Tt = (\Tt,\mu,\epsilon)$ is a monad on the category $\Set/T_0 \cong [T_0,\Set]$  
whose functor part is generated by a polynomial 
\[ \xygraph{*{\xybox{\xygraph{!{0;(1.75,0):} {T_0}="p0" [r] {E}="p1" [r] {B}="p2" [r] {T_0}="p3" "p0":@{<-}"p1"^-{s}:"p2"^-{p}:"p3"^-{t}}}}}
 \]
and whose multiplication $\mu:\Tt^2\to \Tt$ and unit $\epsilon: \Tt_0=Id_{T_0}\to T$ are cartesian natural transformations (see below). Here, $Id_{T_0}$ is the identity monad on $[T_0,\Set]$ which we will denote it as $\Tt_0$ for convenience.  
The  functor $\Tt:\Set/T_0\to \Set/T_0$ generated by the polynomial  is the composite
$$[T_0,\Set]\stackrel{s^*}{\longrightarrow} [E,\Set] \stackrel{p_*}{\longrightarrow}[B,\Set] \stackrel{t_!}{\longrightarrow}[T_0,\Set] $$
where $s^*$ is the restriction along $s,$
 $p_*$ is the right Kan extension  along $p$,
 and $t_!$ is the left Kan extension along $t.$
Explicitly the functor $\Tt$ is given by the formula
\begin{equation*}\label{PPP}  \Tt(X)(i) = \coprod_{b\in t^{-1}(i)} \prod_{e\in p^{-1}(b)} X({s(e)}), \end{equation*}
which explains the name `polynomial' since it is  a sum of products of formal variables.
The set $T_0$ is called the set of colours (or objects) of $\Tt$ and the set $B$ is called the set of operations of  $\Tt.$  

A polynomial monad is finitary if all preimages of the map $p:E\to B$ are finite sets. We will be considering only finitary polynomial monads and so it will be convenient to call them simply polynomial monads. 

A 
{  cartesian morphism } $\phi:\Tt\to\Ss$   is a commutative diagram in $\Set$  
\[
			\xymatrix{
				T_0 \ar[d] & E %\pb
				 \ar[l]_-{s} \ar[r]^-{p} \ar[d] & B \ar[r]^-{t} \ar[d] & T_0 \ar[d] \\
				S_0& D \ar[l]_-{s'} \ar[r]^-{p'} & C \ar[r]^-{t'} & S_0
			}
			\]
			such that  the middle square is a pullback. 
			Finitary polynomial monads and their cartesian morphisms form a category equivalent to the category of $\Sigma$-free coloured symmetric operads in $\Set$ \cite{Kock,Zav}.

Any small category $A$ generates a polynomial monad, which we will denote $\Aa.$ The corresponding polynomial is given by   
\begin{equation}\label{cat as polymon}  \xygraph{*{\xybox{\xygraph{!{0;(1.75,0):} {A_0}="p0" [r] {A_1}="p1" [r] {A_1}="p2" [r] {A_0}="p3" "p0":@{<-}"p1"^-{s}:"p2"^-{id}:"p3"^-{t}}}}}
 \end{equation}
where $A_1$ is the set of morphisms of $A.$ The category of algebras of this monad is isomorphic to the category of covariant presheaves on $A.$

The data of a $\Sigma$-free polynomial $(P,A)$ amounts, therefore, to an identity-on-objects cartesian map of polynomial monads $\eta:\Aa\to\Pp,$ represented by a commutative diagram       
\[
			\xymatrix{
		 &         A_1 %\pb
				 \ar[dl]_-{s} 
				 \ar[r]^-{id} \ar[d] & A_1\ar[dr]^-{t} \ar[d] &  
				\\
				A_0& E \ar[l]_-{s} \ar[r]^-{p} & B \ar[r]^-{t} & A_0
			}
			\]
in which the central square is a pullback. Observe also that there is an isomorphism of sets of elements  $b\in B$ such that  $t(b) = a\in A_0, \ s(p^{-1}(b)) = \{a_1,\ldots,a_k\} \subset A_0$ and the cardinality $p^{-1}(b) = n$  and the set
$\coprod d(P)(a_{i_1},\ldots,a_{i_n};a),$ where $d(P)$ is a bimodule (\ref{E'}) and the coproduct is  taken over all possible different strings of elements of  $\{a_1,\ldots,a_k\}.$ See the Remark 6.4 in \cite{batanin-berger} for the evidence of this statement.

\subsection{Internal algebra classifiers}

Let $\VV$ be a cocomplete symmetric monoidal category and $\Tt$ be a finitary polynomial monad. We can then construct a functor $\Tt^{\VV}:[T_0,\VV] \to [T_0,\VV]$ as follows:
$$ \Tt^{\VV}(X)(a) = \coprod_{b\in t^{-1}(a)} \bigotimes_{e\in p^{-1}(b)} X_{s(e)}.$$
This defines a monad on $[T_0,\VV].$ 
\begin{defin} The category of algebras of a polynomial monad $\Tt$ 
in a cocomplete symmetric monoidal category $\VV$ is the category of algebras of the monad $\Tt^{\VV}.$
\end{defin} 
\begin{remark} To simplify the notation, the monad $\Tt^{\VV}$ will be denoted simply $\Tt.$  It is normally clear from the context in which category we consider our monad. The category of $\Tt$-algebras in $\VV$ will be denoted $\Alg_{\Tt}(\VV).$ 
\end{remark}

 Algebras of polynomial monads in the symmetric monoidal category of small categories $(\Cat,\times , 1)$ will be called categorical $\Tt$-algebras. The category of categorical algebras of $\Tt$ is isomorphic to the category of internal categories in the category of $\Tt$-algebras in $\Set$.   
 The category of categorical $\Tt$-algebras is naturally a $2$-category. We will use this fact but preserve the notation $\Alg_\Tt(\Cat)$ for this $2$-category. 
 
 \begin{agreement}\label{agreement}   An arbitrary categorical $\Tt$-algebra has as underlying object a $\Tt_0$-family of categories that is an object from $\Cat/\Tt_0.$ We will often refer to this object as a single category. 
 
 In the same spirit 
 a terminal internal category has a unique $\Tt$-algebra structure for any polynomial monad $\Tt$; the latter promotes it to a terminal categorical  $\Tt$-algebra. From now on {\it all terminal objects will be denoted $1$} hoping that this will cause no confusion.
 \end{agreement}

 The following definitions are taken from \cite{EHBat} and \cite{batanin-berger}.
 
 \begin{defin}\label{intalg1}Let $A$ be a categorical $\Tt$-algebra for a polynomial monad $\Tt$.

An {internal $\Tt$-algebra in $A$} is a {lax-morphism} of categorical $\Tt$-algebras from the {terminal} categorical $\Tt$-algebra to $A$.

Internal $\Tt$-algebras in $A$ and $\Tt$-natural transformations form a category $\Int_\Tt(A)$ and this construction extends to a $2$-functor $\Int_\Tt:\Alg_\Tt(\Cat)\to\Cat.$\end{defin}
 
%An internal $\Tt$-algebra in a categorical $\Tt$-algebra $A$ can be explicitly given by a collection of objects $a_i\in A_i$ together with a morphism$$\mu_{(b,\sigma)}:m_{(b,\sg)}(a_{s(\sigma(1))},\ldots, a_{s(\sigma(k))})\rightarrow a_{t(b)},$$ for each operation $(b,\sigma),$ which satisfies obvious associativity, unitarity and equivariance conditions. Here, $m_{(b,\sigma)}$ is the structure functor of $A.$

Given a cartesian map of polynomial monads $f: \Ss \to \Tt$ we have a restriction $2$-functor $f^*: \Alg_\Tt(\Cat)\to \Alg_\Ss(\Cat).$

 \begin{defin}\label{intalg2}Let $A$ be a categorical $\Tt$-algebra for a polynomial monad $\Tt$.

An internal $\Ss$-algebra in $A$ is a {lax-morphism} of categorical $\Ss$-algebras from the {terminal} categorical $\Ss$-algebra to $f^*(A)$.

Internal $\Ss$-algebras in $A$ and $\Ss$-natural transformations form a category $\Int_\Ss(A)$ and this construction extends to a $2$-functor 
$$ \Int_\Ss:\Alg_\Tt(\Cat)\to\Cat.$$ \end{defin}
This $2$-functor is representable and the categorical $\Tt$ algebra which represents it is called  the classifier of internal $\Ss$-algebras inside categorical $\Tt$-algebras. 
We denote it $\Tt^\Ss.$ In general, there is a precise recipe how to construct the classifier $\Tt^\Ss$ out of the data for a cartesian map of polynomial monads $f: \Ss \to \Tt$ (see \cite[Equation 12]{batanin-berger}) as a codescent object of a corresponding bar-construction.

\begin{example}\label{coma as classifier} Let $F:S\to T$ be a functor between small categories. One can consider it as a cartesian map between corresponding  polynomial monads $f:\Ss\to \Tt$ (see (\ref{cat as polymon})). In this case the classifier $\Tt^\Ss$ is the covariant presheaf on $T$ with values in $\Cat$ for which  $\Tt^\Ss(a) = F/a$, the classical comma-category of $F$ over $a.$ This can be easily seen from the above-mentioned \cite[Equation 12]{batanin-berger} but also from the fact that the classifier $2$-functor is a left $2$-adjoint to the classical Grothendieck construction $\int: [T,\Cat]\to \Cat/T$ because the category of internal $\Ss$-algebras in a categorical $T$-presheaf $A$ is isomorphic to the category of extensions of sections of the Grothendieck construction $\int A$ along $F.$ The reader may see \cite[Corollary 3.6]{batanin de leger} for a generalisation of this discussion to arbitrary polynomial monads.   

\end{example}

For any symmetric monoidal category $\VV$ there is a way to consider it as a constant categorical pseudoalgebra $\VV^\bullet_\Tt$ of any polynomial monad $\Tt$, and then consider internal $\Ss$-algebras inside $\VV^\bullet_\Tt$ cf. \cite[Section 6.8]{batanin-berger}.  Moreover, the category $\Int_\Ss(\VV^\bullet_\Tt)$  is isomorphic to  $\Alg_\Ss(\VV)$ \cite[Proposition 6.9]{batanin-berger}. The universal property of $\Tt^\Ss$ in this case tells us that the category of $\Ss$-algebras in $\VV$ is isomorphic to the  category of $\Tt$-functors $\Tt^\Ss \to \VV^\bullet_\Tt .$ Since $\VV^\bullet_\Tt$ is constant such a $\Tt$-functor is just a functor to $\VV$ in the usual sense satisfying some extra conditions. For this reason we identify $\VV$ and $\VV^\bullet_\Tt$ to shorten the notations. We hope it does not create any confusion.

Now let $X\in \Alg_{\Ss}(\VV)$ and $\tilde{X}:\Tt^\Ss\to \VV$ be the morphism of categorical $\Tt$-algebras representing $X.$ 
If $\VV$ is cocomplete symmetric monoidal category, the restriction functor  $f^*:\Alg_\Tt(\VV) \to \Alg_\Ss(\VV)$ has a left adjoint  $f_!:\Alg_\Ss(\VV) \to \Alg_\Tt(\VV).$ The following fact, which generalises the classical formula for pointwise left Kan extension, was proved in \cite[Theorem 6.17]{batanin-berger}:
\begin{theorem}
The underlying object of $f_!(X)$ is isomorphic to $\colim_{\Tt^\Ss}\tilde{X}.$ 
\end{theorem}

\subsection{Exact and semiexact squares of polynomial monads}

Recall \cite[Proposition 4.7]{batanin de leger} that any commutative square of cartesian morphisms of polynomial monads
 \begin{equation}\label{exact}\xymatrix{
 \Aa %\ar@<2.5pt>[r]^{\psi_!} 
 \ar@{<-}@<0pt>[d]_{\beta}
\ar@<0pt>[r]^{\phi}
& 
\Bb %\ar@<0pt>[d]_{\psi}
 \ar@{<-}@<0pt>[d]^{\alpha} \\
\Cc \ar@<2.5pt>[r]^{\psi}
\ar@<2.5pt>[ur]^{\gamma} %\ar@{<-}@<0pt>[d]^{\alpha} %\ar@<2.5pt>[u]^{\beta}  
& 
\Dd }
\end{equation} 
 induces a morphism of classifiers
\begin{equation}\label{induced map} \Dd^\Cc\to \alpha^*(\Bb^\Aa) .\end{equation}

\begin{defin}\label{semiexactdef} 
A commutative square (\ref{exact})
is  exact if the corresponding map of classifiers
(\ref{induced map})
is a final functor of underlying categories (see Agreement \ref{agreement}). 

This square is semiexact if there is a categorical $\Cc$-algebra $\tau$ and a $\Cc$-functor  $i:\tau\to \psi^*(\Dd^\Cc)$ such that the evaluation of the composite 
\begin{equation}\label{semiinduced map}\tau \to  \psi^*(\Dd^\Cc)\to \gamma^*(\Bb^\Aa) \end{equation} 
on each object of $\Cc$ is a final functor, and, moreover, the following square is a pullback
\begin{align} \label{exact fibration}
\xymatrix{
\Cc(\tau) \ar[r]^{} \ar[d]_{} & \Cc(\psi^*(\Dd^\Cc)) \ar[d]^{} \\
\tau \ar[r]_{} & \psi^*(\Dd^\Cc).}
\end{align}
In this square the vertical morphisms are algebra structure morphisms.
\end{defin} 
\begin{example} If polynomial monads in the square (\ref{exact}) correspond to small categories then the exactness of this square is equivalent to the exactness of the corresponding square of small categories in the sense of Guitart \cite{Guitart}. Our terminology was inspired by this classical paper. 

\end{example}

\begin{lem} Any exact square is semiexact. 
\end{lem}
\begin{proof} We take as $\tau$ the classifier $\psi^*(\Dd^\Cc).$ The pullback condition is trivial. 
\end{proof}

\begin{theorem}\label{bc section} Let $\VV$ be a cocomplete symmetric monoidal category. If the square (\ref{exact}) is semiexact then in  the induced square of adjunctions 
\begin{equation*}\label{BCB} \xygraph{!{0;(2.5,0):(0,.5)::}
{\Alg_{\Aa}(\VV)}="p0" [r] {\Alg_{\Bb}(\VV)}="p1" [d] {\Alg_{\Dd} (\VV)}="p2" [l] {\Alg_{\Cc}(\VV)}="p3"
"p0":@<-1ex>@{<-}"p1"_-{\phi^*}|-{}="cp":@<1ex>"p2"^-{\alpha^*}|-{}="ut":@<1ex>"p3"^-{\psi^*}|-{}="c":@<-1ex>@{<-}"p0"_-{\beta^*}|-{}="us"
"p0":@<1ex>"p1"^-{\phi_!}|-{}="dp":@<-1ex>@{<-}"p2"_-{\alpha_!}|-{}="ft":@<-1ex>@{<-}"p3"_-{\psi_!}|-{}="d":@<1ex>"p0"^-{\beta_!}|-{}="fs"
"dp":@{}"cp"|-{\perp} "d":@{}"c"|-{\perp} "fs":@{}"us"|-{\dashv} "ft":@{}"ut"|-{\dashv}}
\end{equation*}
 the  transformation
$$\psi^*(\mathbf{bc}):  \psi^* \psi_!\beta^* \to \gamma^*  \phi_! $$
has a section. 

If (\ref{exact}) is exact this square is a Beck-Chevalley square. 

\end{theorem}
\begin{proof}  Let us start from the second part of the Theorem. Let $X\in \Alg_{\Aa}(\VV)$ and $\tilde{X}:\Bb^\Aa\to \VV$ be the morphism of categorical $\Bb$-algebras representing $X.$ Then 
 the underlying object of $\phi_!(X)$ is isomorphic to $\colim_{\Bb^\Aa}\tilde{X}.$ Thus the underlying object of $\alpha^*(\phi_!(X))$ is isomorphic to $\alpha^*(\colim_{\Bb^\Aa}\tilde{X}).$ Notice that $\alpha^*$ in the last formula means, in fact, the composite of $\alpha^*$ and the forgetful functor to the category of collections. This is clearly the same as $\colim_{\alpha^*(\Bb^\Aa)}\alpha^*(\tilde{X}).$ Since (\ref{induced map}) is final such a colimit
is isomorphic to  $\colim_{\alpha^*(\Dd^\Cc)}\alpha^*(\tilde{X})$ which in turn is isomorphic to the underlying object of $\psi_!(\beta^*(X)).$ 
     
For the first statement of the Theorem it is enough to observe that the pullback condition of \ref{exact fibration} assures us that this is  an exact square in the sense of \cite{mark weber} because the right vertical arrow is the discrete fibration \cite[Proposition 4.3.4]{mark weber}. Hence, $i:\tau\to \psi^*(\Dd^\Cc)$ is an exact morphism of $\Cc$-algebras in the sense of Weber \cite[Definition 2.4.4]{mark weber} and, therefore, any morphism from $\tau$ to an algebraically cocomplete categorical $\Cc$-algebra admits an algebraic left extension along $i$  \cite[Corollary 2.4.5]{mark weber}. This last property means that we still can use the usual pointwise left Kan extension formula on the underlying collections to compute the left Kan extension in the category of $\Cc$-algebras. In particular, the colimit over restriction along $i$ is the underlying object of a $\Cc$-algebra and, moreover, the resulting morphism from this $\Cc$-algebra to $\psi^*\psi_!\beta^*(X)$ is a $\Cc$-algebra morphism.  Thus we can proceed with the proof like before with the difference that we can only claim that the Beck-Chevalley map is a retraction of $\Cc$-algebras.  

\end{proof}

\begin{remark}  It is important to notice that if we drop the pullback condition (\ref{exact fibration}) in Definition \ref{semiexactdef} we still can claim that the Beck-Chevalley map has a section but this section would be just a section on the level of underlying collections. But this is not enough for the purpose of this paper. We really need this section to be a $\Cc$-algebra morphism.  
\end{remark}

\section{Unary tame substitudes}  \label{sec:unary}

In this section we generalise the theory of tame  polynomial monads developed in \cite{batanin-berger}.  
From now on we will assume that all our substitudes are  $\Sigma$-free.  We will also assume that the unit of the substitude $\eta:A\to P_1$ is faithful (and it is the identity on objects by definition). 

 \begin{remark}
Many constructions below can be implemented without these two assumptions but this  requires much more complicated and heavier classifier techniques  developed by Mark Weber in \cite{mark weber}. We choose to work in a simplified situation because the formulations and proofs are more transparent with these assumptions and because the majority of our examples  are from this class.
 \end{remark}

\subsection{Fiberwise disconnected morphisms of polynomial monads} 

\begin{defin}  Let $\phi: \Ss\to \Tt$ be a cartesian morphism of polynomial monads. 
It will be called fiberwise disconnected  if the following square of cartesian morphisms is semiexact:

 \begin{equation}\label{exact disconnected}\xymatrix{
 \Ss %\ar@<2.5pt>[r]^{\psi_!} 
 \ar@{<-}@<0pt>[d]_{\epsilon_\Ss}
\ar@<0pt>[r]^{\phi}
& 
\Tt %\ar@<0pt>[d]_{\psi}
 \ar@{<-}@<0pt>[d]^{id} \\
\Ss_0 \ar@<2.5pt>[r]^{}
 %\ar@{<-}@<0pt>[d]^{\alpha} %\ar@<2.5pt>[u]^{\beta}  
& 
\Tt}
\end{equation}

\end{defin}

Before formulating the following lemma we recall that we continue the use of Agreement (\ref{agreement}) when we speak about classifiers as a single category.
\begin{lem} The following conditions are equivalent: \begin{enumerate}
\item $\phi$ is fiberwise disconnected; \item The classifier $\Tt^\Ss$ contains a discrete final subcategory; \item 
There is a reflective discrete subcategory in $\Tt^\Ss;$ 
\item The connected component  functor  $\Tt^\Ss \to \pi_0(\Tt^\Ss)$ has a right adjoint; \item
The category  $\Tt^\Ss$ is   a coproduct of categories with terminal objects. \end{enumerate} 
\end{lem}
\begin{proof} The equivalence of (2), (3), (4), (5) follows from 
\cite[Lemma 7.6]{batanin-berger}. 

To prove that (1) implies (2) observe that the map of classifiers $\Tt^{\Ss_0} \to \Tt^\Ss$ is the identity on objects but $\Tt^{\Ss_0}$  is discrete. Let $i:\tau\to \Tt^{\Ss^0}$ be the functor from the definition of semiexactness. The condition that $\tau$ is a $\Ss_0$-algebra just means that this is a $\Ss_0$-family of categories. We can then factor the $\Ss_0$-algebra morphism $i$ as $\tau\to\kappa \subset \Tt^{\Ss_0},$ where $\kappa$ is the  image of $i.$ Obviously $\kappa\subset \Tt^{\Ss}$ is final. The inverse implication is clear. \end{proof}

To shorten our wording in the rest of the paper we will often call a fiberwise disconnected morphism simply {\it disconnected}.

 \subsection{Tame polynomial monads}
\label{tame} Let us recall the definition of tameness for polynomial monads from \cite{batanin-berger}. 

Let $\Tt$ be a finitary monad on a cocomplete category $\CC$. Following \cite{batanin-berger} we denote by $\Tt+1$ the finitary monad on $\CC\times\CC$ given by \begin{align*}(\Tt+1)(X,Y)&=(\Tt X,Y)\\(\Tt+1)(\phi,\psi)&=(\Tt\phi,\psi)\end{align*} with evident multiplication and unit. Let $U_\Tt$ be the forgetful functor $\Alg_\Tt\to \CC$ and $F_\Tt$ be the free $\Tt$-algebra functor. The functor $R: \Alg_\Tt\to \Alg_\Tt\times\CC, \ R(X) = (X,U_\Tt(X))  $  has a left adjoint $L(X,K) = X\sqcup F_\Tt(K)  $ where the coproduct is taken in the category of $\Tt$-algebras. We have a commutative square of adjunctions:
\begin{equation}\label{adjoint1} \xygraph{!{0;(3.5,0):(0,.5)::}
{\Alg_\Tt \times \CC}="p0" [r] {\Alg_\Tt}="p1" [d] {\CC}="p2" [l] {\CC \times \CC}="p3"
"p0":@<-1ex>@{<-}"p1"_-{R}|-{}="cp":@<1ex>"p2"^-{U_\Tt}|-{}="ut":@<1ex>"p3"^-{\Delta_{\CC}}|-{}="c":@<-1ex>@{<-}"p0"_-{U_\Tt \times \textnormal{id}_{\CC}}|-{}="us"
"p0":@<1ex>"p1"^-{L}|-{}="dp":@<-1ex>@{<-}"p2"_-{F_\Tt}|-{}="ft":@<-1ex>@{<-}"p3"_-{- \sqcup -}|-{}="d":@<1ex>"p0"^-{F_\Tt \times \textnormal{id}_{\CC}}|-{}="fs"
"dp":@{}"cp"|-{\perp} "d":@{}"c"|-{\perp} "fs":@{}"us"|-{\dashv} "ft":@{}"ut"|-{\dashv}}
\end{equation}
 If $\CC$ has pullbacks that commute with coproducts, and $\Tt$ is a cartesian monad, then  $\Tt+1$ is a cartesian monad as well, and the adjoint square (\ref{adjoint1}) induces a cartesian morphism  $\Tt+1 \to \Tt.$

If $\Tt$ is a polynomial monad on $\CC=\Set/\Tt_0$ generated by the polynomial 
\[ \xygraph{*{\xybox{\xygraph{!{0;(1.75,0):} {\Tt_0}="p0" [r] {E}="p1" [r] {B}="p2" [r] {\Tt_0}="p3" "p0":@{<-}"p1"^-{s}:"p2"^-{p}:"p3"^-{t}}}}}
 \]
then $\Tt+1$ is a polynomial monad on $\Set/\Tt_0\sqcup \Tt_0 $ generated by the polynomial: 
\[ \xygraph{*{\xybox{\xygraph{!{0;(1.75,0):} {\Tt_0 \sqcup \Tt_0}="p0" [r] {E \sqcup \Tt_0}="p1" [r] {B \sqcup \Tt_0}="p2" [r] {\Tt_0 \sqcup \Tt_0}="p3"
"p0":@{<-}"p1"^-{s \sqcup id}:"p2"^-{p \sqcup id}:"p3"^-{t \sqcup id}}}}} \]
and the adjoint square (\ref{adjoint1}) for a polynomial monad $\Tt$  is induced by the following cartesian morphism  of polynomials 
\[ \xygraph{!{0;(2,0):(0,.6667)::} {\Tt_0 \sqcup \Tt_0}="p0" [r] {E \sqcup \Tt_0}="p1" [r] {B \sqcup \Tt_0}="p2" [r] {\Tt_0 \sqcup \Tt_0}="p3" [d] {\Tt_0}="p4" [l] {B}="p5" [l] {E}="p6" [l] {\Tt_0}="p7" "p0":@{<-}"p1"^-{s \sqcup id}:"p2"^-{p \sqcup id}:"p3"^-{t \sqcup id}:"p4"^-{\nabla_I}:@{<-}"p5"^-{t}:@{<-}"p6"^-{p}:"p7"^-{s}:@{<-}"p0"^-{\nabla} "p1":"p6"_-{\psi} "p2":"p5"^-{\phi} "p1":@{}"p5"|-{\textnormal{pb}}} \]
in which $\nabla$ is the folding map, and $\phi$ (resp. $\psi$) is the identity on $B$ (resp. $E$) and the unit  $\eta$ of $\Tt$ on $\Tt_0$.

\begin{remark}At this moment it will be convenient  to change a notation from \cite{batanin-berger}. So, we put:
$$\Tt + \Tt_0 := \Tt+1.$$   \end{remark}

\begin{defin}\cite{batanin-berger}\label{tame monad} A polynomial monad $\Tt$ is said to be \emph{tame} if the  morphism $\Tt+\Tt_0\to \Tt$ is  disconnected.\end{defin}

The $\Tt$-algebra  $\Tt^{\Tt +\Tt_0}$ classifies the so-called {\it semifree coproducts}, that is, the coproducts  $X\sqcup F_\Tt(K)$ in the category of $\Tt$-algebras  where $X$ is an arbitrary $\Tt$-algebra and $F_\Tt(K)$ a free $\Tt$-algebra on a collection $K\in [T_0,\Set]$. 
This implies that for a pair $(X,K)\in \Alg_\Tt\times \Alg_{\Tt_0}$ there is a canonical functor $\widetilde{(X,K)}:\Tt^{\Tt +\Tt_0} \to \Set$ such that 
$$U_\Tt(X\sqcup F_\Tt(K)) \cong \colim_{\Tt^{\Tt +\Tt_0}} \widetilde{(X,K)}.$$

The classifier $\Tt^{\Tt +\Tt_0}$
can be described in terms of  graphical language based on trees \cite{batanin-berger}. The set of operations $B\sqcup \Tt_0$ consists of corollas of two types
\begin{equation} \label{Xoperation} \xygraph{{\xybox{\xygraph{!{0;(.7,0):(0,1)::} {\scriptstyle{b}} *\xycircle<6pt>{-} (-[l(2)u] {\scriptstyle{X}},-[ul] {\scriptstyle{X}},-[u] {\scriptstyle{X}},-[ru] {\scriptstyle{X}},-[r(2)u] {\scriptstyle{X}},-[d] {\scriptstyle{X}})}}}
[r(2)] {\textnormal{and}} [r(1.5)]
{\xybox{\xygraph{!{0;(.7,0):(0,1)::} {\scriptstyle{1_i}} *\xycircle<6pt>{-} (-[u] {\scriptstyle{K}},-[d] {\scriptstyle{K}})}}}} \end{equation}
where $b\in B$ and the number of the inputs in the corolla is equal to $|p^{-1}(b)|.$ In the second corolla $1_i\in B$ for $i\in \Tt_0$ represents the units of $B=\Tt(1).$

Similarly, an object $\mathbf{b}$ of $\Tt^{\Tt +\Tt_0}$ is then represented by a corolla
\begin{equation}\label{XKoperation} \xygraph{!{0;(.7,0):(0,1)::} {\scriptstyle{b}} *\xycircle<6pt>{-} (-[l(2)u] {\scriptstyle{K}},-[ul] {\scriptstyle{X}},-[u] {\scriptstyle{X}},-[ru] {\scriptstyle{K}},-[r(2)u] {\scriptstyle{K}},-[d])} \end{equation}
with incoming edges coloured by $X$ and $K$.  The $X$-edges correspond to the operations on the $\Tt$-algebra summand of the semi-free coproduct, while the $K$-edges correspond to the free summand. A morphism $\mathbf{b}'\to\mathbf{b}$ in $\Tt^{\Tt +\Tt_0}$ is given by a set of elements $b_1,\ldots,b_k\in B$, one for each $X$-coloured edge of $\mathbf{b}$, such that $b'= b(1,\ldots,b_1,1,\ldots, b_k,\ldots,1)$, where the $1'$s correspond to $K$-edges and the $b_i'$s correspond to $X$-edges of $\mathbf{b}$. Given such a map, the corolla representing $\mathbf{b}'$ is obtained from the corolla representing $\mathbf{b}$ by replacing the $i$-th $X$-edge of $\mathbf{b}$ with as many $X$-edges as the corresponding operation $b_i\in B$ has elements in its fiber $p^{-1}(b_i)$.

\subsection{Unary tame substitudes} 
 Let $(P,A)$ be a  substitude in $\Set$.  
We can now describe a new polynomial monad which we denote $\Pp + \Aa.$  
As before, let $\Pp$ be represented by the polynomial 
\begin{equation}\label{P+A} \xygraph{*{\xybox{\xygraph{!{0;(1.75,0):} {A_0}="p0" [r] {E}="p1" [r] {B}="p2" [r] {A_0.}="p3" "p0":@{<-}"p1"^-{s}:"p2"^-{p}:"p3"^-{t}}}}}
 \end{equation}
  and let $\Aa$ be the polynomial monad which corresponds to the small category $A.$ The unit of the substitude $(P,A)$ generates  the canonical morphism of polynomial monads $\eta:\Aa \to \Pp.$

 Then $\Pp + \Aa$ is the monad given by the polynomial 
\[ \xygraph{*{\xybox{\xygraph{!{0;(1.75,0):} {A_0 \sqcup A_0}="p0" [r] {E \sqcup A_1}="p1" [r] {B \sqcup A_1}="p2" [r] {A_0 \sqcup A_0}="p3"
"p0":@{<-}"p1"^-{s \sqcup s_A}:"p2"^-{p \sqcup id}:"p3"^-{t \sqcup t_A}}}}} \]
where $A_1$ is the set of morphisms of the category $A$ and $s_A$ and $t_A$ are source and target maps in $A.$

As previously, there is a cartesian morphism  of polynomial monads 
\[ \xygraph{!{0;(2,0):(0,.6667)::} {A_0 \sqcup A_0}="p0" [r] {E \sqcup A_1}="p1" [r] {B \sqcup A_1}="p2" [r] {A_0 \sqcup A_0}="p3" [d] {A_0}="p4" [l] {B}="p5" [l] {E}="p6" [l] {A_0}="p7" "p0":@{<-}"p1"^-{s \sqcup s_A}:"p2"^-{p \sqcup id}:"p3"^-{t \sqcup t_A}:"p4"^-{\nabla'}:@{<-}"p5"^-{t}:@{<-}"p6"^-{p}:"p7"^-{s}:@{<-}"p0"^-{\nabla'} "p1":"p6"_-{\psi'} "p2":"p5"^-{\phi'} "p1":@{}"p5"|-{\textnormal{pb}}} \]
in which $\nabla'$ is the identity on each copy of $A_0$, and $\phi$ (resp. $\psi$) is the identity on $B$ (resp. $E$) and on $A_1$ is induced by the unit of substitude.

%Recall that there is an isomorphism of categories of algebras $\Alg_P \cong \Alg_{\mathcal{P}_0}.$ 
We now have the following square of adjunctions generated by the cartesian map of monads above: 
\begin{equation}\label{adjoint} \xygraph{!{0;(3.5,0):(0,.5)::}
{\Alg_P \times [A,\Set]}="p0" [r] {\Alg_P}="p1" [d] {[A_0,\Set]}="p2" [l] {[A_0,\Set] \times [A_0,\Set]}="p3"
"p0":@<-1ex>@{<-}"p1"_-{R_P}|-{}="cp":@<1ex>"p2"^-{\eta_0^*}|-{}="ut":@<1ex>"p3"^-{\Delta_{}}|-{}="c":@<-1ex>@{<-}"p0"_-{\eta_0^* \times i^*}|-{}="us"
"p0":@<1ex>"p1"^-{L_P}|-{}="dp":@<-1ex>@{<-}"p2"_-{(\eta_0)_!}|-{}="ft":@<-1ex>@{<-}"p3"_-{- \sqcup -}|-{}="d":@<1ex>"p0"^-{(\eta_0)_! \times i_!}|-{}="fs"
"dp":@{}"cp"|-{\perp} "d":@{}"c"|-{\perp} "fs":@{}"us"|-{\dashv} "ft":@{}"ut"|-{\dashv}}
\end{equation}
Here, the functor $R_P: \Alg_P\to \Alg_P\times [A,\Set]$ is $R(X) = (X,\eta^*(X))  $ and its  left adjoint is $L(X,K) = X\sqcup \eta_!(K) ,$ where $\eta^*$ is the forgetful functor from $\Alg_P$ to $[A,\Set]$ and $\eta_!$ its left adjoint as usual.

We now can form a new $P$-algebra   $\Pp^{\Pp+\Aa}$ which classifies the coproducts of $P$-algebras of the form  $X\sqcup \eta_!(K),$ whereas  $\Pp^{\Pp+\Aa_0}$ is the classifier for semifree coproducts of $P$-algebras $X\sqcup (\eta_0)_!(K)$ as discussed after Definition \ref{tame monad}. 
 
The morphism $\Pp+\Aa_0\to \Pp$ factors as $\Pp+\Aa_0\to \Pp+\Aa\to \Pp$ and we have an induced 
map of classifiers
$$\Pp^{\Pp+\Aa_0}\to \Pp^{\Pp+\Aa}.$$  
The classifier $\Pp^{\Pp+\Aa}$ has the same objects as $\Pp^{\Pp+\Aa_0}$ but more morphisms. Like in 
$\Pp^{\Pp+\Aa_0}$, a morphism $\mathbf{b}'\to\mathbf{b}$ is given by a set of elements $b_1,\ldots,b_k\in B$  one for each $X$-coloured edge of $\mathbf{b}$, but also a list of elements  $\alpha_1,\ldots,\alpha_l \in A_1$ for each $K$-coloured edge such that $b'= b(\alpha_1,\ldots,b_1,\alpha_i,\ldots, b_k,\ldots,\alpha_l)$, where the $\alpha_i$ correspond to $K$-edges and the $b_i'$s correspond to $X$-edges of $\mathbf{b}$. The corolla representing $\mathbf{b}'$ is obtained from the corolla representing $\mathbf{b}$ by replacing the $i$-th $X$-edge of $\mathbf{b}$ with as many $X$-edges as the corresponding operation $b_i\in B$ has elements in its fiber $p^{-1}(b_i)$ as in $\Pp^{\mathcal{P}_0+\Aa_0}.$
 
Replacing the monad $\Pp$ by the monad $\Aa$ in the construction (\ref{P+A})  we obtain a new monad $\Aa+\Aa$ over $\Pp.$ 
 We  have a commutative triangle of cartesian maps of polynomial monads: 
 \begin{equation}\label{triangle}
	\xymatrix{
		\Aa+\Aa\ar[rr] \ar[rd] && \Pp+\Aa \ar[ld] \\
		& \Pp}
	\end{equation}
 
 This triangle induces a morphism of  $\Pp^{\Aa+\Aa}$  to $\Pp^{\Pp+\Aa}$ which is the identity on objects and injective on morphisms. So, we can consider $\Pp^{\Aa+\Aa}$ as a subcategory of  $\Pp^{\Pp+\Aa}.$  We will call a morphism $\mathbf{b}'\to\mathbf{b}$ in $\Pp^{\Pp+\Aa}$   {\it unary}   if it belongs to $\Pp^{\Aa+\Aa}.$ Explicitly, a morphism $b'= b(\alpha_1,\ldots,b_1,\alpha_i,\ldots, b_k,\ldots,\alpha_l)$ is unary if      all elements  $b_1,\ldots,b_k\in B$ are  unary operations in $P$ which come from  morphisms of $A.$ 
 
 \begin{defin} A substitude $(P,A)$  is unary tame if the square 
  \begin{equation}\label{exact unary}\xymatrix{
 \Pp+\Aa %\ar@<2.5pt>[r]^{\psi_!} 
 \ar@{<-}@<0pt>[d]_{\beta}
\ar@<0pt>[r]^{\phi}
& 
\Pp %\ar@<0pt>[d]_{\psi}
 \ar@{<-}@<0pt>[d]^{id} \\
\Aa + \Aa \ar@<2.5pt>[r]^{\psi}
 %\ar@{<-}@<0pt>[d]^{\alpha} %\ar@<2.5pt>[u]^{\beta}  
& 
\Pp}
\end{equation} 
 is semiexact.
\end{defin} 

The following Lemma provides an elementary criteria for unary tameness. It will be very useful in practice when we discuss our applications later. 

 \begin{lem}\label{a+a} A substitude $(P,A)$ is unary tame if and only if there exists a morphism of  $A$-presheaves in $\Cat$   $$\lambda\to \eta^*( \Pp^{\Aa+\Aa})$$ such that the composite $\lambda\to \eta^*( \Pp^{\Aa+\Aa}) \to \eta^*( \Pp^{\Pp+\Aa})$ is a levelwise final functor and the square
 \begin{align} \label{exact fibration u}
\xymatrix{
\Aa(\lambda) \ar[r]^{} \ar[d]_{} & \Aa(\eta^*(\Pp^{\Aa+\Aa})) \ar[d]^{} \\
\lambda \ar[r]_{} & \eta^*(\Pp^{\Aa+\Aa})}
\end{align} 
  is a pullback.
 \end{lem}
\begin{proof} If such a $\lambda$ exists we can take $\tau = (\lambda,\lambda)$ because $$\psi^*(\Pp^{\Aa+\Aa}) = (\eta^*(\Pp^{\Aa+\Aa}),\eta^*(\Pp^{\Aa+\Aa})).$$  
In the other direction, if $\tau \to \psi^*(\Pp^{\Aa+\Aa})$ satisfies the condition of semiexactness 
$\tau$ is, in fact, a pair of categorical  $A$-presheaves $(\lambda,\lambda')$ together with a morphism to $(\eta^*(\Pp^{\Aa+\Aa}),\eta^*(\Pp^{\Aa+\Aa})).$
 It is not hard to see that either of $\lambda$ and $\lambda'$  satisfies the condition of lemma.
    \end{proof}
\begin{remark}
The role of the pullback condition \ref{exact fibration u} is somewhat subtle  in this Lemma. It ensures the pullback condition \ref{exact fibration} from the definition of semiexactness which, in turn, ensures that the section  to the Beck-Chevalley map induced by this square as in Theorem \ref{bc section} is a map of $A$-presheaves and is not just a levelwise section. It will be used later in the proof of Proposition \ref{retract}.    
\end{remark}

\begin{example} If $\Pp$ is tame as polynomial monad then $(P,A_0)$ is a unary tame substitude. In this case the final subcategory in question is just the same discrete final subcategory in $\Pp^{\Pp +\Aa_0}$ from the definition of tameness.  \end{example}

 \subsection{Semifree coproduct and convolution}
 
 We start from the following elementary observation. Consider the classifier $\Pp^\Aa.$ If $X$ is an $\Aa$-algebra (in a cocomplete symmetric monoidal category $\VV$), that is, a functor $A\to \VV$, then we can construct a functor $\tilde{X}:\Pp^\Aa \to \VV.$ The colimit of this functor computes the free $P$-algebra on $X.$ Combining it with Lemma \ref{lemma:convolution2}  we get  a formula
 $$\colim_{\Pp^\Aa} \tilde{X} \cong \coprod_n \odot^n_{P}(X,\ldots,X).$$  
We also can prove this formula by direct calculations knowing that the objects of $\Pp^\Aa$ correspond to the operations in $P$ (that is, corollas as on the left in the picture (\ref{Xoperation}))  and the functor $\tilde{X}$ sends such an object to the corresponding value $X(s_1(b))\otimes\ldots\otimes X(s_k(b)).$ Notice that we either choose a planar structure on the corollas to be able to fix the order in the tensor product above or, equivalently, take the colimit over the action of groupoid $\Sm_k$ on this object (that is, use the unordered tensor product). The resulting functors are canonically isomorphic. But in the first case  we may think that the objects of $\Pp^\Aa$ are planar corollas and the order in the tensor product is taken according to their planar structures. Then we take the colimit over unary operations in $A,$ which is exactly the definition of nonsymmetric convolution.  
 Let now $\psi:\Aa+\Aa \to \Pp$ be the map of polynomial monad presented in the diagram (\ref{triangle}). Similarly to the above we have the following 
\begin{lem}\label{shufflecoproduct} Let $(X,K)\in [A,\VV]\otimes[A,\VV].$ Then the underlying $A$-presheaf  of the coproduct $\psi_!(X)\sqcup\psi_!(K)$ has the following form: 
$$\colim_{\Pp^{\Aa+\Aa}} \widetilde{(X,K)} \cong \coprod_{p,q\ge 0, \sigma\in Sh_{p,q}} \sigma(\odot^{p+q}_{P})(\underbrace{X,\ldots,X}_p,\underbrace{K,\ldots,K}_q)$$  
Here, $Sh_{p,q}$ is the set of $(p,q)$-shuffles and $$\sigma(\odot^{n}_{P})({X_1,\ldots,X_n}) = \odot^{n}_{P}(X_{\sigma(1)},\ldots,X_{\sigma(n)}).$$

\end{lem}
\begin{proof} The only thing which requires explanation in this formula is the presence of the set of shuffles. It follows from the fact that we can take as  objects of $ 
 \Pp^{\Aa+\Aa}$ two coloured planar corollas as shown in (\ref{XKoperation}). Such  coloured planar corollas are in one to one correspondence with $(p,q)$-shuffles. 
 \end{proof}
 
This lemma implies  the following relationships between  coproduct $X\sqcup \eta_!(K)$ and convolution for a unary tame substitude.

\begin{pro}\label{retract} Let  $(P,A)$ be a unary tame substitude and let     $X\in \Alg_P(\VV)$ and $K\in [A,\VV].$  The underlying $A$-presheaf of the coproduct $X\sqcup \eta_!(K)$ is naturally a retract of 
$$\coprod_{p,q\ge 0, \sigma\in Sh_{p,q}} \sigma(\odot^{p+q}_{P})(\underbrace{\eta^*(X),\ldots,\eta^*(X)}_p,\underbrace{K,\ldots,K}_q).$$  
 
\end{pro}
\begin{proof} Consider the adjunction square generated by the semiexact square (\ref{exact unary}):
\begin{equation*}\label{BCBC} \xygraph{!{0;(2.5,0):(0,.5)::}
{\Alg_{\Pp+\Aa}(\VV)}="p0" [r] {\Alg_{\Pp}(\VV)}="p1" [d] {\Alg_{\Pp} (\VV)}="p2" [l] {\Alg_{\Aa+\Aa}(\VV)}="p3"
"p0":@<-1ex>@{<-}"p1"_-{\phi^*}|-{}="cp":@<1ex>"p2"^-{\alpha^*}|-{}="ut":@<1ex>"p3"^-{\psi^*}|-{}="c":@<-1ex>@{<-}"p0"_-{\beta^*}|-{}="us"
"p0":@<1ex>"p1"^-{\phi_!}|-{}="dp":@<-1ex>@{<-}"p2"_-{\alpha_!}|-{}="ft":@<-1ex>@{<-}"p3"_-{\psi_!}|-{}="d":@<1ex>"p0"^-{\beta_!}|-{}="fs"
"dp":@{}"cp"|-{\perp} "d":@{}"c"|-{\perp} "fs":@{}"us"|-{\dashv} "ft":@{}"ut"|-{\dashv}}
\end{equation*}
Recall that in this square $\alpha_!$ and $\alpha^*$ are the identities and $\phi_! = L_P, \phi^* = R_P.$

According to Theorem (\ref{bc section}), for each $(X,K)\in \Alg_{\Pp + \Aa}$ we have a retraction of $\Aa+\Aa$-algebras:
$$ \gamma^*  \phi_!(X,K) \to \psi^* \psi_!\beta^*(X,K) \to \gamma^*  \phi_!(X,K) .$$
The first component of the object $\gamma^*  \phi_!(X,K)$ is exactly the underlying presheaf  of the coproduct $X\sqcup \eta_!(K),$ whereas 
the first component of $\psi^* \psi_!\beta^*(X,K)$ is the underlying presheaf of the coproduct $\psi_!\eta^*(X)\sqcup \psi_!(K).$
An application of Lemma  \ref{shufflecoproduct} completes the proof.

\end{proof}
\begin{example} To illustrate Proposition \ref{retract} let's see what it tells us in case of $Ass.$ The nonsymmetric convolution of $Ass$ is just the tensor product in $\VV.$ According to \cite[Section 9.2]{batanin-berger} the underlying object of the coproduct of monoids  $X$ and a free monoid on $K$ is given by $$
\coprod_{k\ge 0} U(X)\otimes (K\otimes U(X))^{\otimes^k}\subset 
\coprod_{p,q\ge 0, \sigma\in Sh_{p,q}} \sigma(\odot^{p+q}_{P})(\underbrace{U(X),\ldots,U(X)}_p,\underbrace{K,\ldots,K}_q)
$$
The retraction on a summand $\sigma(\odot^{p+q}_{P})(\underbrace{U(X),\ldots,U(X)}_p,\underbrace{K,\ldots,K}_q)$ multiplies any consecutive string of $U(X)$s between two $K$s using the multiplication in $X$ and introduces a copy of $U(X)$ between any two consecutive $K$s using the unit in $X.$ 
For instance:
$$K\otimes K\otimes U(X)\otimes U(X)\otimes U(X)\otimes K \cong I\otimes K\otimes I\otimes K\otimes (U(X)\otimes U(X)\otimes U(X))\otimes K \otimes I \to $$
$$\to U(X)\otimes K\otimes U(X)\otimes K\otimes U(X)\otimes K \otimes U(X).$$  

\end{example}

\section{Transfer theorem} \label{sec:transfer}

\subsection{Convolution and Transfer of semimodel structures } For the reader's convenience we recall standard material about Quillen functors of many variables.
Let $C_1,\ldots,C_k, C$ be a list of cocomplete  categories.  
 Let  $T:C_1\times \ldots\times C_k \rightarrow C$ be a functor. Let $f_1:K_1\rightarrow L_1,\ldots, f_k:K_k\rightarrow L_k$ be   morphisms  in $C_1,\ldots, C_k$ correspondingly.  The functor $T$ generates a commutative $k$-cube  of morphisms in $C$ as follows. The $k$-cube ${\mathtt I}^{\mathtt k}$ as a poset has vertices  the subsets of the set  $\{1,\ldots,k\}$ ordered by inclusion. Let $S\subset \{1,\ldots,k\}$ be such a vertex. We put 
 $$T(S) = T(X_1,\ldots, X_k) ,$$
 where $X_i = K_i$ if $i\notin S$  and $X_i=L_i$ if $i\in S.$ For an inclusion $j:P\subset S$ we associate the morphism 
 \begin{equation}\label{pr}  T(j) = T(h_1,\ldots,h_k)\end{equation}
 where $h_i = f_i$ if $i\notin P$  but   $i\in S$ and $h_i = id$ otherwise. 
 
 Let also ${\mathtt I_{\circ}^{\mathtt k}}$ be a punctured $k$-cube  ${\mathtt I}^{\mathtt k}\setminus \{1,\ldots,k\}.$ 
 
  \begin{defin} Let  $T:C_1\times \ldots\times C_k \rightarrow C$ be a functor. The corner map  $f_1\boxx f_2\boxx\ldots\boxx f_k$ is the morphism  
 $$colim_{{\mathtt I_{\circ}^{\mathtt k}}}T(S) \rightarrow  T(L_1,\ldots, L_k).$$ 
 \end{defin}
 
 \begin{remark} \label{ar} \rm Let   $(C,\otimes,I)$ be   a monoidal category such that $\otimes$ preserves colimits in both variables. Let $T(X,Y) = X\otimes Y$. The functor 
 $-\boxx- : Arr(C)\times Arr(C)\rightarrow Arr(C)$ is the tensor part of a monoidal structure on the category $Arr(C)$ of arrows of $C$ \cite{Muro}. 
 It is  Day's convolution product generated by the monoidal structure on the poset $0<1$, whose unit object is $1$ and whose tensor product is $max(-,-).$  More generally, if the sequence of functors $T_k:C^k\rightarrow C, k\ge 0$ is a cocomplete lax-monoidal structure on $C,$ then the sequence of functors $\boxx^k, k\ge 0$ provides a cocomplete lax-monoidal structure on $Arr(C)$ \cite{DS2}. \end{remark}

Now assume that the categories $C_i \ , \ 0\le i\le k $ and $C$ are  semimodel categories.%left homotopically structured (a discussion of the homotopical properties of $C^i$ may be found in \cite{white-yau5}).
 
 \begin{defin}  \label{defn:left Quillen multifunctor}
 We call  $T:C_1\times \ldots\times C_k \rightarrow C$ {\em left Quillen} if it preserves colimits in each variable and $f_1\boxx f_2\boxx\ldots\boxx f_k$ is a cofibration  whenever all $f_i$ are cofibrations with  cofibrant domains and, moreover,  the morphism $f_1\boxx f_2\boxx\ldots\boxx f_k$ is  trivial cofibration  provided one of $f_i$ is a trivial cofibration. 
 \end{defin}

 \begin{remark}\rm \label{k=0}  For $k=1$ this definition states simply that $T$ preserves cofibrations and trivial cofibrations with cofibrant domains. For $k=0$  we have an empty product of categories so the functor $T$ is just a functor $T:1\rightarrow C$, i.e., an object $T(1)$ of $C$. Here $1$ is the terminal category.   The $0$-cube is a one-point poset and the punctured cube is empty. Therefore,  $T:1\rightarrow C$ is left Quillen if and only if $T(1)$ is cofibrant.\end{remark}

\begin{defin} We will call a $\Sigma$-free substitude $(P,A)$ a left Quillen substitude  with respect to a  semimodel model structure on $[A,\VV]$ if its nonsymmetric convolution forms a sequence of left Quillen functors on $[A,\VV].$

\end{defin}

\begin{example}
The substitude for the operad $\E=Ass$ is left Quillen  with respect to a semimodel structure on $\VV=[1,\VV]$ if and only if $\VV$ is a monoidal semimodel category. 

\end{example}

With this definition in hand, we are ready to state the main theorem of this section, which is a crucial building block for our applications in Part \ref{part:n-operads}. We note that the condition of the following theorem is satisfied for all the substitudes considered in Part \ref{part:n-operads}.

\begin{theorem}\label{semitransfer} Let $(P,A)$ be a $\Sigma$-free unary tame substitude whose unit is faithful and let $(\VV,\otimes,I)$ be a cocomplete symmetric monoidal category. Assume that the category $[A,\VV]$ is equipped with 
a cofibrantly  generated semimodel structure.
 Then  $\Alg_P(\VV)$ admits a semimodel structure  transferred from $[A,\VV]$ provided $(P,A)$ is left Quillen with respect to $[A,\VV].$
 
  Moreover, the forgetful functor  $U:\Alg_P \to [A,\VV]$ preserves cofibrant objects as well as  cofibrations between cofibrant objects.
\end{theorem}

Notice that  according to the Remarks  \ref{0convolution} and    \ref{k=0} the underlying object of the initial $P$-algebra is cofibrant in $[A,\VV].$ This enables us to use the transfer techniques described in Theorem \ref{prop:semi-helper}. We follow the idea of the proof of the Transfer Theorem for algebras over tame polynomial monads from \cite{batanin-berger}, but with some modifications to capture the action of the unary operations of the substitude. 
 
These technical preparations are done in the Subsections \ref{freealgext} and \ref{canfilt}. In these subsections the only assumption we use is the $\Sigma$-freeness of the substitude.  The proof of Theorem \ref{semitransfer} will be completed in Subsection \ref{semitransferend}, where we actually need unary tameness. 
 
\subsection{Classifier for free algebra extensions}\label{freealgext}  Let $(P,A)$ be a substitude in $\Set.$  Let ${\mathbb P}_{f,g}$ be the category whose objects are quintuples
$(X,K,L,g,f),$  where $X$ is a $P$-algebra, $K,L$  are objects in $[A,\Set]$ and $g:K\rightarrow \eta^*(X),\,f:K\rightarrow L$ are morphisms in $[A,\Set].$ There is an obvious forgetful functor
$$\Uu_{f,g}:{\mathbb P}_{f,g} \rightarrow [A_0,\Set]\times [A_0,\Set]\times [A_0,\Set],$$
taking the quintuple $(X,K,L,f,g)$ in $\mathbb{P}_{f,g}$ to the triple $\eta^*(X),i^*(K),i^*(L)),$ 
where $i:A_0\to A$ is the inclusion of the maximal discrete subcategory of $A,$  and $\eta_0$ is the composite of the unit $\eta$ and $i$ as usual.  

\begin{pro}\label{Cart}

\begin{itemize} \item[(i)] The functor $U_{f,g}$ is monadic and the induced monad $\Pp_{f,g}$ is polynomial;
\item[(ii)] There is a commutative square of adjunctions

\begin{equation*}\label{adjointext} \xygraph{!{0;(3.5,0):(0,.5)::}
{{\mathbb P}_{f,g}}="p0" [r] {\Alg_P}="p1" [d] {[A_0,\Set]}="p2" [l] {[A_0,\Set] \times [A_0,\Set]\times [A_0,\Set]}="p3"
"p0":@<-1ex>@{<-}"p1"_-{R_P}|-{}="cp":@<1ex>"p2"^-{\eta_0^*}|-{}="ut":@<1ex>"p3"^-{\Delta_{}}|-{}="c":@<-1ex>@{<-}"p0"_-{U_{f,g}}|-{}="us"
"p0":@<1ex>"p1"^-{L_P}|-{}="dp":@<-1ex>@{<-}"p2"_-{(\eta_0)_!}|-{}="ft":@<-1ex>@{<-}"p3"_-{- \sqcup -}|-{}="d":@<1ex>"p0"^-{F_{f,g}}|-{}="fs"
"dp":@{}"cp"|-{\perp} "d":@{}"c"|-{\perp} "fs":@{}"us"|-{\dashv} "ft":@{}"ut"|-{\dashv}}
\end{equation*}

\noindent in which  $\Delta$  is the diagonal and $R_P$ is given by $$R_P(Y)= (Y,\eta^*(Y),\eta^*(Y),id_{\eta^*(Y)},id_{\eta^*(Y)}).$$

\item[(iii)]The left adjoint $L_P$ to $R_P$ is given by the following pushout in $\Alg_P$:

\begin{align*} \label{free cofibration}
\xymatrix{
\eta_!(K) \ar[r]^{\eta_!(f)} \ar[d]_{\hat{g}} & \eta_!(L) \ar[d]^{} \\
X \ar[r]^{} & L_P(X,K,L,g,f)
}
\end{align*}
\noindent in which $\hat{g}$ is the mate of $g$.
\end{itemize}

\end{pro}

\begin{proof} The proof is completely analogous to the proof of Proposition 7.2 in \cite{batanin-berger}

\end{proof} 

We will need an explicit description of the classifier $\Pp^{\Pp_{f,g}}.$ It coincides almost verbatim to the description of $T^{T_{f,g}}$  given in \cite[Section 7.4]{batanin-berger}. So, we recall it briefly.

The objects of $\Pp^{\Pp_{f,g}}$ are corollas decorated by the elements of $B=\Pp(1)$ with its $|p^{-1}(b)|$ incoming edges coloured by one of the three colours $X, K, L$:
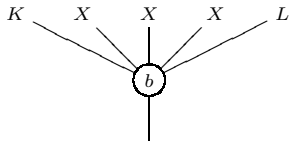
\begin{equation}\label{tricolor} \xygraph{!{0;(.7,0):(0,1)::} {\scriptstyle{b}} *\xycircle<6pt>{-} (-[l(2)u] {\scriptstyle{K}},-[ul] {\scriptstyle{X}},-[u] {\scriptstyle{X}},-[ru] {\scriptstyle{X}},-[r(2)u] {\scriptstyle{L}},-[d(.9)])} \end{equation}
These incoming edges will be called  $X$-edges, $K$-edges or $L$-edges accordingly.

Morphisms of $\Pp^{\Pp_{f,g}}$ can be described in terms of generators and relations. There are three types of generators. First, we have the generators coming from the $P$-algebra structure on $X$-coloured edges and unary operation on $K$ and $L$ edges corresponding to morphisms of $A.$ The relations between these generators witness the relations between operations in $(P,A)$. We will call these generators \emph{$X$-generators}.

The next type of generators corresponds to the morphism $f:K\rightarrow L.$ Such a generator simply replaces a $K$-edge with an $L$-edge in the corolla. Generators of this kind will be called \emph{$F$-generators}. 

Finally, we have generators corresponding to $g:K\rightarrow \eta^*(X).$ Such a generator replaces a $K$-edge with an $X$-edge. Generators of this kind will be called \emph{$G$-generators}. 

An important property for us is that   every span $b\stackrel{\phi}{\leftarrow} a\stackrel{\psi}{\rightarrow}a'$ in which $\phi$ is an $F$-generator (resp. $G$-generator) and  $\psi$ is an $X$-generator, extends uniquely to a commutative square
\begin{equation}\label{relation1}
\xygraph{{a}="p0" [r] {a'}="p1" [d] {b'}="p2" [l] {b}="p3" "p0":"p1"^-{\psi}:"p2"^-{\phi'}:@{<-}"p3"^-{\psi'}:@{<-}"p0"^-{\phi}} 
\end{equation}
in which $\phi'$ is an $F$-generator (resp. $G$-generator) and $\psi'$ is an $X$-generator.

We now proceed like in Section 7 of \cite{batanin-berger} and introduce other monads associated to $(P,A).$ 

Let ${\mathbb P}_{f}$ be the category whose objects are quadruples
$(X,K,L,f),$ where $X$ is a $P$-algebra, $K,L$  are objects in $[A,\Set]$ and $f:K\rightarrow L$ is a morphism in $[A,\Set].$

Let ${\mathbb P}_{g}$ be the category whose objects are quadruples $(X,K,L,g),$  where $X$ is a $P$-algebra, $K,L$  are objects in $[A,\Set]$ and $g:K\to \eta^*(X)$ is a morphism in $[A,\Set]$.

The obvious forgetful functors $U_f:{\mathbb P}_f\to [A_0,\Set]\times [A_0,\Set] \times [A_0,\Set]$ and $U_g:{\mathbb P}_g\to [A_0,\Set]\times [A_0,\Set]\times [A_0,\Set]$ are monadic, yielding monads $\Pp_f$ and $\Pp_g$ for which there are propositions analogous to Proposition \ref{Cart}. We leave the details to the reader.

We put $\Pp+2\Aa=(\Pp+\Aa)+\Aa$ which is also a polynomial monad on $[A_0,\Set]\times [A_0,\Set]\times [A_0,\Set]$ as are $\Pp_{f,g}$, $\Pp_f$ and $\Pp_g$.

There is a commutative square of forgetful functors over $[A_0,\Set]\times [A_0,\Set]\times [A_0,\Set]$ 
\begin{equation} 
\xygraph{{\xybox{\xygraph{!{0;(3,0):(0,.5)::} {\mathbb P_{f,g}}="p0" [r] {\mathbb P_f}="p1" [d] {\Alg_P \times [A,\Set] \times [A,\Set]}="p2" [l] {\mathbb P_g}="p3" "p0":"p1"^-{}:"p2"^-{}:@{<-}"p3"^-{}:@{<-}"p0"^-{}}}}
 \end{equation} 
 and all four forgetful functors have left adjoints so that we get a commutative square
of monad morphisms going in the opposite direction and augmented over $P_0$ via cartesian natural transformations: \begin{equation} 
{\xybox{\xygraph{!{0;(1.5,0):(0,.6667)::} {\Pp_{f,g}}="p0" [r] {\Pp_f}="p1" [d] {\Pp + 2\Aa}="p2" [l] {\Pp_g}="p3" "p0":@{<-}"p1"^-{}:@{<-}"p2"^-{}:"p3"^-{}:"p0"^-{}}}}\end{equation}
We thus obtain a commutative square of categorical $P$-algebra maps of the corresponding classifiers, which enables us to analyse the category structure of ${{\Pp}^{\Pp_{f,g}}}$.
%[r(4)] 
\begin{equation}\label{classifiersquare} {\xybox{\xygraph{!{0;(1.5,0):(0,.6667)::} {{\Pp}^{\Pp_{f,g}}}="p0" [r] {{\Pp}^{\Pp_f}}="p1" [d] {{\Pp}^{\Pp + 2\Aa}}="p2" [l] {{\Pp}^{\Pp_g}}="p3" "p0":@{<-}"p1"^-{}:@{<-}"p2"^-{}:"p3"^-{}:"p0"^-{}}}}}
\end{equation}
Finally, we have a map of monads over $\Pp:$
$$\Pp+{\Aa} \to \Pp+2{\Aa}$$
which is the identity on $\Pp$ and sends $\Aa$ to the second copy of $\Aa$ in $(\Pp+\Aa)+\Aa.$  Thus we have a map of classifiers:
$$\Pp^{\Pp+{\Aa}} \to \Pp^{\Pp+2{\Aa}}.$$

\begin{lem}\label{Tfg}
\begin{enumerate}\item
The  classifiers ${{\Pp}^{\Pp_{f,g}}},{{\Pp}^{\Pp_f}},{{\Pp}^{\Pp_g}}, {{\Pp}^{\Pp + 2\Aa}} $ all have the same object-set, and the diagram of (\ref{classifiersquare}) identifies ${{\Pp}^{\Pp_f}},{{\Pp}^{\Pp_g}}$ with subcategories of ${{\Pp}^{\Pp_{f,g}}}$ which intersect in ${{\Pp}^{\Pp + 2\Aa}} $ and which generate ${{\Pp}^{\Pp_{f,g}}}$ as a category.
\item The composite $$\Pp^{\Pp+{\Aa}} \to \Pp^{\Pp+2{\Aa}} \to  \Pp^{\Pp_f}     $$ 
has a left adjoint $p$ such that the counit of the adjunction is the identity. Thus $\Pp^{\Pp+{\Aa}} $ is a reflective subcategory of $\Pp^{\Pp_f}.$
\item The composite $$\Pp^{\Pp+{\Aa}} \to \Pp^{\Pp+2{\Aa}} \to  \Pp^{\Pp_g}     $$ 
has a left adjoint $r$ such that the counit of the adjunction is the identity. Thus $\Pp^{\Pp+{\Aa}} $ is a reflective subcategory of $\Pp^{\Pp_g}.$
\end{enumerate}
\end{lem}
\begin{proof} The proof of the first point of the lemma is completely analogous to the proof of \cite[Lemma 7.3]{batanin-berger}. 

For the second point observe that there is a map of monads $k:\Pp_g\to \Pp+\Aa$ over $\Pp.$ The corresponding restriction functor $k^*:\Alg_{\Pp+\Aa}\to \Alg_{\Pp_g}$ sends a pair $(X,L)$ to the quadruple $(X,\emptyset,L,\iota).$ Here, $\emptyset$ is the initial object in $[A,\Set]$ and $\iota:\emptyset\to U_P(X)$ is the unique morphism. The map $k$ induces the reflection $p:\Pp^{\Pp_g}\to \Pp^{\Pp+\Aa}.$  Explicitly this functor takes an object of $\Pp^{\Pp_g}$ and replaces all $K$-edges by $X$-edges. 
The unit of this adjunction is generated by applying $G$-generators to all $K$-edges in the object. 

For the third point we have a similar map of monads $l:\Pp_f\to \Pp+\Aa$ over $\Pp.$ The restriction functor $l^*:\Alg_{\Pp+\Aa}\to \Alg_{\Pp_f}$ sends a pair $(X,L)$ to the quadruple $(X,\emptyset,L,\iota),$ where  $\iota:\emptyset\to L$ is the unique morphism again. The map $l$ induces the reflection $r:\Pp^{\Pp_f}\to \Pp^{\Pp+\Aa}.$  This reflection $r$ on an object from $\Pp^{\Pp_f}$  replaces all $K$-edges by $L$-edges and the unit is obtained by iterated application of $F$-generators.

\end{proof}

\subsection{Canonical filtration}\label{canfilt}
To shorten notations we put $\ts = \h.$ We say that an object $a$ of $\ts$ is of type $(p,q)$ if $a$ contains exactly $p$ $K$-edges and $q$ $L$-edges, and we call $p+q$ the \emph{degree} of $a$. Let $k\ge 1.$ We define: 
\begin{itemize}
\item $\ts^{(k)}$ to be the full subcategory of $\ts$ spanned by all objects of degree $\leq k;$
\item $\LX^{(k)}.$ to be a full subcategory of $\cop \subset \ho\subset \ts$  spanned by all objects of degree $\leq k;$
\item  $\wa^{(k)}$ to be the full subcategory of $\ts$ spanned by all objects of degree exactly $k$;
\item $\qa^{(k)}$ to be the full subcategory of $\wa^{(k)}$ spanned by all objects of type $(p,k-p)$ such that $p\ne 0;$ 
\item  $\la^{(k)}$ to be the full subcategory of $\wa^{(k)}$ spanned by all objects of type $(0,k);$ 
\item  $\overline{\qa}^{(k)}$ to be the full subcategory of $\ts^{(k)}$ spanned by the objects not contained in $\la^{(k)}$.
\end{itemize}

\begin{proposition}\label{filtration}\label{SS}
For any functor $\mathbf{X}:\h \rightarrow \VV$, the colimit of $\mathbf{X}$ is a sequential colimit of pushouts in $\VV$.

More precisely, for $S_k = \colim_{\ts^{(k)}} \mathbf{X}|_{\ts^{(k)}}$, we get $$S=\colim_{\ts} \mathbf{X}\cong \colim_k S_k,$$ where the canonical map $S_{k-1}\to S_k$ is part of the following pushout square in $\VV$
\begin{equation}\label{KLX}
\xygraph{!{0;(1.5,0):(0,.6667)::}
{Q_k}="p0" [r] {L_k}="p1" [d] {S_k}="p2" [l] {S_{k-1}}="p3"
"p0":"p1"^-{w_k}:"p2"^-{}:@{<-}"p3"^-{}:@{<-}"p0"^-{\alpha_k}
"p2" [u(.4)l(.3)] (-[d(.2)],-[r(.15)])}
\end{equation}
in which $Q_k$ (resp. $L_k$) is the colimit of the restriction of $\mathbf{X}$ to $\qa^{(k)}$ (resp. $\la^{(k)}$).
\end{proposition}
\begin{proof} First, observe that $\ts^{(k-1)}$ is a reflective subcategory of $\overline{\qa}^{(k)}.$  The reflection is constructed as follows. We have the restriction of $r$ from Lemma \ref{Tfg} to $\overline{\qa}^{(k)}$ with the target category $\LX^{(k)}.$ 
Composing it with the inclusion to $\ts^{(k)}$ we have a functor $r':\overline{\qa}^{(k)}\to \ts^{(k)}.$ Observe, however, that the objects of $\overline{\qa}^{(k)}$ either have all edges with $X$ colour or contain at least one $K$-edge. So application of a $G$-generator produces an object from $\ts^{(k-1)}$ which means that we can  factor $r'$ through  $\ts^{(k-1)}$ and so we have a reflection $r'':\overline{\qa}^{(k)}\to \ts^{(k-1)}.$

Similarly, the category $\la^{(k)}$ is also a reflective subcategory of $\wa^{(k)}$ by the restriction of the reflection $p$ from Lemma  \ref{Tfg}.
     
Consider the following diagram of categories where the central square commutes: 
\begin{equation}\label{tamesquare} 
 \xygraph{!{0;(1.5,0):(0,.6667)::}
% the first line rescales the picture, (1.5,0): makes the whole picture 50% larger, (0,.6667): undoes the effect of this in the vertical direction, ie by multiplying the vertical scale by 2/3's
{{\qa}^{(k)}}="p0" [r] {\wa^{(k)}}="p1" [d] {\ts^{(k)}}="p2" [l] {\overline{\qa}^{(k)}}="p3"
% this line drops the objects of the square in clockwise order, and saves positions
"p0":"p1"^-{}:"p2"^-{}:@{<-}"p3"^-{}:@{<-}"p0"^-{}
% this line draws the arrows of the square in a clockwise direction
% the code @{<-} reverses the direction of an arrow
% in this case there are no labels on arrows, hence the empty {} which are there anyway because I have an emacs macro which generates the code for a square
"p3" :@{<-}[l] {\ts^{(k-1)}} "p1" :@{<-}[r] {\la^{(k)}}
% this line adds the extra arrows
} \end{equation}
Restricting $X$ to this subcategory and taking colimits we obtain a commutative diagram like (\ref{KLX}). We only need to know that this is a pushout diagram in $\VV.$

A closer inspection of the central square (\ref{tamesquare}) reveals that it is a categorical pushout of a special kind: the category $\ts^{(k)}$ is obtained as the set-theoretical union of the categories $\overline{\qa}^{(k)}$ and $\wa^{(k)}$ along their common intersection $\qa^{(k)}$. Indeed, away from this intersection, there are no morphisms in $\ts^{(k)}$ between objects of $\overline{\qa}^{(k)}$ and objects of $\wa^{(k)}$. By Lemma 7.13 from \cite{batanin-berger}, this implies that (\ref{KLX}) is a pushout square in $\VV$.

\end{proof}

\subsection{End of proof of Transfer Theorem}\label{semitransferend} We are ready to complete the proof of Theorem \ref{semitransfer}.

\begin{proof} By Lemma \ref{a+a} for a  unary tame $(P,A)$ there exists a categorical $A$-presheaf $\lambda$ such that  $\lambda\to \cop$ is a levelwise  final subcategory (we drop the forgetful functor $\eta^*$ from the notation for simplicity). %then the classifier $\ho$ also contains a final subcategory of unary morphisms,
Let also ${3}\Aa  = \Aa+\Aa+\Aa$ be the obvious  polynomial monad over $\Pp.$
We then define  a subpresheaf $\td$ as a pullback  
\begin{align*} \label{}
\xymatrix{
\td \ar[r]^{} \ar[d]_{} &\hos \ar[d]_{}\ar[r]^{}   &\ho \ar[d]^{D} \\
\lambda \ar[r]^{}&\Pp^{2\Aa} \ar[r]  & \cop
}
\end{align*}
where the morphism $D:\ho\to\cop$ is induced by the folding map of polynomial monads $\Aa+\Aa\to \Aa.$ 
To prove that the subcategory $\td$ is final we first observe that the functor $D$ is a discrete fibration. Indeed, the effect of  $D$ on objects is just a replacement of any $L$-coloured edge by a $K$-coloured edge. It is then obvious that the unique lifting condition of an arrow from $\cop$ is satisfied because colouring is defined uniquely by the source from $\ho.$  

Final functors and discrete fibrations form one of the two comprehensive factorisation systems in $\Cat.$ It is a well-known property of this factorisation system that  final functors are stable under pullbacks along discrete fibrations (see \cite[Proposition 1.10]{BergerK}). 

Finally, we use the pullback (\ref{exact fibration u}) and cartesianess of the monad $\Aa$ to see that the following square is a pullback as well.  

\begin{equation}\label{tcart} 
\xymatrix{
\Aa(\td) \ar[r]^{} \ar[d]_{} & \Aa(\Pp^{3\Aa}) \ar[d]^{} \\
\td \ar[r]_{} & \Pp^{3\Aa}}
\end{equation} 

The inclusion \begin{equation}\label{111} \td\cap \qa^{(k)}\inc\ho\cap\qa^{(k)}\end{equation} 
can be obtained as a pullback of the form

\begin{equation*} 
\xymatrix{
\td\cap \qa^{(k)} \ar[r]^{} \ar[d]_{} & \ho\cap\qa^{(k)}\ar[d]^{} \\
\td \ar[r]_{} & \ho}
\end{equation*} 
in which the right vertical inclusion is a levelwise discrete fibration. Hence, the inclusion (\ref{111}) is final. Consider now a pushout
\begin{equation*} 
\xymatrix{
\td\cap \qa^{(k)} \ar[r]^{} \ar[d]_{H}& \hos\cap \qa^{(k)}\ar[d] \ar[r]   & \ho\cap\qa^{(k)}\ar[d]^{} \\
\td^{(k)} \ar[r]&\ua^{(k)} \ar[r]   & \qa^{(k)}}
\end{equation*} 
In this pushout the functor $H$ adds $F$-generators. Clearly the most right vertical morphism in this pushout is just the canonical inclusion. Since the left class of the comprehensive factorisation system is closed under pushouts the bottom composite in this pushout is thus a final inclusion. 
One can also prove easily  that $\td^{(k)},  \ua^{(k)}$ and $\qa^{(k)}$ satisfy the suitable analogues of the pullback condition (\ref{tcart}).

Now, given an algebra $(X,K,L,f,g)$ of $P_{f,g}$ we see that the colimit $Q_k$ of the restriction of $\mathbf{X} = \widetilde{(X,K,L,f,g)}$ over $\qa^{(k)}$ is naturally a retract of the colimit of its restriction on $\ua^{(k)}$.  

To analyse the structure of this colimit over $\ua^{(k)}$  we first observe like in Lemma \ref{shufflecoproduct} that the colimit of the 
restriction of $\mathbf{X}$ on $\hos\cap \qa^{(k)}$ 
 (as an $A$-presheaf)   is given by the formula
$$\coprod_{p+q =k,p>0,s\ge 0, \sigma\in Sh_{s,p,q}} \sigma(\odot^{s+p+q}_{P})(\underbrace{\eta^*(X),\ldots,\eta^*(X)}_s,\underbrace{K,\ldots,K}_p,\underbrace{L,\ldots,L}_q).$$  
The colimit of $\mathbf{X}$ over $\ua^{(k)}$ is obtained from this coproduct by taking further colimit. The additional morphisms which appear here are morphisms induced by $F$-generators (that is we have one additional generating morphism for each $K$ induced by replacement  of $K$-coloured edge by an $L$-coloured edge) and, therefore, this colimit is a coproduct of colimits over punctured cubes for the nonsymmetric convolution $\odot_P.$ 

Using the argument similar to the proof of Proposition \ref{retract} we see that the map $w_k:Q_k\to L_k$ is a retract of a coproduct of corner map for $\odot_P^{k},\ n\ge 0,$ for which  the punctured cube morphisms are of the form $\sigma(\boxx^{s+p+q})(X,\ldots,X,f.\ldots,f)$, where $\sigma$ runs over   $(s,p+q)$-shuffles, $f:K\to L$ is a trivial cofibration  with cofibrant domain   and $X$ is a cofibrant $A$-presheaf. Hence, $w_k$ is a trivial cofibration itself and so by Proposition \ref{filtration} the morphism $p:X\to S$ in \ref{free cofibration} is a weak equivalence as well.  We finished the proof.  \end{proof}  

\begin{remark} Notice that the conditions of the Transfer Theorem concern the properties of nonsymmetric convolution of $(P,A)$ on $[A,\VV]$ and we don't even assume that $\VV$ itself is a monoidal model category. 
\end{remark}

\subsection{Beck-Chevalley morphisms of substitudes}\label{BCsub}
 
\begin{definition} Let $(f,g):(P,A) \to (Q,B)$ be a map of  substitudes. We call the morphism Beck-Chevalley if the following square is a Beck-Chevalley square:

\begin{equation*}\label{BSmorphism} \xygraph{!{0;(2.5,0):(0,.5)::}
{\Alg_P
(\VV)}="p0" [r] {\Alg_Q(\VV)}="p1" [d] {[B,\VV]}="p2" [l] {[A,\VV]}="p3"
"p0":@<-1ex>@{<-}"p1"_-{f^*}|-{}="cp":@<1ex>"p2"^-{(\eta_Q)^*}|-{}="ut":@<1ex>"p3"^-{g^*}|-{}="c":@<-1ex>@{<-}"p0"_-{(\eta_P)^*}|-{}="us"
"p0":@<1ex>"p1"^-{f_!}|-{}="dp":@<-1ex>@{<-}"p2"_-{(\eta_Q)_!}|-{}="ft":@<-1ex>@{<-}"p3"_-{g_!}|-{}="d":@<1ex>"p0"^-{(\eta_P)_!}|-{}="fs"
"dp":@{}"cp"|-{\perp} "d":@{}"c"|-{\perp} "fs":@{}"us"|-{\dashv} "ft":@{}"ut"|-{\dashv}}
\end{equation*}

\end{definition}

\begin{defin} Let $(P,A)$ be a substitude and the categories $[A,\VV]$ and $\Alg_P(\VV)$ are equipped with semimodel structures in such a way that the pair  $(\eta_P)^*\dashv (\eta_P)_!$ is a Quillen adjunction. In this case, we will say the $(P,A)$ is derivable (with respect to fixed semimodel structures). A morphism of derivable substitudes is a morphism of substitudes for which the commutative square of adjunctions is a square of Quillen adjunctions.
 
 \end{defin}

\begin{definition} We call a  morphism of derivable substitudes  $(f,g):(P,A) \to (Q,B)$  a homotopy Beck-Chevalley morphism if the square above (\ref{BSmorphism}) is a homotopy Beck-Chevalley square. 

\end{definition}

\begin{proposition}\label{2of3} The class of Beck-Chevalley morphisms satisfies the two out of three property.
The same is true for homotopy Beck-Chevalley morphisms of derivable substitudes.

\end{proposition}
\begin{proof} 
It is not hard to prove that the Beck-Chevalley transformation (\ref{BC transformation}) for a horizontal pasting of two Beck-Chevalley square is obtained as a pasting of Beck-Chevalley transformations of each squares. Hence the result.

\end{proof}

\begin{theorem}\label{BCperfect} Let $(f,g):(P,A) \to (Q,B)$ be a  Beck-Chevalley morphism 
between unary tame left Quillen 
substitudes for which  the semimodel structures on  $\Alg_P(\VV)$ and $\Alg_Q(\VV)$ are transferred from some cofibrantly generated semimodel structures on $[A,\VV]$ and $[B,\VV].$  If   $g_!\dashv g^*$ is a pair of Quillen equivalences then 
 $f_!\dashv f^*$ is also a pair of Quillen equivalences.  
\end{theorem}
\begin{proof}
Immediate from Corollary \ref{lifting QE for transfer}.
\end{proof}

\part{Localisation of algebras of substitudes} \label{part:localization}

In this part we study a particular but very important case of our Transfer Theorem for $\Alg_P(\VV)$, when the model structure on $[A,\VV]$ is obtained as a localisation of the projective model structure $[A,\VV]_{proj}.$ %which itself is obtained as transferred model structure along $i^*:[A,\VV]\to [A_0,\VV].$  

\section{Localisation of presheaf categories} \label{sec:localisers}

\subsection{Fundamental localisers} 

We recall some definitions and results from the homotopy theory of small categories founded by Grothendieck.

Let $\Ww$ be a class of functors between small categories. Let $1$ be the terminal category.
\begin{defin} The class $\Ww$ is called fundamental localiser if it satisfies the following conditions:
\begin{enumerate}\item $\Ww$ contains all identities;
\item $\Ww$ satisfies the two out of three property;
\item If $i:A\to B$ has a retraction $r:B\to A$ and $r\cdot i:B\to B$ is in $\Ww$ then $i$ is in $\Ww;$ 
\item If $A$ is a small category with a terminal object then $A\to {1}$ is in $\Ww;$
\item If in a commutative triangle of functors between small categories
\begin{equation*}\label{FW}
	\xymatrix{
		A\ar[rr]^u \ar[rd]_v && B \ar[ld]^w \\
		& C}
	\end{equation*}
 the functor $u/c:A/c\to B/c$ is in $\Ww$ for each object $c\in C$ then $u$ is in $\Ww.$

\end{enumerate}

\end{defin}
 
 \begin{definition} A class of functors between small categories is called a weak fundamental localiser if it satisfies properties (1)-(4) above and property (5) for $w=id:B\to B.$
 \end{definition}

For a (weak) fundamental localiser $\Ww$ we call its elements $\Ww$-equivalences.
A small category $A$ is called $\Ww$-aspherical if the unique functor $A\to 1$ is in $\Ww.$ A functor $u:A\to B$  is $\Ww$-aspherical if  for all   $b\in B$ the category $u/b$ is $\Ww$-aspherical. Any $\Ww$-aspherical functor is a $\Ww$-equivalence. 

 \begin{remark} The condition (5) of the definition above can be formulated as follows. If the morphism of  presheaves  $C^A\to C^B$ in $\Cat$ induced by (\ref{FW}) is the fiberwise $\Ww$-equivalence then $u$ is a $\Ww$-equivalence.   
 \end{remark}

There is a notion of proper fundamental localiser (\cite[Definition 4.3.21]{cis06} which we don't reproduce here. Instead we recall that  $\Ww$ is a proper fundamental localiser if and only if there exists a set $\mathcal{S}$ of small categories which $\Ww$ makes trivial, in the sense that for any $A\in \mathcal{S}$ the functor $A\to 1$ is in $\Ww$, and $\Ww$ is minimal with this property  \cite[Theorem 6.1.11]{cis06}. We use notation $\Ww=\Ww(\mathcal{S})$ in this case.

\begin{example}\label{Winf}  Let $\mathcal{S}$ contain only one element: the arrow category $0\to 1.$  
Then 
$\Ww(\mathcal{S})$ is equal to $\Ww_{\infty}$ the class of functors whose nerve is a weak equivalence. This is easy to see because the arrow category has a terminal object and so is trivialised by every fundamental localiser. But $\Ww_{\infty}$ is the minimal fundamental localiser \cite[Corollary 4.2.19]{cis06}.

\end{example}

\begin{example}\label{Wn} Let $\mathcal{S}$ consist of a unique element $S^{k+1},$ where $S^{k+1}$ is a small category which has homotopy type of $(k+1)$-sphere.
For example, one can take $S^{k+1} = \Delta/\partial(\Delta_{k+2})$ as in \cite[Section 9.2]{cis06}.

Then $\Ww(\mathcal{S})$-equivalences are $k$-equivalences, that is, functors  inducing isomorphisms of homotopy groups up to $k$ (including $k$) \cite[Corollaire 9.2.15]{cis06}.
The  $\Ww(\mathcal{S})$-aspherical categories are exactly the $k$-connected categories. 
Following \cite{cis06} we will denote the localiser $\Ww(\mathcal{S})$ simply $\Ww_k.$ 

The localiser $\Ww_0$ contains functors that induce isomorphism on sets of connected components, and $\Ww_0$-aspherical categories are connected categories. Without loss of generality, we can assume that $\Ww_0$ contains all other fundamental localisers that we consider in this paper \cite[Proposition 9.3.2]{cis06}.  
 
\end{example}

\subsection{Locally constant presheaves}
 {
Let $A$ be a small category and let $\VV$ be a model category. Let $Ho[A,\VV]$ be  the localisation of the category of covariant presheaves $[A,\VV]$ with respect to levelwise weak equivalences.
 
  \begin{definition} A presheaf $F: A\to \VV,$  is called $\Ww$-locally constant %\cite{cis06}[Definition ?]
if for any  $\Ww$-aspherical small category  $A'$ and any functor $u:A'\to A$ the presheaf $u^*(F):A'\to \VV$ is homotopy equivalent in $Ho[A',\VV]$  to a constant presheaf. 
\end{definition} }

Let $LC_{\Ww}[A,\VV]$ denote the full subcategory of $Ho[A,\VV]$ formed by $\Ww$-locally constant presheaves. It is easy to see that any functor $u:A\to B$ induces a restriction $u^*: LC_{\Ww}[B,\VV]\to LC_{\Ww}[A,\VV].$ 

Call a functor between small categories $u:A\to B$ {\it a local $\Ww$-equivalence} if $u^*: LC_{\Ww}[B,\VV]\to LC_{\Ww}[A,\VV]$  is an equivalence of categories for any model category $\VV.$ Let us denote by $\Ww_{loc}$ the class of local $\Ww$-equivalences  between small categories.

\begin{theorem}\label{almostcisinski} 
\begin{enumerate} 
\item The class $\Ww_{loc}$   is a weak fundamental localiser. 
\item Any $\Ww$-aspherical category is also $\Ww_{loc}$-aspherical.
\item If $\Ww$ is proper then $\Ww\subseteq \Ww_{loc}.$
Thus for a proper fundamental localiser any $\Ww$-equivalence induces an equivalence of homotopy categories of $\Ww$-locally constant presheaves.    
 \end{enumerate}

\end{theorem}

\begin{proof} The proof of the first statement follows word for word  the proof of \cite[Theorem 1.3]{cis}.

The second statement follows from the following tautological argument.  Let $A$ be a $\Ww$-aspherical category. We have to prove that it is $\Ww_{loc}$-aspherical. 
That is, $A\to 1$ is a $\Ww_{loc}$-equivalence. This amounts to the statement that every $\Ww$-locally constant presheaf $F:A\to \VV$ is equivalent to a constant presheaf. 
But $F:A\to \VV$ is locally constant by definition if, for any $\Ww$-aspherical $A'$ and any functor $u:A'\to A$, the presheaf $u^*(F)$ is equivalent to a constant. We take $u=id$ to finish the proof.

The third statement follows from \cite[Proposition 6.1.16]{cis06}.
\end{proof} 

We do not reproduce the argument of Cisinski's Theorem 1.3 from \cite{cis} but we want a version of  one statement that  Cisinski calls the Formal Serre spectral sequence  \cite[Proposition 1.24]{cis}. For the proof we refer the reader to \cite{cis} again.

\begin{definition}[\cite{cis06}, 6.4.1] \label{W-locally-constant-functor}  A functor $u:A\to B$ between two small categories is called $\Ww$-locally constant if for any morphism $b\to b'$ in B the functor between comma categories $u/b\to u/b'$ is a   $\Ww$-equivalence. \end{definition}

\begin{proposition} If $u$ is $\Ww_{loc}$-locally constant then for any model category $\VV$ the total left derived functor of $u_!:$ 
$$Ho([A,\VV])\to Ho ([B,\VV])$$
preserves $\Ww$-locally constant presheaves. 
\end{proposition}
\begin{corollary}\label{uop=lc}  If $u^{op}:A^{op}\to B^{op}$ is $\Ww_{loc}$-locally constant then the total right derived functor  
$$u_*: Ho([A,\VV])\to Ho ([B,\VV])$$
preserves $\Ww$-locally constant presheaves for
 any model category $\VV.$
\end{corollary}

\begin{proof} It is enough to apply the previous proposition to $\VV^{op}$ and observe that  $F:B\to  \VV$ is a $\Ww$-locally constant presheaf if and only if   $F^{op}:B^{op}\to \VV^{op}$ is $\Ww$-locally constant. This follows readily  from the fact that the class $\Ww$ is closed under $(-)^{op}$ \cite[Proposition 1.1.22]{Maltsiniotis}. \end{proof}
\begin{remark}
Notice that we will need a full model structure on $\VV$ in order to have a model structure on $\VV^{op}.$  Though this corollary can be proved  in greater generality (replacing $\VV$ by a derivator) we don't need it in this paper.  

\end{remark}

\begin{example}\label{inftyexample} Let us explain the connection to  the original Cisinski's Theorem 1.3 from \cite{cis} more precisely. 

 Let $\Ww= \Ww(\mathcal{S}).$ Let us call a  presheaf $F: A\to \VV$ on a small category $A$   $\mathcal{S}$-locally constant %\cite{cis06}[Definition ?]
if for any  $A'\in \mathcal{S}$  and any functor $u:A'\to A$, the presheaf $u^*(F):A'\to \VV$ is homotopy equivalent in $Ho[A',\VV]$  to a constant presheaf. Let $LC_{\mathcal{S}}[A,\VV]$ denote the homotopy category of $\mathcal{S}$-locally constant presheaves. Call a functor between small categories $u:A\to B$ {\it a local $\mathcal{S}$-equivalence} if $u^*: LC_{\mathcal{S}}[B,\VV]\to LC_{\mathcal{S}}[A,\VV]$  is an equivalence of categories for any (semi)-model category $\VV.$ Let us denote by $\Ww_{\mathcal{S}}$ the class of local $\mathcal{S}$-equivalences  between small categories. 

Cisinski's argument works in this case too and so $\Ww_{\mathcal{S}}$ is a weak fundamental localiser.   

Let $\mathcal{S}$ as in the Example \ref{Winf}. Then a presheaf $F:A\to \VV$ is $\mathcal{S}$-locally constant if and only if $F(\alpha)$ is a weak equivalence in $\VV$ for every morphism $\alpha$ in $A.$  That is, it is a locally constant presheaf in the sense of \cite{cis}. 

But $\Ww_{\infty}$ is also the minimal weak fundamental localiser \cite[Theorem 6.1.18]{cis06} so any $\Ww_{\infty} = \Ww(\mathcal{S})$-aspherical category is also $\Ww_{\mathcal{S}}$-aspherical. So any Thomason (a.k.a. simplicial) weak equivalence between  small categories induces an equivalence between homotopy categories of locally 
constant presheaves in the sense of Cisinski. This is the content of  \cite[Theorem 1.3]{cis}. 
\end{example}

Unfortunately, while the case $\Ww=\Ww_{\infty}$ admits the explicit simple description of the $\Ww$-locally constant presheaves  given in the Example \ref{inftyexample}, we do not know other fundamental localisers where the $\Ww$-locally constant presheaves admit such an elementary  description.   

\subsection{Local (semi)model structures on  presheaves}\label{locofpre}

We also want to establish an analogue of Cisinski's local model structure on presheaves categories \cite[Proposition 2.3]{cis}. 
In contrast with \cite{cis} we don't ask for the basis category $\VV$ to be left proper. This results to a not necessary full local model structure for presheaves in $\VV.$ For a left proper $\VV$ we will have full local model structure on presheaves, which is also left proper.

\begin{lemma} Let $\VV$ be a combinatorial model category. Let $A$ be a small category and $e:A\to 1$ be the unique functor. Then there exists a left Bousfield (semimodel) localisation $[A,\VV]_{proj}^0$ of the projective model structure  $[A,\VV]_{proj}$ such that the induced adjoint pair $e_!\vdash e^*$ is a Quillen equivalence between $[A,\VV]_{proj}^0$ and $\VV.$ The local objects of this localisation are presheaves $A\to \VV$ which are levelwise fibrant and which 
are levelwise equivalent to constant presheaves.  
\end{lemma} 
\begin{remark} The functor $e^*$ is, of course, the diagonal functor.
Its left adjoint $e_!$ is the colimit over $A$   and  its right adjoint $e_*$ is just the limit over $A.$
\end{remark}

\begin{proof} 
Since $\VV$ is combinatorial, so is $[A,\VV]$ with the projective model structure.  Following Cisinski \cite{cis} we fix a regular cardinal $\alpha$ such that every object of $[A,\VV]$ is an $\alpha$-filtered colimit of a set $T$ of $\alpha$-small objects, the class of weak equivalences is stable with respect to $\alpha$-filtered colimits, and there exists a cofibrant resolution functor $Q$ that
preserves $\alpha$-filtered colimits.

We now consider the essentially small set of  arrows in $[A,\VV]$ of the form:
$$\epsilon:QF\to e^*e_!(QF),$$
where $F$ is an $\alpha$-small presheaf in $T$ and      
$\epsilon$ is the unit of the adjunction. Let $[A,\VV]_{proj}^0$ be the left Bousfield localisation of $[A,\VV]$ with respect to this class of morphisms. As in the proof of \cite[Proposition 2.3]{cis} we deduce then that for any $F:A\to \VV$ the map 
$$\epsilon:QF\to e^*e_!(QF)$$
is a local weak equivalence, and, hence $\epsilon:F\to e^*e_!(F)$ is a local weak equivalence for any cofibrant $F.$ 

Let now $E$ be a fibrant object in $[A,\VV]$ which means that it is levelwise fibrant. It is a local object in the localised category if and only if for any cofibrant $F$ the unit $\epsilon$ induces a weak equivalence:
 $$[A,\VV](\bar{F},E) \leftarrow [A,\VV](\overline{e^*e_!(F)},E)$$
 where $\bar{F}$ is a cosimplicial resolution of the presheaf $F$ in $[A,\VV].$ We will show that any local fibrant $E$ is equivalent to a constant presheaf. It is not hard to see that there is a weak equivalence of cosimplicial presheaves:
 $$\overline{e^*e_!(F)} \to  e^*e_{!}(\bar{F})$$   
and we have isomorphisms:  
  $$ [A,\VV](e^*e_!(\bar{F}),E) \simeq [1,\VV](e_!(\bar{F}),e_*({E})) \simeq [A,\VV](\bar{F},e^*e_*({E})). $$
So we have a weak equivalence 
$$[A,\VV](\bar{F},E) \gets  [A,\VV](\bar{F},e^*e_*({E}))$$ 
 induced  by  counit $\eta$ of the adjunction $e^*\dashv e_*,$ for each cofibrant $F.$ 
Hence, $$\eta:e^*e_*(E)\to E$$ is a levelwise weak equivalence and $E$ is equivalent to the constant presheaf.

It is now obvious that $e_!: [A,\VV]_{proj}^0\to [1,\VV]= \VV$ is the left part of a Quillen equivalence. 

\end{proof}

\begin{corollary} \label{cor:Bous loc for inj and Reedy}
There exists a left Bousfield (semimodel) localisation $[A,\VV]^0_{inj}$ of the injective model structure on $[A,\VV]$ with respect to the same set of local equivalences such that
the local objects are injectively fibrant presheaves equivalent to the constant presheaves. The identity functor
$$[A,\VV]^0_{inj} \to  [A,\VV]^0_{proj}$$
is a right Quillen equivalence.

Moreover, if $A$ is a Reedy category then there is a similar localisation $[A,\VV]^0_{reedy}$ of the Reedy model structure and the identity functor gives right Quillen equivalences:
$$[A,\VV]^0_{inj} \to [A,\VV]^0_{reedy}\to  [A,\VV]^0_{proj}$$

\end{corollary} 

\begin{proof} For the injective model structure let us lift the localisation on $[A,\VV]_{proj}$ along the Quillen adjunction given by the identity 
$[A,\VV]_{inj}\to [A,\VV]_{proj}$. By Theorem  \ref{loclifting}  we have a Quillen equivalence between $[A,\VV]^0_{inj} $ and $[A,\VV]^0_{proj}.$ The same argument works for the Reedy model structure.  
\end{proof}

\begin{remark} It is interesting to notice  that  we have a Quillen equivalence between $[A,\VV]^0_{inj}$ and $[1,\VV]$ given by the pair $e^*\dashv e_*.$   But neither $e_*$ or $e_!$ is a right or left equivalence for the localised  Reedy model structure in general. But under some conditions we may have such Quillen equivalences. This subtle property is discussed in \cite{dag,RSS}. 
\end{remark} 

We want to use the constructed model structures to induce a model structure on $[A,\VV]$ with local objects given by $\Ww$-locally constant presheaves. 

\begin{theorem} \label{W local structure} Let $\Ww$ be a proper fundamental localiser and $\VV$ a combinatorial model category.  Then there exist  (semimodel) localisations of the projective and injective model structures on $[A,\VV]_{proj}^\Ww, [A,\VV]_{inj}^\Ww$ such that the local objects are fibrant (in their respective model structures) and $\Ww$-locally constant presheaves.   
If $A$ is Reedy then there exists a localisation $[A,\VV]_{reedy}^\Ww$ with similar properties.  The identity functors give right Quillen equivalences:
$$[A,\VV]^\Ww_{inj} \to [A,\VV]^\Ww_{reedy}\to  [A,\VV]^\Ww_{proj}.$$

\end{theorem} 
\begin{proof} 

Let start from the projective model structure on $[A,\VV].$ Let $A'$ be a $\Ww$-aspherical category and let $u:A'\to A$ be a functor. Then consider the following lifting   of localisations (where $[A,\VV]_{proj}^{u} $ is defined as the lift of $[A',\VV]_{proj}^0$ along $u^*$):

\begin{align} \label{diagram:lifting}
\xymatrix{
[A,\VV]_{proj} \ar@<2.5pt>[r]^{id} \ar@<-2.5pt>[d]^{u^*}
& [A,\VV]_{proj}^{u} \ar@<2.5pt>[l]^{id} \ar@<-2.5pt>[d]^{u^*} \\
[A',\VV]_{proj} \ar@<2.5pt>[r]^{id}\ar@<8.5pt>[u]^{u_!}  
&[A',\VV]_{proj}^0 \ar@<2.5pt>[l]^{id} \ar@<8.5pt>[u]^{u_!}
}
\end{align}

The local fibrant objects in $[A,\VV]_{proj}^u$ are exactly the presheaves that are levelwise fibrant and become equivalent to a constant presheaf after the restriction along $u$ by Theorem \ref{loclifting}.

Recall that there exists a model structure on $Cat$ whose weak equivalences are exactly the functors from $\Ww,$ whose fibrations are $\Ww$-smooth and $\Ww$-proper functors (see \cite[Section 5.3.1]{cis06} for the definition) and this model structure is  proper because $\Ww$ is proper \cite[Theorem 5.3.14]{cis06}.  Let $\phi:Q\to A'$ be a trivial fibration in this model structure and let $A'$ be $\Ww$-aspherical. We then claim that the localisation $[A,\VV]_{proj}^u$ coincides with $[A,\VV]_{proj}^{u(\phi)}.$ Indeed, we can prove this statement if one can prove that $[A',\VV]_{proj}^0$ coincides with 
$[A',\VV]_{proj}^{\phi}$ because   $[A,\VV]_{proj}^{u(\phi)}$ can be obtained as two successive liftings with intermediate lifting  $[A',\VV]_{proj}^\phi.$ 
We then want to prove that the local fibrant objects in  $[A',\VV]_{proj}^{\phi}$ and $[A',\VV]_{proj}^0$ coincide. 

Local objects in $[A',\VV]_{proj}^{\phi}$ are (fibrant) presheaves $F:A'\to \VV$ such that $\phi^*(F)$ are equivalent to constant presheaf given by $lim_Q (\phi^*(F)).$   
We have to prove then that $F$ is equivalent to a constant presheaf. 

First, let us show that  $\phi$ is an initial functor. Indeed, for every object $a\in A'$ the natural inclusion functor $Q_a\to a/\phi$ is from $\Ww$ because $\phi$ is smooth. Let $1$ be the terminal category. Recall that the definition of $\Ww$-smoothness \cite[5.3.1]{cis06} requires $Q_a\to a/\phi$ to be $\Ww$-aspherical, where $Q_a$ is the pullback  

\begin{align} %\label{diagram:lifting}
\xymatrix{
Q_a \ar@<0.5pt>[r]^{} \ar@<0.5pt>[d]^{}
& 1 %\ar@<2.5pt>[l]^{} 
\ar@<0.5pt>[d]^{a}
 \\
Q \ar@<0.5pt>[r]^{\phi}%\ar@<2.5pt>[u]  
&A' %\ar@<2.5pt>[l] 
%\ar@<2.5pt>[u]
}
\end{align}

Since $\phi$ is a fibration and $a:1\to A'$ is a $\Ww$-equivalence we deduce that the left vertical arrow is a weak equivalence and since $\phi$ is trivial the top arrow is also a weak equivalence. This means that $Q_a$ is at least connected because $\Ww\subseteq \Ww_0.$ Hence all comma-categories  $a/\phi$ are connected and $\phi$ is initial. Thus $lim_Q (\phi^*(F))\simeq lim_{A'} F.$ 

We now know that $\phi^*(F)(q) \to lim_Q (\phi^*(F))$ is a weak equivalence for each $q\in Q.$ But $\phi$ is surjective on objects and, hence, $F(a)\to  lim_{A'} F$ is a weak equivalence for each $a,$ therefore, $F$ is equivalent to a constant presheaf. 

Now, for each object $a \in A$ let $a:1\to A$ be a functor which picks up $a$ as above. We factor it as $1\to R(a)\to A$ in the model category above, where $1\to R(a)$ is a trivial cofibration (so $R(a)$ is $\Ww$-aspherical) and $u_a:R(a)\to A$ is a fibration.  Consider the localisation $[A,\VV]_{proj}^\Ww$ obtained by successively  lifting the localisations $[R(a),\VV]_{proj}^0$ along the functors $u_a^*, a\in A$ (this is where we need smallness of $A$ to ensure that we can do all these liftings). 

We claim that in $[A,\VV]_{proj}^\Ww$ the local objects are exactly the $\Ww$-locally constant fibrant presheaves. We would prove this statement if we can show that for any $\Ww$-aspherical $A'$ and any $u:A'\to A$ the local weak equivalences in $[A,\VV]_{proj}^u$ are among the local weak equivalences in $[A,\VV]_{proj}^\Ww.$ 

Let $b:1\to A'$ be a functor and $a:1\to A'\stackrel{u}{\to} A.$ We also can factor $b$ in the model category on $Cat$ above as a cofibration followed by a trivial fibration $1\to Q\stackrel{\phi}{\to} A'$. Observe, however, that $1\to Q$ is also a trivial cofibration because $A'$  is $\Ww$-aspherical. 

Consider the commutative square

  \begin{align} %\label{diagram:lifting}
\xymatrix{
1 \ar@<0.5pt>[r]^{} \ar@<0.5pt>[d]^{}
& R(a) %\ar@<2.5pt>[l]^{} 
\ar@<0.5pt>[d]^{u_a}
 \\
Q \ar@<0.5pt>[r]^{u(\phi)}%\ar@<2.5pt>[u]  
&A %\ar@<2.5pt>[l] 
%\ar@<2.5pt>[u]
}
\end{align}

We then have a lifting $Q\to R(a)$ in this square, and this shows that all local weak equivalences from  $[A,\VV]_{proj}^{u(\phi)}$ (and so from $[A,\VV]_{proj}^u$) are among the local weak equivalences from $[A,\VV]_{proj}^{u_a}.$ 

The proofs for the injective and Reedy structures are similar.

\end{proof}

The following Theorem is an analogue of the corresponding statements from Section 2 of \cite{cis}.

\begin{theorem}\label{WQE} Let $u:A\to B$ be a functor between two small categories and $\Ww$ be a proper fundamental localiser. Then
$$u^*:[B,\VV]_{proj}^\Ww\to [A,\VV]_{proj}^\Ww$$
is a right Quillen functor which is a Quillen equivalence if $u\in \Ww.$

If $u^{op}:A^{op}\to B^{op}$  is a $\Ww$-locally constant  functor (see Definition \ref{W-locally-constant-functor}) then 
$$u^*:[B,\VV]_{inj}^\Ww\to [A,\VV]_{inj}^\Ww$$
is a left Quillen functor which is a Quillen equivalence if $u\in \Ww.$ 

\end{theorem}

The proof is analogous to \cite[Propositions 2.6,2.8]{cis}.   

\subsection{Truncated homotopy types}

Even though we don't know a satisfactory explicit characterisation of $\Ww$-locally constant presheaves, the proof above shows that there exists a set of $\Ww$-aspherical categories and functors $R(a)\to A$ which `detects' $\Ww$-locally constant presheaves. Also, under some circumstances we can describe the locally constant presheaves in terms of the $\Ww_{\infty}$-locally constant presheaves.

Recall \cite[Corollary 4.2.18]{cis06} that for every accessible fundamental localiser $\Ww$ there exists a model structure on simplicial sets such that its cofibrations are monomorphisms and its weak equivalences are exactly maps between simplicial sets such that the induced  functor between categories of simplices belongs to $\Ww.$ We call the fibrant objects in this model category \textit{$\Ww$-local simplicial sets} and contractible objects \textit{$\Ww$-aspherical simplicial sets}.  For $\Ww =\Ww_{\infty}$ the corresponding model structure is just the classical Kan-Quillen model structure.  

\begin{definition} Let $\VV$ be a model category. We call it $k$-truncated if for any objects $X,Y$, the mapping space $Map_V(X,Y)$ is a $\Ww_k$-local simplicial set. This means that all  homotopy groups of this mapping space  vanish for $i>k$ for any choice of a base point.
\end{definition}

\begin{theorem}\label{kloc=lok} Let  $\VV$ be an $k$-truncated model category. Then the identity functor
$$[A,\VV]_{proj}^{\Ww_{r}}\to [A,\VV]_{proj}^{\Ww_{\infty}}$$
is a right Quillen equivalence for all $r\ge k+1.$ 
\end{theorem}

\begin{proof} It is enough to prove the theorem for $r=k+1.$ Since $\Ww_{\infty}$ is the minimal  fundamental localiser, every $\Ww$-locally constant presheaf is also $\Ww_{\infty}$-locally constant. So, for any $\Ww$  the identity functor  $[A,\VV]_{proj}^\Ww\to [A,\VV]_{proj}^{\Ww_{\infty}}$ is a right Quillen functor. To prove that this is a Quillen equivalence for $\Ww=\Ww_{k+1}$ we have to prove that these categories have the same class of local objects. 

It is not hard to see that this amounts to the following statement: every $\Ww_{\infty}$-locally constant presheaf $F\to \VV$ on a  $\Ww_{k+1}$-aspherical category $A'$ is equivalent to a constant presheaf.    

The category $A'$ is connected. Since the homotopy category of $\Ww_{\infty}$-locally constant presheaves depends only on the homotopy type of $A'$  we  can also think of $A'$ as a monoid \cite{McDuff}. Then a presheaf $F:A'\to \VV$ is just a morphism of monoids $F:A'\to \VV(X,X)$ for some $X\in \VV.$ 
Since $F$ is locally constant this functor can be lifted to a map of monoids $F':\Omega B(A')\to Map_V(X,X)$  where $\Omega B(A')$ is the simplicial group whose homotopy type is $\Omega N(A')$ \cite{DwyerKan}. But $N(A')$ is $(k+1)$-connected, so $\Omega B(A')$ is $k$-connected but $Map_V(X,X)$ is $k$-local. So, $F'$ is homotopic to a constant map as we wanted.

\end{proof}

\section{Local (semi)model structures on algebras} \label{sec:loc-alg}

\subsection{Disconnected and constantly disconnected functors and bimodules} 

\begin{defin}\label{tame} Let $u:A\to B$ be a functor between small categories.  We call it {\it fiberwise {disconnected}}   if it is fiberwise disconnected as a morphism of polynomial monads. \end{defin} This, of course, means that for any $b\in B$, the category $u/b$ has a discrete reflective subcategory or, equivalently, $u/b$ has a terminal object in each connected component. 
\begin{defin} And we call $u:A\to B$ \textit{constantly fiberwise disconnected} if, in addition, $u$ is $\Ww_0$-locally constant (i.e., $b\to b'$ induces an isomorphism $\pi_0(u/b)\to \pi_0(u/b')$).
\end{defin}

As before we are going to shorten our terminology and drop the word `fiberwise' if it does not lead to confusion.

\begin{proposition}\label{tame functor} Let $u:A\to B$ be a functor such that $u^{op}:A^{op}\to B^{op}$ is disconnected (i.e., for any $b\in B$ the category $b/u$ has initial object in each connected component.)
Then the restriction functor
$$u^*:[B,\VV]_{proj}\to [A,\VV]_{proj}$$
is left Quillen. Moreover, if in addition $u^{op}$ is constantly disconnected (i.e., $b\to b'$ induces an isomorphism $\pi_0(b/u)\leftarrow \pi_0(b'/u)$), then for any proper fundamental localiser $\Ww$
 $$u^*:[B,\VV]^\Ww_{proj}\to [A,\VV]^\Ww_{proj}$$
 is also left Quillen.  
\end{proposition}

\begin{proof}  Let $u_*$ be the right adjoint to $u^*$ as usual. We need to show that  $u_*$ preserves fibrations and trivial fibrations. The value of $u_*(F)(b) \ , b \in B$ is given by
$$\lim_{b/u} \tilde{F} = \prod_{c\in \pi_0(b/u)} \tilde{F}(\bar{c}),$$
where $\tilde{F}: b/u \to \VV , \ \tilde{F}(u(a)\to b) = F(a)$ and $\bar{c}$ is a chosen initial object  in the connected component $c$ of $b/a.$   
So, $u_*$ sends any natural transformation $\phi:F\to G$ to a natural transformation whose components are  products of certain components of $\phi.$  Fibrations and trivial fibrations in the projective model structure are defined componentwise. Hence $u_*(\phi)$ is a (trivial) fibration if $\phi$ is because (trivial) fibrations are closed under products.

To check this for the local model structure, it is enough to see that, under the conditions of the proposition, the functor $u^{op}$ is $\Ww_{\infty}$-locally constant (see Example \ref{inftyexample}). The result now follows from Corollary \ref{uop=lc}.  
\end{proof}

\begin{defin} \label{defn: locally constant Reedy bimodule}
Let 
\begin{equation}\label{bimod} F: B^{op}\times A\rightarrow \Set \end{equation} 
be a bimodule.
 We will call the bimodule \ref{bimod}  disconnected (constantly disconnected) if the functor $p^{op}:(\gint F)^{op}\to B^{op}$ is disconnected (constantly disconnected).   

\end{defin} 

\begin{proposition}\label{locally tame bimodule} 
For a disconnected bimodule $F:B^{op}\times A \to \Set$ the functor
$$p^*\circ \pi_!: [B,\VV]_{proj}\rightarrow [\gint F,\VV]_{proj} \rightarrow [A,\VV]_{proj}$$
 is a left Quillen 
functor. % if projective model structure exists.

If $F$ is constantly disconnected then for any proper fundamental localiser $\Ww$ the functor
$$p^*\circ \pi_!: [B,\VV]^\Ww_{proj}\rightarrow [A,\VV]_{proj}^\Ww$$
is a left Quillen functor. % if the local model structure exists.

\end{proposition}

\begin{proof} The functor $p^*$ is left Quillen by Proposition \ref{tame functor}. The functor $\pi_!$ is always left Quillen for the projective model structure and its localisation, because $\pi^*$ preserves $\Ww$-locally constant presheaves. So the composite is also left Quillen.

\end{proof}

\subsection{Left localisable substitudes} \label{left localisable-substitudes}

Let $({P},A)$ be a $\Set$-substitude. 
 Let $\Ww$ be a proper fundamental localiser and let $\VV$ be a symmetric monoidal combinatorial model category. The projective semimodel structure on  $\Alg_P(\VV)$ is, by definition,   the  semimodel structure   transferred along 
 $U_P:\Alg_P(\VV)\to [A,\VV]_{proj}$ or, equivalently,  along the composite $$\Alg_P(\VV)\stackrel{\eta^*}{\to}[A,\VV]\stackrel{i^*}{\to} [A_0,\VV].$$ 
 Assuming the existence of such a transfer  we can try to lift the localisation $[A,\VV]_{proj}\to [A,\VV]^\Ww_{proj}$ to the category of ${P}$-algebras to obtain the localisation $\Alg_{{P}}^\Ww(\VV).$ 
 
 \begin{definition} \label{wlcalgebra} A ${P}$-algebra $X$ is called $\Ww$-locally constant if its underlying presheaf $\eta^*(X):A\to \VV$ is a $\Ww$-locally constant presheaf.
 
 \end{definition}
 
 \begin{theorem}\label{lifting of localisation} If $\VV$ is a symmetric  monoidal combinatorial model category with cofibrant unit and $(P,A)$ is a $\Sigma$-free $\Set$-substitude then 
 \begin{enumerate}
 \item The projective semimodel  structure  on $\Alg_{P}(\VV)$ exists;
  \item For any proper fundamental localiser the local semimodel model structure $\Alg_{{P}}^\Ww(\VV)$ exists and its fibrant objects are exactly $\Ww$-locally constant ${P}$-algebras; 
  \item 
The local model structure $\Alg_{{P}}^\Ww(\VV)$ coincides  with the transferred semimodel structure along the forgetful functor
$U:\Alg_{{P}}(\VV)\to [A,\VV]_{proj}^\Ww$ 
provided  this transferred semimodel structure exists.   \end{enumerate} 
 \end{theorem}

\begin{proof} 
Recall from the discussion just after Definition \ref{defn:substitute-map} that a substitude is a coloured operad with additional structure. 
Since $P$ is a $\Sigma$-free substitude and the unit of $\VV$ is cofibrant, the coloured operad associated with $(P,A)$ is $\Sigma$-cofibrant. We now apply \cite[Theorem 6.3.1]{white-yau1} to transfer the projective semimodel structure to $\Alg_{P}(\VV)$ from the projective model structure on $[P_0,\VV]$.

The existence of the local semimodel structure $\Alg_{P}^\Ww(\VV)$ follows from Theorem \ref{loclifting}. 
Point (3) of the theorem follows from  Theorem 5.6 of \cite{batanin-white-eilenberg-moore}.  \end{proof}
By default we assume that  $\Alg_P(\VV)$ is equipped with the projective semimodel structure. We connect Theorem \ref{lifting of localisation} with Theorem \ref{semitransfer} below.

\begin{defin} \label{defleftloc} We will say that a $\Sigma$-free substitude $({P},A)$ is left localisable if for any symmetric monoidal combinatorial model category $\VV$ with cofibrant unit $I$ and a proper fundamental localiser $\Ww$ the local semimodel structure $\Alg_P^\Ww(\VV)$ is  transferred from  $[A,\VV]_{proj}^\Ww$ and the forgetful functor   $U_P:\Alg_P(\VV)\to [A,\VV]_{proj}$  preserves cofibrant objects.\end{defin}

 Since cofibrant objects in $[A,\VV]_{proj}$ and  $[A,\VV]^\Ww_{proj}$ coincide we have for a left localisable substitude that 
 $U_P:\Alg_P(\VV)\to [A,\VV]^\Ww_{proj}$  also preserves cofibrant objects.
    Immediately from Theorem \ref{BCperfect} we have 
\begin{corollary} \label{BC_for_W} Let $\VV$ be  symmetric monoidal combinatorial model category  with cofibrant unit and let $\Ww$ be a proper fundamental localiser. If  $(f,g):(P,A)\to (Q,B)$ is a Beck-Chevalley morphism between left localisable substitudes then: 
\begin{enumerate} \item
The morphism  $(f,g):(P,A)\to (Q,B)$ is also homotopically Beck-Chevalley with respect to the local model structures  $[A,\VV]^\Ww_{proj}$ and $[B,\VV]^\Ww_{proj};$ 
\item If $g$ is a $\Ww$-equivalence then the  adjoint pair $f_!\dashv f^*$ is a Quillen equivalence between $ \Alg_{P}^\Ww(\VV)$ and $ \Alg_{{Q}}^\Ww(\VV).$\end{enumerate}
\end{corollary}

We need a criteria for recognition of left localisable substitudes. We first establish
a sufficient combinatorial condition  for a substitude to be  left Quillen:

{ 
\begin{proposition}\label{locoperad} Let $\VV$ be  symmetric monoidal combinatorial model category  with cofibrant unit and
let  $({P},A)$ be a $\Sigma$-free  substitude. Then: \begin{enumerate} \item $({P},A)$ is left Quillen with respect to $[A,\VV]_{proj}$ if 
  the bimodule $d(P)$  is disconnected (cf. (\ref{E'})).

 \item  $({P},A)$ is left Quillen with respect to $[A,\VV]_{proj}^\Ww$ for any proper fundamental localiser $\Ww$
 if $d(P)$ is constantly disconnected.  \end{enumerate}

 \end{proposition}   }

\begin{proof}  
We can use the formula (\ref{convolution2}) for convolution.
Disconnectedness of $d(P)$  implies that $ \pi'_!(p')^*$ is a left Quillen functor with respect to the projective model structure and is left Quillen with respect to local model structure if $d(P)$ is constantly disconnected by Proposition \ref{locally tame bimodule}.
Thus to finish the proof of the first point of the proposition, it is  enough  to show that $\tilde{\otimes}^k$  is left Quillen. It follows from Barwick's results  \cite[Propositions 3.43 and 4.50]{Barwieck} 
that the functor  
$$\tilde{\otimes}^2: [A,\VV]_{proj}\times [A,\VV]_{proj}\rightarrow [A\times A,\VV]_{proj}$$
is a left Quillen functor of two variables. Applying this result and an obvious induction to our case, we complete the proof.   
 
 { To prove that $({P},A)$ is left Quillen with respect to $[A,\VV]_{proj}^\Ww$ (using  formula (\ref{convolution2}) again)} we see that it is enough to prove that the functor $\tilde{\otimes}^2$ is a  
  left Quillen functor for the {local} model structure as well. 
 By the usual adjunction argument this  comes down to the following statement of preservation of local objects:

  Let $Z:A\times A \rightarrow \VV$ be a $\Ww$-locally constant fibrant presheaf. Then for a cofibrant $Y:A\to \VV$ %and a $\Ww$-constant and fibrant $Z:Q\times Q\to \VV$ 
   the presheaf   \ $
\underline{[A,\VV]}(Y,Z(s,-)) \ , \                      s\mapsto \int_t \underline{\VV}(Y(t),Z(s,t))
$
is a $\Ww$-locally constant presheaf. Here $\underline{\VV}$ and $\underline{[A,\VV]}$   are internal hom-functors in $\VV$ and $[A,\VV]$ correspondingly.

 We know that $Z$ is equivalent to a constant presheaf  after restriction to any $\Ww$-aspherical category $B'\to A\times A.$  So, for an arbitrary $\Ww$-aspherical $A'$ and a functor 
$u:A'\to A$  there is a levelwise weak equivalence $(1_s\times u)^*(Z)\to C_s$  to a constant presheaf $C_s,$  where $1_s:1\to A$ is the functor that picks $s\in A.$ We can obviously assume that $C_s$ is fibrant. Since $Y$ is cofibrant and $Z$ is fibrant we have a levelwise weak equivalence 
  $$\int_t \underline{\VV}(Y(t),Z(s,t))\to \int_t \underline{\VV}(Y(t),C_s(t)) .$$ 
  The last functor is obviously constant.

\end{proof} 
Finally,  applying 
Theorem \ref{semitransfer} we get the following combinatorial criteria for localisability
\begin{theorem}\label{left localisable criteria} Let $(P,A)$ be a unary tame substitude for which $d(P)$ is constantly disconnected. Then $(P,A)$ is left localisable. 
\end{theorem}

\subsection{First examples and applications}

\begin{example} The classical theory of localisation in a category begins from a class of morphisms $\mathtt{A} $ in a category $C$ which we want to invert. The class $\mathtt{A}$ generates a subcategory $A$ in $C$ and we can assume that $A\hookrightarrow C$  is the identity on objects.  The localisation  $C\to C[{A}^{-1}]$ is characterised by the universal property that $[C[{A}^{-1}],\VV]$ is exactly the subcategory of presheaves $[C,\VV]$ whose restriction on any morphism from $A$ is an isomorphism. Here $\VV$ is any category. If $A=C$ we get $C[{C}^{-1}] \cong \Pi_1(C)$ the fundamental groupoid of $C$ and the `equation':
$$[C[{C}^{-1}],\VV]\cong  [\Pi_1(C),\VV] .    $$ 
Cisinski's localisation $[C,\VV]_{proj}^{\Ww_{\infty}}$ does exactly the same for the `inversion' of morphisms up to all higher homotopies. More precisely, 
there is an $\infty$-equivalence of $(\infty,1)$-categories:
$$[C,\VV]_{proj}^{\Ww_{\infty}} \to  [\Pi_{\infty}(C),\VV] $$
where $\Pi_{\infty}$ is the $\infty$-fundamental groupoid of $C$ (cf. discussion in Section 1.2 \cite{cis})

It is natural to ask if we can get a version of Cisinski's localisation when we are  interested in the weak `inversion' of morphisms from a subcategory $A$, and, more generally, with respect to a proper fundamental localiser $\Ww$. One way to achieve this is to repeat the proof of Theorem   \ref{W local structure}  by considering  functors from $\Ww$-aspherical categories $A'$ to $C$ which are factorisable through  $A$.

A quicker way is to consider the pair $(C,A)$ as a substitude and apply Theorem \ref{left localisable criteria} to this substitude. It is trivial to check that all the conditions of this theorem are satisfied in this case. 
\end{example} 

\begin{example} In this example we consider a monoidal version of Cisinski  localisation. More precisely, let $(C,\otimes,e)$ be a (strict) monoidal small category and $\alpha:A\subset C$ be an identity-on-objects monoidal inclusion. We can construct a nonsymmetric substitude $(P_C,A)$ out of this pair in a standard way. Namely,
$$P_C(a_1,\ldots,a_n;a) = C(a_1\otimes\ldots\otimes a_n, a),$$ 
$$\alpha:A(a_1,a_2)\to C(a_1,a_2).$$ 
The category of algebras of this substitude is the category ${\rm Mon}[C,\VV]$ of lax-monoidal functors $F:C\to \VV$  and their monoidal transformations. It is well known that this is also equivalent to the category of  monoids in $[C,\VV]$ with respect to Day's convolution, hence the notation. 

\begin{proposition} The symmetrisation $sym(P_C,A)$ of the nonsymmetric substitude $(P_C,A)$ is left localisable. 
\end{proposition}  

\begin{proof} We have to check that $sym(P_C,A)$ is unary tame and $d(sym(P_C))$ is constantly disconnected. In fact, $sym(P_C,A)$ is even a tame substitude. But first observe that the polynomial monads which correspond to $sym(P_C)$ and $P_C$ are equal, so we will use the same notation $\Pp_C$ for it. Moreover, the bimodule $d(sym(P_C))$ is equal to $(\epsilon\times 1)^*P_C$ in this case and the nonsymmetric convolution coincides with the convolution of $(P_C,A)$ (see Remarks \ref{d=e} and \ref{nonsym}). Thus we can ignore symmetries in the argument below.

 The classifier $\Pp_C^{\Pp_C + \Aa}$ can be easily described. Its objects are morphisms in $C$ of the form:
$$f: \sigma(X_1\otimes X_2 \otimes \dots  \otimes X_p \otimes K_1 \ldots \otimes K_q) \to X$$
 where $\sigma$ is a $(p,q)$-shuffle, plus one more object $\emptyset\to e.$ 
 A morphism $\phi: h\to f$ in this classifier is given by a sequence of objects $f_i: X^i_1\otimes\ldots\otimes X^i_{n_i} \to X_i, \ 1 \le i\le p$  in $\Pp_C^{\Pp_C + \Aa_0}$   and a sequence of morphisms  in $A$ \ $g_j:K_j\to K_j, \ 1\le j \le q$  such that 
 $h$ is equal to the composite 
 $$f(\sigma(f_1\otimes \ldots\otimes f_p\otimes g_i\otimes ,\ldots,\otimes g_q)).$$  
In this description we also assume that one of the objects $X_i$ in the domain of $f$ can be equal to the unit of $C$ and in this case $f_i$ has an empty string as its domain. 

To prove unary tameness we use Lemma \ref{a+a}. 
We show that the subcategory $\tau$ in the connected components of $\Pp_C^{\Aa + \Aa}$ which consists of  the morphisms of $C$ whose domain are alternating strings:
$$X_1\otimes K_1\otimes X_2\otimes\ldots \otimes K_n\otimes X_{n+1} \to X$$ 
is final in $\Pp_C^{\Pp_C + \Aa.}$ 
Indeed, an arbitrary morphism $h$ can be factored through a morphism of this type. For this we consider a canonical bracketing of the domain of $h$ where we put brackets on the maximal substrings of objects of $X$-type (including empty strings if necessary). We then take as $f_i$ on the $i$-th substring the identity morphism or the special morphism $\emptyset\to e.$ For example, if $h$ is 
$$X_1\otimes X_2\otimes K_1\otimes K_2 \to X$$
we factor it as 
$$X_1\otimes X_2\otimes K_1\otimes K_2 \to (X_1\otimes X_2)\otimes K_1\otimes e \otimes K_2 \otimes e \to X ,$$ 
where $f_1 = id : X_1\otimes X_2 \to X_1\otimes X_2 , f_2:\emptyset \to e$, and $f_3:\emptyset\to e.$   
Hence, the category $f/\tau$ is nonempty for all $f.$ It is not hard to see that, in fact, the object of $\tau $ constructed above is the terminal object in $f/\tau,$ hence, $\tau$ is a final subcategory.

The category $\gint P_C$ has objects the morphisms
$f:a_1\otimes \ldots \otimes a_n \to a$ in $C.$  A morphism from 
 $a_1\otimes \ldots \otimes a_n\stackrel{f}{\rightarrow} a$ to $b_1\otimes \ldots \otimes b_n\stackrel{h}{\rightarrow} b$ is given by 
  an $n$-tuple  of morphisms in $A$, $f_i: a_{i}\rightarrow b_{i}$, and $g:a\rightarrow b$ such that the following square commutes:

\begin{align} \label{diagram5}
\xymatrix{
a_{1}\otimes \ldots \otimes a_{n} \ar[r]^{\scriptstyle \hspace{5mm} f} \ar[d]_{f_1\otimes \ldots \otimes f_n} & a \ar[d]^{g} \\
 b_1\otimes \ldots \otimes b_n \ar[r]_{\hspace{5mm}\scriptstyle h} & b.
}
\end{align}

The functor $p:\gint P_C \to \MM A$ (Section \ref{subsec:convolution}) sends an object $f:a_1\otimes \ldots \otimes a_n \to a$  to the string $(a_1,\ldots, a_n).$ Now the comma-category of $p$ under $X= (a_1,\ldots,a_n)$ has an initial object given by   $u:X\to p(Y)$, where $Y$ is the object in $\gint P_C$ given by $id:a_1\otimes\ldots\otimes a_n\to a_1\otimes\ldots\otimes a_n$
and components of $u$ are also the identities  $id: a_1\to a_1,\ldots,id:a_n\to a_n.$ Notice that the requirement that $A$ is closed under tensor product operation is used exactly here. 
This completes the proof.  
\end{proof}

We thus have 
\begin{proposition} Let $\VV$ be a combinatorial symmetric monoidal model category with cofibrant unit and let $\Ww$ be a proper fundamental localiser. Let $C$ be a small monoidal category and $A\subset C$ be its monoidal subcategory with the same objects. Then there exists a local semimodel category ${\rm Mon}[C,\VV]_{proj}^\Ww[A^{-1}]$ whose fibrant objects are  lax monoidal functors with fibrant values whose restriction on $A$ are $\Ww$-locally constant. 
\end{proposition} 

We can say more in the case $A=C.$ We denote the corresponding localised model category  simply ${\rm Mon}[C,\VV]_{proj}^\Ww.$
\begin{proposition} For any strict monoidal functor $F:C\to D$ between small monoidal categories, the induced morphism of substitudes $(f,F): (P_C,C)\to (P_D,D)$ is Beck-Chevalley. \end{proposition}  
\begin{proof} We apply Theorem \ref{bc section}. We need to show that the morphism of classifiers
$$\Dd^\Cc \to \alpha^*\Pp_D^{\Pp_C}$$
is a final functor. Recall that $\alpha:\Dd\to \Pp_D$ is the morphism between polynomial monads generated by unit of the substitude $(P_D,D)$ .  

The finality of this functor is equivalent to the statement that, for any fixed morphism $\phi: F(c_1\otimes\ldots\otimes c_n)\to d$, the following category of factorisations of $\phi$ is nonempty and connected.
Objects of this category are  pairs of morphisms  $h:c_1\otimes \ldots\otimes c_n\to b, \ g: F(b) \to d$ such that $\phi $ factors as $F(h)$ followed by $g.$ A morphism from such a pair to a pair
$h':c_1\otimes \ldots\otimes c_n\to b', \ g': F(b) \to d$ is given by a morphism $t:b\to b'$  making the following diagram commute: 
 \begin{equation*}\xymatrix{
F(c_1\otimes\ldots\otimes c_n) %\ar@<2.5pt>[r]^{\psi_!} 
 \ar@{->}@<0pt>[d]_{F(h)}
\ar@<2.5pt>[r]^{\hspace{5mm} F(h')}
& 
F(b') %\ar@<0pt>[d]_{\psi}
 \ar@{->}@<0pt>[d]^{g'} \\
F(b)\ar@<2.5pt>[r]^{g}
\ar@<2.5pt>[ur]^{F(t)} %\ar@{<-}@<0pt>[d]^{\alpha} %\ar@<2.5pt>[u]^{\beta}  
& 
d }
\end{equation*} 
It is obvious that the category of factorisations has an initial object given by the pair $id:c_1\otimes \ldots\otimes c_n\to c_1\otimes \ldots\otimes c_n,  \ \phi: F(c_1\otimes \ldots\otimes c_n) \to d.$ Hence, it is nonempty and connected, as required.

\end{proof}

\begin{corollary} Let $\Ww$ be a proper fundamental localiser, let $C$ and $D$ be small monoidal categories, and let $\VV$ be a combinatorial symmetric monoidal model category with cofibrant unit. 
Any monoidal $\Ww$-equivalence $C\to D$ induces a Quillen equivalence between  ${\rm Mon}[C,\VV]_{proj}^\Ww$ and ${\rm Mon}[D,\VV]_{proj}^\Ww.$

If $(C,\otimes,e)$ is a  $W$-aspherical monoidal category then there exists a Quillen equivalence between ${\rm Mon}[C,\VV]_{proj}^\Ww$ and the category ${\rm Mon}(\VV)$ of monoids in $\VV.$ 
\end{corollary} 

\begin{proof} The unique functor $C\to 1$ is the monoidal $\Ww$-equivalence. The category ${\rm Mon}[1,\VV]_{proj}^\Ww$ is obviously just ${\rm Mon}(\VV)$ with the projective model structure. \end{proof} 
\end{example} 

\begin{remark} It is tempting to try to develop a theory of monoidal localisation for braided monoidal and symmetric monoidal functors. Unfortunately, our present technique is not enough for this purpose, since the corresponding substitude $P_C$ is not $\Sigma$-free for a braided monoidal $C.$ Conditions under which algebras still have a transferred (semi)model structure, even without $\Sigma$-freeness, may be found in \cite{white-yau1}.
\end{remark}

\part{Higher braided operads} \label{part:n-operads}

In this part, we provide our main applications of the theorems above.

\section{$n$-operads} \label{sec:n-operads}

In this section, we review the basics of $n$-operads, previously studied by the first author in \cite{SymBat, EHBat, LocBat, batanin-baez-dolan-via-semi}. Algebras over $n$-operads have the requisite structure to model $n$-tuply monoidal $(n+k)$-categories, as required for the Baez-Dolan Stabilisation Hypothesis \cite{BD}. We begin with the structure that underlies an $n$-operad.

\subsection{$n$-ordered sets, $n$-ordinals and quasibijections} \label{sec:n-ordinals}
{\em An $n$-ordered set} is a set $X$ with a given $n$-tuple $(<_0,\dots,<_{n-1})$ of nonreflexive complementary orders, meaning that any $x,y\in X$ can be compared with respect to exactly one of the orderings $<_0,\dots,<_{n-1}$  \cite[Definition 2.3]{SymBat}. An $n$-ordered set is {\em totally $n$-ordered} if $i <_p j$ and $j <_r k$ implies $i <_{\min(p,r)} k$ \cite[Definition 2.4]{SymBat}. 

 A structure of a totally $n$-ordered set on $X$ induces the following linear order $<_X$ on $X$ called the total order:  $i<_X j$ if and only if there exists $0\le p\le n-1$ such that $i<_p j.$ We denote this linearly ordered set by $[X].$   Given two $n$-ordered sets, $X$ and $Y$, with the same underlying set, we say $X$ {\em dominates} $Y$ if $i <_p j$ in $X$ implies either $i<_r j$ in $Y$ for some $r\geq p$ or $j <_r i$ in $Y$ for some $r > p$ \cite[Definition 2.6]{SymBat}.
 
 For $k\ge 0$ we denote by $\bar{k}$ the linearly ordered  finite set $ 1<\ldots<k.$    {\it The Milgram poset $\JJ_n(k)$} is the poset of total complementary $n$-orders on $\{1,\dots,k\}$ with respect to the domination relation. 
 
 \begin{defin} Let $n\ge 1.$ An $n$-ordinal is a totally $n$-ordered finite set $T$ such that the linearly ordered set $|T|$ coincides with one of the finite ordinals $\bar{k}.$
 \end{defin}

\begin{remark} It is also possible to consider a one point set as the only $0$-ordinal. This plays a role in the general theory of $n$-operads as had been developed in  the original paper of the first author on this subject \cite{BatM} but in this paper we only concern ourselves with $n$-operads for $n>0.$
 
 \end{remark}

 Every $n$-ordinal can be represented as a pruned planar tree with $n$ levels (pruned $n$-tree) or as an $n$-dimensional globular graph (see \cite{SymBat} for a discussion). The empty $n$-ordinal is represented by the only  degenerate pruned $n$-tree $z^n U_0$ which consists of only a root on the level $0.$ The terminal $n$-ordinal is represented by a linear  tree $U_n$ (or just an $n$-globe in globular notation).  
   
 \begin{defin}\label{mapofordinals}  \  A morphism of $n$-ordered sets
 $$\sigma: T \rightarrow S$$   is a map $\sigma:T\rightarrow S$ of underlying sets such that  $$i<_p j \ \mbox{in} \ T $$   implies that
 \begin{enumerate}
\item \ $\sigma(i) <_r \sigma(j)$   for some $r\ge p$ or
\item \  $\sigma(i)= \sigma(j)$ or
\item \  $\sigma(j) <_r \sigma(i)$ for $r>p .$
\end{enumerate}
 \end{defin}
$n$-ordered sets and their morphisms as above form a category.
The following lemma is obvious but useful to remember.
\begin{lemma} The Milgram poset $\JJ_n(k)$ is isomorphic to the subcategory of the category of $n$-ordered sets whose objects are 
total complementary $n$-orders on $\{1,\dots,k\}$ and whose morphisms are morphisms of $n$-ordered sets which are identities on the underlying set.
\end{lemma} 
\begin{remark} The homotopy type of $N(\JJ_n(k))$ coincides with the homotopy type of the ordered configuration space of $k$ points in $\mathbb{R}^n$ (see the end of the proof of \cite[Theorem 5.1]{LocBat}).     
\end{remark}

{\it The category of $n$-ordinals  $Ord(n)$} is the (skeletal) subcategory of the category of $n$-ordered sets spanned by $n$-ordinals.
Notice that this category has pullbacks along the morphisms from the terminal $n$-ordinal (such a morphism is determined completely by a choice of an element $i\in S$). In plain words, given a morphism of $n$-ordinals $\sigma:T\to S$ and  $i\in S$ we just consider the preimage $\sigma^{-1}(i)$  in the underlying set of $T$ which  acquires  a natural structure of an $n$-ordinal from $T.$
We call it {\it the fiber of $\sigma$ over $i$}.

The total order provides us with a functor $$[-]:Ord(n)\rightarrow \FinSet,$$ 
where $\FinSet$ is  the skeletal category of finite sets whose objects are finite ordinals $\bar{k}, k\ge 0$ and whose morphisms are arbitrary morphisms of underlying sets. 

All the notions above can be easily redefined for $n=\infty.$ We thus have  
the category of $\infty$-ordinals $Ord(\infty)$ together with a total order functor
$$[-]: Ord(\infty)\rightarrow \FinSet .$$

\begin{defin} A map of  $n$-ordinals is called a quasibijection if it is a bijection of the underlying sets.\end{defin}

Let $\QQ_n$ (for $1\le n\le\infty$) be {\it the subcategory of quasibijections} of $Ord(n) $ and let $\Sm$ be the groupoid of invertible morphisms in $\FinSet$ which is isomorphic to the groupoid of symmetric groups. The total order functor induces  a functor which we will denote by the same symbol: 
$$[-]:\QQ_n\rightarrow \Sm .$$
For $n=2$ it was shown in \cite{LocBat} that $[-]$ factors through the groupoid of braid groups $\Br.$

It is clear that the category $\QQ_n$ is the union of connected
 components $\QQ_n(k)$ where $k$ is the cardinality of the $n$-ordinals.  

\begin{theorem}\label{Milgram}
\begin{enumerate}\item For a finite $n$, the nerve  $N(\QQ_n(k))$    has the homotopy type of the unordered configuration space of $k$-points in $\mathbb{R}^n ;$ 
\
\item The functor
$$[-]:\QQ_{\infty} \rightarrow  \Sm,$$
induces a weak equivalence of nerves;
 \item The fiber of the  functor $[-]: \QQ_n\rightarrow \Sm$ over an object $\bar{k}\in \Sm$ is isomorphic to the Milgram poset $\JJ_n(k)$. 
 \item  For $n\ge 3$ the   functor $[-]: \QQ_n\rightarrow \Sm$ is $\Ww_k$-aspherical  for  $0\le k\le n-1.$ 
 \item For $n=2$ the functor $[-]_2:\QQ_2\to \Br$  is a $k$-equivalence for $1\le k\le \infty.$

\end{enumerate}
\end{theorem}

\begin{proof}
Most of the statements of this theorem are just reformulations of statements of Theorem 5.1 and Lemma 5.1 from \cite{LocBat}. 
We  add the proof of the point (3). Asphericity of $[-]$ follows immediately.

It is not hard to see that the group $\Sm_k$ acts  on $\JJ_n(k)$ and the quotient $\JJ_n(k)/\Sm_k$ is isomorphic to $\QQ_n(k).$  
One can think of an element from $\JJ_n(k)$ as a pair $(T,\pi)$ where $T$ is an $n$-ordinal and $\pi$ is a permutation from $\Sm_k $ and  
 $(T,\pi) > (S,\xi)$ in $ \JJ_n(k)$ when there exists a quasibijection $\sigma:T\rightarrow S$ and $\xi\cdot\pi = \sigma .$ 
 
 One can identify then the category $\JJ_n(k)$ with the comma-category $[-]: \QQ_n(k) \to \Sm(k).$ Recall that we consider the group $\Sm(k)$ as a groupoid with a unique object $\bar{k}.$ Then an object of the comma-category $|-|/\bar{k}$  is exactly  $T\in \QQ_n(k)$ equipped with a permutation $\pi:|T|\to \bar{k}$ that is an object of $\JJ_n(k).$ Morphisms  also coincide with the description above.  
   
\end{proof}

\subsection{$n$-operads}%\label{sec:n-operads}

We now recall  the definition of pruned $(n-1)$-terminal  $n$-operad \cite{SymBat}. Since we do not need other types of
$n$-operads in this paper we will call them simply $n$-operads. 
The notation $U_n$ means the terminal $n$-ordinal.  

Let $\VV$ be a symmetric monoidal category. For a morphism of $n$-ordinals $\sigma:T\rightarrow S$ the $n$-ordinal $T_i$ is the fiber $\sigma^{-1}(i) .$ 
\begin{defin}\label{defnoper}  An $n$-operad in $\VV$ is 
a collection $A_T, \ T\in Ord(n)$ of objects of $\VV$ equipped with the following structure :

- a morphism $e: I \rightarrow  A_{U_n}$ (the unit);

- for every morphism $\sigma:T \rightarrow S$ in $Ord(n) ,$ 
a morphism 
$$m_{\sigma}:A_S\otimes A_{T_0}\otimes ... \otimes A_{T_k}
 \rightarrow A_T \mbox{\ \ (the multiplication}).$$

They must satisfy the following identities:

- for any composite $$T\stackrel{\sigma}{\rightarrow} S \stackrel{\omega}{\rightarrow} R ,$$
the associativity diagram

{\unitlength=1mm

\begin{picture}(300,45)(2,0)

\put(20,35){\makebox(0,0){\mbox{$\scriptstyle A_R\otimes
A_{S_{\bullet}}\otimes A_{T_0^{\bullet}} \otimes  ...
\otimes 
 A_{T_i^{\bullet}}\otimes  ... \otimes A_{T_k^{\bullet}}   
$}}}
\put(20,31){\vector(0,-1){12}}

\put(94,31){\vector(0,-1){12}}

\put(88,35){\makebox(0,0){\mbox{$\scriptstyle A_R\otimes
A_{S_{0}}\otimes A_{T_1^{\bullet}} \otimes  ...
\otimes A_{S_{i}}\otimes
 A_{T_i^{\bullet}}\otimes  ... \otimes A_{S_{k}}\otimes
A_{T_k^{\bullet}}   
$ }}}

\put(50,35){\makebox(0,0){\mbox{$\scriptstyle \simeq $}}}

\put(20,15){\makebox(0,0){\mbox{$\scriptstyle A_S\otimes 
A_{T_1^{\bullet}} \otimes  ...
\otimes 
 A_{T_i^{\bullet}}\otimes  ... \otimes A_{T_k^{\bullet}}
$}}}

\put(94,15){\makebox(0,0){\mbox{$\scriptstyle A_R\otimes 
A_{T_{\bullet}} 
$}}}

\put(60,5){\makebox(0,0){\mbox{$ \scriptstyle A_T 
$}}}

\put(35,11){\vector(4,-1){19}}

\put(85,11){\vector(-4,-1){19}}

\end{picture}}

\noindent commutes,
where $$A_{S_{\bullet}}= A_{S_0}\otimes ...
\otimes A_{S_k},$$  
$$A_{T_{i}^{\bullet}} = A_{T_i^0} \otimes ...\otimes A_{T_i^{m_i}}$$
and $$ A_{T_{\bullet} } =  A_{T_0}\otimes ...
\otimes A_{T_k};$$

- for an identity $\sigma = id : T\rightarrow T$ the diagram

{\unitlength=1mm
\begin{picture}(50,25)(30,2)

\put(97,20){\vector(-1,0){20}}

\put(60,17){\vector(0,-1){8}}

\put(60,20){\makebox(0,0){\mbox{\small$A_T\otimes 
A_{U_n}\otimes ... \otimes A_{U_n} 
$}}}

\put(114,20){\makebox(0,0){\mbox{\small$A_T\otimes 
{I}\otimes ... \otimes {I} 
$}}}

\put(60,5){\makebox(0,0){\mbox{\small$A_T 
$}}}

\put(105,15){\vector(-4,-1){30}}

\put(90,9){\makebox(0,0){\mbox{\small$id
$}}}

\end{picture}}

\noindent commutes;

- for the unique morphism $T\rightarrow U_n$ the diagram

{\unitlength=1mm
\begin{picture}(50,25)(30,2)

\put(87,20){\vector(-1,0){15}}

\put(60,17){\vector(0,-1){8}}

\put(60,20){\makebox(0,0){\mbox{\small$A_{U_n}\otimes 
A_T
$}}}

\put(98,20){\makebox(0,0){\mbox{\small$I \otimes
A_T
$}}}

\put(60,5){\makebox(0,0){\mbox{\small$A_T 
$}}}

\put(95,17){\vector(-3,-1){25}}

\put(84,11){\makebox(0,0){\mbox{\small$id
$}}}

\end{picture}}

\noindent commutes.

\end{defin}
 
 We will denote by $O_n(\VV)$ the category of  $n$-operads.
 
 A {\it constant-free $n$-operad} is defined in a way  similar to an $n$-operad but we do not include the object $A_{z^nU_0}$ in the definition.  We also require  the maps of $n$-ordinals used in this definition to be  surjections.
 Analogously we define  a {\it constant-free symmetric operad} as a classical symmetric operad without a $0$-term. 
  The  category  of  constant-free $n$-operads is denoted   $CFO_n(\VV).$ 
  
  \begin{remark} Of course, the category of constant-free $n$-operads is equivalent to the full subcategory of $n$-operads for which $A_{z^nU_0}$ is the initial object. Similarly for constant-free symmetric operads. We give a slightly different definition to emphasise the role of the surjective maps of $n$-ordinals. It will be important for our proof of localisability of the corresponding substitudes. 
  
  \end{remark}

 Finally, a constant-free $n$-operad (constant-free symmetric operad) is called {\it normal} if $A_{U_n}=I$ (resp. $A_1 = I$) is the unit of the category $\VV.$ The  subcategory   of  normal $n$-operads is denoted  by $NO_n(\VV).$ 
  
 To shorten our exposition we accept the following agreement.
\begin{agreement}\label{ag}  Any $n$-operads or symmetric operads we will use may have three types: general, constant free or normal. We assume throughout the text that the  type of operads is fixed and  we will use the notations $Op_n(\VV)$ ($SO(\VV), BO(\VV)$)  for the category of $n$-operads (symmetric operads and braided operads correspondingly) of this fixed type unless we specifically indicate what type of operads we have in mind.  \end{agreement}

\subsection{Symmetrisation of $n$-operads}

Let $\sigma:T\rightarrow S$ be a quasibijection and $A$ be an $n$-operad.
Since a fiber of $\sigma$ is the terminal $n$-ordinal $U_n ,$  the multiplication
$$\mu_{\sigma}: A_{S}\otimes(A_{U_n}\otimes ... \otimes
A_{U_n})\longrightarrow A_{T}$$
in composition with the morphism
$$ A_S\rightarrow  A_{S}\otimes(I\otimes ... \otimes
I) \rightarrow  A_{S}\otimes(A_{U_n}\otimes ... \otimes
A_{U_n})$$ induces a morphism
$$ A(\sigma): A_S\rightarrow A_T.$$
It is not hard to see  that in this way $A$ becomes a  functor on  $\QQ^{op}_n .$ 
 So we have  a forgetful functor from the category of  $n$-operads $Op_n(\VV)$ to the category  $[\QQ_n^{op},\VV].$

The desymmetrisation functor $des_n$ from symmetric operads to $n$-operads, for finite $n$, was defined in \cite{EHBat} using pullback along the functor $[-]:Ord(n)\rightarrow \FinSet .$  It was shown that this functor has a left adjoint which we call symmetrisation and denote $sym_n.$

  Since $n$-operads are algebras of a $\Sigma$-free coloured operad \cite{batanin-berger} whose underlying category is the opposite to the category of quasibijections, one considers a $\Sigma$-free  substitude
$(\ON,\QQ_n^{op})$ whose algebras in $\VV$ are $n$-operads of a given type (see Section \ref{sec:examples}). We also consider a $\Sigma$-free symmetric substitude $(SO,\Sm^{op})$ whose algebras are symmetric operads (we again use Agreement \ref{ag} about different types of operads). Similarly we get a substitude $(BO,\Br^{op})$ %and localisations $BO^\Ww(\VV)$ and $[\Br^{op},\VV]^\Ww$
 for braided operads and braided collections.
\begin{remark} To construct the substitude $(BO,\Br^{op})$ we need to use {\it vines} in $\mathbb{R}^3$ \cite{lavers} instead of planar trees   but otherwise the construction is similar to $({SO},\Sm^{op}).$  \end{remark}
The adjoint pair of symmetrisation and desymmetrisation is then induced by a substitude map:
$$(\tau, [-]^{op}):(\ON,\QQ_n^{op})\to ({SO},\Sm^{op}).$$ 
The first component $\tau$ sends a tree decorated by $n$-ordinals to its underlying tree. 
The desymmetrisation and symmetrisation functors are  $\tau^*$ and $\tau_!$ correspondingly \cite{batanin-baez-dolan-via-semi}.
To shorten the notation we will denote the functor $[-]^{op}$ by $[-].$ Since $\Sm$ is a groupoid, this should not lead to any confusion.  

The adjunction between braided and symmetric operads admits a similar treatment. Namely, the canonical map $\pi$ from braid groups to symmetric groups can be extended to a map of substitudes:
$$ (\theta, \pi^{op}):({BO},\Br^{op})\to ({SO},\Sm^{op}).$$
and $(\tau, [-]^{op}):(O^{(2)},\QQ_2^{op})\to ({SO},\Sm^{op})$ factors as 
$$   (O^{(2)},\QQ_2^{op})\stackrel{(\kappa, [-]_2^{op})}{\longrightarrow}    ({BO},\Br^{op})\stackrel{(\theta, \pi^{op})}{\longrightarrow} ({SO},\Sm^{op}).$$

\begin{proposition}  \label{operadicstabilization}
The morphisms of substitudes $(\tau, [-]^{op}),  (\kappa, [-]_2^{op})$ and      $(\theta, \pi^{op})$ are Beck-Chevalley morphisms.

\end{proposition}

\begin{proof}

We have to prove that 
the following commutative square of adjunctions is Beck-Chevalley.

\begin{equation}\label{BC} \xygraph{!{0;(2.5,0):(0,.5)::}
{Op_n(\VV)}="p0" [r] {SO(\VV)}="p1" [d] {[\Sm^{op},\VV]}="p2" [l] {[\QQ^{op}_{n},\VV]}="p3"
"p0":@<-1ex>@{<-}"p1"_-{des_n}|-{}="cp":@<1ex>"p2"^-{U}|-{}="ut":@<1ex>"p3"^-{[-]^*}|-{}="c":@<-1ex>@{<-}"p0"_-{U_n}|-{}="us"
"p0":@<1ex>"p1"^-{sym_n}|-{}="dp":@<-1ex>@{<-}"p2"_-{F}|-{}="ft":@<-1ex>@{<-}"p3"_-{[-]_!}|-{}="d":@<1ex>"p0"^-{F_n}|-{}="fs"
"dp":@{}"cp"|-{\perp} "d":@{}"c"|-{\perp} "fs":@{}"us"|-{\dashv} "ft":@{}"ut"|-{\dashv}}
\end{equation}

According to Theorem \ref{BCB} we need to establish that the commutative square of polynomial monads
\begin{equation}\label{exact1}\xymatrix{
 \On %\ar@<2.5pt>[r]^{\psi_!} 
 \ar@{<-}@<0pt>[d]_{\beta}
\ar@<0pt>[r]^{\tau}
& 
\SO %\ar@<0pt>[d]_{\psi}
 \ar@{<-}@<0pt>[d]^{\alpha} \\
\Qn \ar@<2.5pt>[r]^{[-]^{op}}
%\ar@<2.5pt>[ur]^{\gamma} %\ar@{<-}@<0pt>[d]^{\alpha} %\ar@<2.5pt>[u]^{\beta}  
& 
\SM }
\end{equation} 
is exact. That is that the
 induced morphism of classifiers
\begin{equation} {\SM}^{\Qn}\to \alpha^*(\SO^{\On}) \end{equation} 
is a final functor.

 The classifier $\SO^{\On}$ has been described in \cite{SymBat}. The classifier  ${\SM}^{\Qn}$ is isomorphic to $\JJ_n^{op} =  \coprod_k (\JJ_n(k))^{op}$  as was shown in the proof of Theorem \ref{Milgram} and Example \ref{coma as classifier}.  
The finality of the  inclusion $\JJ_n^{op} \subset \SO^{\On}$ is the content of Lemma 4.3 from \cite{SymBat}. 

For the morphism $(\theta, \pi^{op})$ we have the following square of adjunctions:

\begin{equation*}\label{BrC} \xygraph{!{0;(2.5,0):(0,.5)::}
{BO(\VV)}="p0" [r] {SO(\VV)}="p1" [d] {[\Sm^{op},\VV]}="p2" [l] {[\Br^{op},\VV]}="p3"
"p0":@<-1ex>@{<-}"p1"_-{\theta^*}|-{}="cp":@<1ex>"p2"^-{U}|-{}="ut":@<1ex>"p3"^-{\pi^*}|-{}="c":@<-1ex>@{<-}"p0"_-{U_b}|-{}="us"
"p0":@<1ex>"p1"^-{\theta_!}|-{}="dp":@<-1ex>@{<-}"p2"_-{F}|-{}="ft":@<-1ex>@{<-}"p3"_-{\pi_!}|-{}="d":@<1ex>"p0"^-{F_b}|-{}="fs"
"dp":@{}"cp"|-{\perp} "d":@{}"c"|-{\perp} "fs":@{}"us"|-{\dashv} "ft":@{}"ut"|-{\dashv}}
\end{equation*}

The corresponding Beck-Chevalley map is $\mathbf{bc}:\pi_!U_b\to U \theta_!.$ Now it is well known that $\theta_!$ on underlying collections is given by the quotient with respect to the action of pure braid groups. This again can be proved by the general theory of classifiers and algebraic Kan extensions but it is an elementary fact which does not require any deep theory. Obviously $\pi_!$ is also given by such a quotient. This means that  $\mathbf{bc}$ is an isomorphism.

Finally, the morphism $(\kappa, [-]_2^{op})$ is Beck-Chevalley by Proposition \ref{2of3}.

\end{proof}

\section{Locally constant $n$-operads}\label{sec:locally constant n operads} 
 
In this section we assume that $\Ww$ is a proper fundamental localiser and $\VV$ is a combinatorial monoidal model category with cofibrant unit.

\subsection{Locally constant $n$-operads, recollection}

\begin{defin}  A $\Ww$-locally constant $n$-operad in $\VV$ is an $n$-operad $A$ in $\VV$ such that its underlying $\QQ_n^{op}$-presheaf is  $\Ww$-locally constant.\end{defin}

When $\Ww=\Ww_{\infty}$ this definition coincides with the definition of locally constant $n$-operad from \cite{LocBat}. 

A morphism of $n$-operads  is a weak equivalence if it is a termwise weak equivalence of the collections. The homotopy category of operads is the category of operads localised with respect to the class of weak equivalences.
Let  $LCO^W_n(\VV)$ be  the full subcategory of the homotopy category of $O_n(\VV)$ of  $\Ww$-locally constant $n$-operads. For $\Ww=\Ww_{\infty}$ we will refer to the category $LCO^{\Ww_{\infty}}_n(\VV)$ as the homotopy category of locally constant operads to maintain the terminology of \cite{LocBat}.

For $n=1 $ the category $LCO^{\Ww_{\infty}}_1(\VV)$  is isomorphic to the homotopy category of nonsymmetric operads, and the (derived) symmetrisation functor is given by multiplication on symmetric groups. Recall \cite{LocBat} that an $n$-operad $A$ is called {\it quasisymmetric} 
if for any quasibijection $\sigma:T\to S$ between $n$-ordinals its effect $A(\sigma):A(S)\to A(T)$ is an isomorphism. 
 
The following theorem  combines Theorem 7.1 and 7.2 of \cite{LocBat}.

\begin{theorem}[\cite{LocBat}]\label{lc=2} 
   \begin{enumerate}
\item The categories of quasisymmetric $2$-operads and braided $2$-operads are equivalent.

\item For any $n\ge 3$ the categories of quasisymmetric $n$-operads  and symmetric operads are equivalent. 

\item The homotopy categories of locally constant $2$-operads, quasisymmetric $2$-operads and  braided operads are equivalent. 

\item The homotopy categories of locally constant $\infty$-operads,  quasisymmetric $\infty$-operads, and  symmetric operads are equivalent.

\end{enumerate}
\end{theorem}

\subsection{Model theoretical refinement}
The purpose of this section is to lift Theorem \ref{lc=2} to the model categorical level.

 To further simplify notations  we write  $[\QQ^{op}_n,\VV]$ for the model category $[\QQ^{op}_n,\VV]_{proj}.$  We use the notation $Op_{n}^{\Ww}(\VV) = \Alg_{O^{(n)}}^\Ww(\VV)   $ for the category of $n$-operads  with the local semimodel structure lifted from $[\QQ_n^{op},\VV]^\Ww$ (see Theorem \ref{lifting of localisation}(2)). Notice that we do not assume that this semimodel structure is transferred from $[\QQ_n^{op},\VV]^\Ww$.  The fibrant objects in $Op_{n}^{\Ww}(\VV)$ are  termwise fibrant $\Ww$-locally constant $n$-operads.  

We use a similar construction for symmetric operads.  We define the $\Ww$-local semimodel category of symmetric operads $SO^{\Ww}(\VV)$ (braided operads $BO^{\Ww}(\VV)$) as a lifting along  the forgetful functor $U: SO(\VV)\rightarrow [\Sm^{op},\VV] $ ( $BO(\VV)\rightarrow [\Br^{op},\VV] $) of the   localisation  $[\Sm^{op},\VV]\rightarrow [\Sm^{op},\VV]^\Ww$ ($[\Br^{op},\VV]\rightarrow [\Br^{op},\VV]^\Ww.$ 

We then have:
\begin{proposition} \label{lc=3} Let $\VV$ be  symmetric monoidal combinatorial model category  with cofibrant unit. Then:
\begin{enumerate}\item
The homotopy category $Ho(Op_{n}^{\Ww}(\VV))$ is equivalent to the homotopy category of $\Ww$-locally constant $n$-operads $LCO_n^\Ww(\VV).$
In particular, the category $Ho(Op_{n}^{\Ww_{\infty}}(\VV))$ is equivalent to the category of locally constant $n$-operads.
\item 
The homotopy category $Ho(SO^{\Ww}(\VV))$ (resp. $Ho(BO^{\Ww}(\VV))$) is equivalent to the homotopy category of the category of symmetric (braided) operads whose underlying symmetric (braided) collection is fibrant and $\Ww$-locally constant. 

\item For $k\ge 1$ the categories   $SO^{\Ww_k}(\VV)$ and $[\Sm^{op},\VV]^{\Ww_k}$ (resp. $BO^{\Ww_k}(\VV)$ and $[\Br^{op},\VV]^{\Ww_k}$)      are isomorphic to $SO(\VV)$ and $[\Sm^{op},\VV]$ ($BO(\VV)$ and $[\Br^{op},\VV]$)  correspondingly as (semi)model categories. 

\item The category $Ho(SO(\VV)^{\Ww_0})$ (resp. $Ho(BO(\VV)^{\Ww_0})$) is equivalent to the homotopy category of operads whose underlying symmetric (braided) collection has a homotopically trivial action of symmetric (braid) groups (meaning that such a collection is equivalent to a constant collection).

\end{enumerate}

 \end{proposition}

\begin{proof} We only need to prove the  statements (3) and (4). Since $\Sm^{op}$ is a groupoid, then any  functor on it is $\Ww_{\infty}$-locally constant. The statement then amounts to the following general fact. Let $G$ be a groupoid. Then any functor $G\to \VV$ is $\Ww_k$-locally constant for any $k\ge 1.$ Indeed, let $A'$ be a $\Ww_k$-aspherical category and $u:A'\to G$ be a functor. Then $u$ factors through the fundamental groupoid of $A':$
$$A'\to \Pi_1(A')\to G$$
and, since $A'$ has trivial $\pi_0$ and $\pi_1$, the fundamental groupoid $\Pi_1(A')$ is trivial as well. So, $u^*(F)$ is equivalent to a constant functor. 

For the third statement observe that $\Sm_n$ is a connected groupoid for all $n\ge 0.$ Therefore, every $\Ww_0$-locally constant presheaf on it is weakly equivalent to a constant presheaf.   
\end{proof}

From general properties of localisation we get

\begin{proposition} Let $\VV$ be  symmetric monoidal combinatorial model category  with cofibrant unit. Then:
\label{6.3} \begin{enumerate} 
\item The  Quillen adjunction $sym_n: Op_n(\VV) \rightarrow SOp(\VV):des_n$ factors through the $\Ww_k$-local model structure for $1\le k\ge \infty:$

 {\unitlength=1mm
\begin{picture}(200,30)(-5,5)
\put(20,25){\makebox(0,0){\mbox{$Op_n(\VV)$}}}
%\put(24,25){\vector(1,0){13}}
\put(24,21){\vector(2,-1){16}}
\put(35,13){\vector(-2,1){16}}

\put(51,13){\vector(2,1){16}}
\put(63,21){\vector(-2,-1){16}}

\put(43,25){\makebox(0,0){\mbox{$ $}}}
%\put(43,22){\vector(0,-1){7}}
\put(43,10){\makebox(0,0){\mbox{$Op_{n}^{\Ww_k}(\VV)$}}}
\put(58,26){\vector(-1,0){26}}
\put(32,24){\vector(1,0){26}}
\put(40,21){\shortstack{\mbox{$sym_n $}}}
\put(22,14){\shortstack{\mbox{$id$}}}
\put(32,17){\shortstack{\mbox{$id$}}}

\put(38,17){\shortstack{\mbox{$ $}}}

\put(60,15){\shortstack{\mbox{$sym_{\scriptscriptstyle n}^{\scriptscriptstyle } $}}}
\put(44,16){\shortstack{\mbox{$des_{\scriptscriptstyle n}^{\scriptscriptstyle } $}}}

\put(66,25){\makebox(0,0){\mbox{$SOp(\VV)$}}}
\put(40,27){\shortstack{\mbox{$des_n $}}}
\end{picture}}

\item \label{BSsquare} The following square is a square of Quillen adjunctions for $k\ge 1:$

\begin{equation}\label{BCW} \xygraph{!{0;(2.5,0):(0,.5)::}
{Op^{\Ww_k}_n(\VV)}="p0" [r] {SO(\VV)}="p1" [d] {[\Sm^{op},\VV]}="p2" [l] {[\QQ^{op}_{n},\VV]^{\Ww_k}}="p3"
"p0":@<-1ex>@{<-}"p1"_-{des_n}|-{}="cp":@<1ex>"p2"^-{U}|-{}="ut":@<1ex>"p3"^-{[-]^*}|-{}="c":@<-1ex>@{<-}"p0"_-{U_n}|-{}="us"
"p0":@<1ex>"p1"^-{sym_n}|-{}="dp":@<-1ex>@{<-}"p2"_-{F}|-{}="ft":@<-1ex>@{<-}"p3"_-{[-]_!}|-{}="d":@<1ex>"p0"^-{F_n}|-{}="fs"
"dp":@{}"cp"|-{\perp} "d":@{}"c"|-{\perp} "fs":@{}"us"|-{\dashv} "ft":@{}"ut"|-{\dashv}}
\end{equation}

\item For $n=2$ there are similar statements when we again replace symmetric groups on braid groups, symmetric operads on braided operads and corresponding symmetrisations.
\end{enumerate}

\end{proposition}

We now have the following refinement of Theorem \ref{lc=2}. 

\begin{theorem}\label{model=2} Let $\VV$ be  symmetric monoidal combinatorial model category  with cofibrant unit. Then:
   \begin{enumerate}
\item The braided symmetrisation
$$ bsym_2: Op_2^{\Ww_{\infty}}(\VV) \to  BO(\VV)$$
is  a left Quillen equivalence;

\item Similarly the symmetrisation 
$$ sym_{\infty}: Op_{\infty}^{\Ww_{\infty}}(\VV) \to  SO(\VV)$$
is a left Quillen equivalence. 

\end{enumerate}
\end{theorem}

\begin{proof} This follows from Propositions \ref{lc=3},~\ref{6.3} and Theorem \ref{lc=2}.
\end{proof}

\begin{remark} In the next two sections we will give a proof of a generalisation of this theorem  independent of Theorem \ref{lc=2}.  
\end{remark}

\section{Localisability of substitudes for operads} \label{sec:examples}

The purpose of this section is to establish the following theorem:

\begin{theorem} \label{thm: operads for n-Ops are Reedy locally constant}
The substitudes $(\ON,{Q}^{op}),({BO},\Br^{op})$ and $({SO},\Sm^{op})$ are left localisable.

\end{theorem}

The proofs of localisability of  these  substitudes are very similar to each other. We provide a detailed proof
for the case of the substitude $(NO^{(n)}, \QQ_n^{op})$ whose algebras  are normalised $n$-operads. We then explain the changes necessary for this  proof to work for general $n$-operads. The proofs for $({BO},\Br^{op})$ and $({SO},\Sm^{op})$ are, in fact, simpler and we leave them as an exercise for the reader. 
 But we first need some preparation. 
 
\subsection{Polynomial monad for $n$-operads}\label{polyforn}
A detailed description of polynomial monad $\NE$ can be found in \cite[Section 12]{batanin-berger} but we do need to remind the reader of some points from there. 
We also discuss how to identify the underlying category of this monad with $\QQ_n^{op}.$

 We can describe the set of operations of   $\NE$ in terms of $n$-planar  trees \cite{batanin-berger}.   Let  $S$ be an $n$-ordinal and $Trees^{n,k}_S$ denote the set of { labeled, decorated, reduced planar trees with $k$ vertices} (called reduced $n$-planar trees in \cite{batanin-berger}) dominated by $S.$  
 
 We now recall the necessary definitions.
An  $n$-planar tree  consists of a planar tree $\tau$  with $k$ vertices  equipped with: 
\begin{itemize} 
\item  a structure of a totally  $n$-ordered set  $T_v$ (decoration) on the set of incoming edges of every vertex $v$ such that the total linear order generated by $T_v$  coincides with the order coming from the planar structure of $\tau;$
\item  a  labeling of the set of its leaves that is a bijection $\rho_{\tau}: |S|\rightarrow L(\tau)$ between the set of leaves of $\tau$ and the underlying set of the ordinal $S.$
\item a linear order on the set of vertices of $\tau.$
\end{itemize}

An $n$-planar tree is \textit{reduced} if in $\tau$ each vertex has  at least two incoming edges. 

According to \cite{SymBat} any $n$-planar tree $\tau$ determines an $n$-complementary relation on the set $|S|$ in the following way. Let $w$ be a vertex or a leaf of $\tau$ and $v$ be a vertex of $\tau.$ We will say that $w$ is {\it above}  $v$ if there exists a path in $\tau$ from $w$ to $v$ which does not contain 
two consecutive input edges of the same vertex.  For  any two leaves or vertices  there exists a unique vertex $v(k,l)$ that is below  $k,l$ and such that any other vertex below to $k$ and $l$ is below $v.$ 

An $n$-complementary  relation $\tau(S)$ on $|S|$ generated by $\tau$ is constructed as follows. 
For $p,q\in |S|$ let $k,l$ be the corresponding leaves on $\tau$ (using $\rho_\tau$). Let $e_p$ be the input edge in $v(k,l)$, i.e., the last edge in the path from $k$ to $v(k,l).$ Analogously let $e_q$ be the input edge in $v(k,l)$ which is the last edge in the path from $l$ to $v(k,l).$ Let $T(k,l)$ be the $n$-ordinal decorating $v(k,l).$ By definition, $p<_r q$ in $\tau(S)$ if  $e_p<_r e_l$ in $T(k,l).$

A reduced $n$-planar tree $\tau$ belongs to $Trees^{n,k}_S$  if it satisfies the   following condition:
\begin{itemize}
\item[($\triangle$)] $S$ dominates the complementary relation  $\tau(S).$ 

\end{itemize}

{\unitlength=1mm

\begin{picture}(60,50)(-20,0)

\begin{picture}(10,10)(0,0)
\put(14,22.5){\circle{5}}
\put(7,30.5){\makebox(0,0){\mbox{$\scriptstyle  10$}}}
\put(21,30.5){\makebox(0,0){\mbox{$\scriptstyle  1$}}}
\put(14,32){\makebox(0,0){\mbox{$\scriptstyle  7$}}}
\put(15.9,20.8){\line(1,-1){5}}
\put(12.1,24.4){\line(-1,1){5}}
\put(14,25){\line(0,1){6}}
\put(15.9,24.4){\line(1,1){5}}

\put(14,22.5){\makebox(0,0){\mbox{$\scriptstyle T_2 $}}}
\end{picture}

\begin{picture}(10,10)(2.5,8.5)
\put(14,22.5){\circle{5}}

\put(14,25){\makebox(0,0){\mbox{$ $}}}
\put(15.9,20.8){\line(1,-1){5}}
\put(12.1,24.4){\line(-1,1){5}}
\put(14,25){\line(0,1){6}}
\put(15.9,24.4){\line(1,1){5}}
\put(14,32){\makebox(0,0){\mbox{$\scriptstyle  12$}}}
\put(21,30.5){\makebox(0,0){\mbox{$\scriptstyle  11$}}}

\put(14,22.5){\makebox(0,0){\mbox{$\scriptstyle V_4 $}}}
\end{picture}

\begin{picture}(10,10)(5,17)
\put(14,22.5){\circle{5}}

\put(14,25){\makebox(0,0){\mbox{$ $}}}
%\put(15.9,20.8){\line(1,-1){5}}
%\put(12.1,24.4){\line(-1,1){5}}
\put(14,20){\line(0,-1){6}}
\put(15.9,24.4){\line(1,1){5}}

\put(14,22.5){\makebox(0,0){\mbox{$\scriptstyle T_3 $}}}
\end{picture}

\begin{picture}(10,10)(7.5,8.5)
\put(14,22.5){\circle{5}}

\put(14,25){\makebox(0,0){\mbox{$ $}}}
%\put(15.9,20.8){\line(1,-1){5}}
%\put(12.1,24.4){\line(-1,1){7}}
\put(13,25){\line(-1,2){3}}
\put(15.9,24.4){\line(1,1){8.4}}

\put(14,22.5){\makebox(0,0){\mbox{$\scriptstyle S_5 $}}}
\end{picture}

\begin{picture}(10,10)(22.75,-2.5)
\put(14,22.5){\circle{5}}

\put(7,29.5){\makebox(0,0){\mbox{$\scriptstyle  5$}}}
\put(20,29.5){\makebox(0,0){\mbox{$\scriptstyle  8$}}}
\put(14,32){\makebox(0,0){\mbox{$\scriptstyle  3$}}}
%\put(15.9,20.8){\line(1,-1){5}}
\put(12.1,24.4){\line(-1,1){4}}
\put(14,25){\line(0,1){6}}
\put(15.9,24.4){\line(1,1){4}}

\put(14,22.5){\makebox(0,0){\mbox{$\scriptstyle P_7 $}}}
\end{picture}

\begin{picture}(10,10)(18.3,-4)
\put(14,22.5){\circle{5}}

\put(14,25){\makebox(0,0){\mbox{$ $}}}
%\put(15.9,20.8){\line(1,-1){5}}
\put(12.1,24.4){\line(-1,1){4}}
\put(14,25){\line(0,1){4}}
\put(15.9,24.4){\line(1,1){4}}
\put(7,29.5){\makebox(0,0){\mbox{$\scriptstyle  4$}}}
\put(20,29.5){\makebox(0,0){\mbox{$\scriptstyle  13$}}}

\put(14,22.5){\makebox(0,0){\mbox{$\scriptstyle V_1 $}}}
\end{picture}

\begin{picture}(10,10)(29.5,-13)
\put(14,22.5){\circle{5}}

\put(7,30.5){\makebox(0,0){\mbox{$\scriptstyle  2$}}}
\put(21,30.5){\makebox(0,0){\mbox{$\scriptstyle  6$}}}
\put(14,32){\makebox(0,0){\mbox{$\scriptstyle  9$}}}
%\put(15.9,20.8){\line(1,-1){5}}
\put(12.1,24.4){\line(-1,1){5}}
\put(14,25){\line(0,1){6}}
\put(15.9,24.4){\line(1,1){5}}

\put(14,22.5){\makebox(0,0){\mbox{$\scriptstyle W_6 $}}}
\end{picture}

\end{picture}}
\begin{figure}[h]
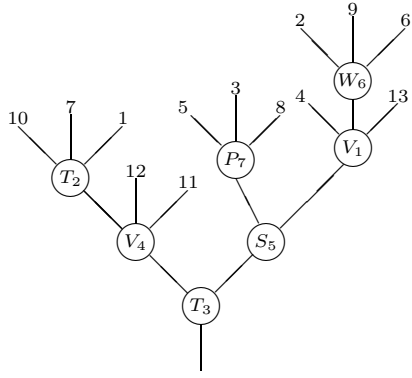
\caption{{\small Typical reduced $n$-planar tree with seven vertices. Here    $V,W,P,T,S$ are $n$-ordinals decorating corresponding vertices and the subscripts indicate the linear order. The numbers on the leaves fix a bijection $\rho_{\tau}.$    }}\label{XKL}\end{figure} 

The coloured  operad $NO^{(n)}$ has nonempty and nonterminal $n$-ordinals as objects. For any  list of $n$-ordinals $T_1,\ldots,T_k;S$ the set of operations
$NO^{(n)}( T_1,\ldots,T_k;S)$  is the subset of $Trees^{n,k}_S$ which consist of $n$-planar trees with decorations exactly $T_1,\ldots,T_k$ whose linear order coincides with the ordering of the list $T_1,\ldots,T_k.$ The symmetric group acts by changing the linear order (but does not change the tree structure). The composition law in this operad is given by an insertion  of a reduced $n$-planar tree to a vertex of another reduced $n$-planar tree. This substitution operation is explained in detail in \cite{batanin-berger}.

\begin{remark} We now indicate the required changes to get the analogous description of general $n$-operads (symmetric, braided) and of constant-free $n$-operad (symmetric, braided). The difference is in the restrictions on the underlying planar tree $\tau.$ In case of general operads we drop the reducedness assumption. For constant free we require half-reducedness, that is we exclude trees with stamps (valency $0$ vertices) but still allow vertices with only one incoming edge. 

The symmetric operad case is easy since we don't need any decorations (except for a bijection $\rho_\tau$) and domination conditions. In the braided case we have to use elements of the braid groups  instead of the bijections $\rho_\tau .$

\end{remark} 

 \begin{lem} \label{lemma:op category of quasibijections} 
The underlying categories of  operads $O^{(n)},CFO^{(n)}$ and $NO^{(n)}$ are the opposite to the category of quasibijections,  the category of quasibijections between nonempty $n$-ordinals   and the category of quasibijections between  nonempty and nonterminal $n$-ordinals correspondingly. \end{lem}

\begin{proof}
The proof is the same for all three operads above. We need to calculate $NO^{(n)}(T;S).$  The underlying planar tree for a typical element from $NO^{(n)}(T;S)$ must be a corolla.  The condition that $S$ dominates the complementary relations generated by $T$ is equivalent to the requirement that $\rho_{\tau}$ considered as a map of totally $n$-ordered sets $S\to L(\tau) = T$ is a quasibijection. Hence, we have a bijection $NO^{(n)}(T;S)\rightarrow \QQ_n(S,T).$ This bijection obviously takes the substitution operation in $NO^{(n)}$ to the composition of  quasibijections in reverse order. 
\end{proof}

\subsection{The bimodules $d(NO^{(n)}), d(CFO^{(n)})$ and $d(O^{(n)})$ are constantly disconnected}\label{cd}

The combinatorics of the bimodule $d(NO^{(n)})$ is simpler and we start our proof from it. We then observe that the proof for two other bimodules  can be reduced to the previous case.

Let us denote by $$\tau_{\rho_{\tau}}(T_1,\ldots,T_k;S)\in Trees_S^{n,k}$$ a typical element from $NO^{(n)}( T_1,\ldots,T_k;S).$  One can take the value of the  bimodule $d(NO^{(n)})( T_1,\ldots,T_k;S)
$  equal to the set of $\tau_{\rho_{\tau}}(T_1,\ldots,T_k;S)$ for which the linear order on vertices of $\tau$ coincides with the order generated by  walking around $\tau$ in the clockwise direction starting from the root vertex as first. 
For the first point observe that the fiber $(T_1,\ldots,T_k)/p'$ of the functor $$p':\gint d(NO^{(n)})\rightarrow (\QQ^{op}_n)^k,$$ for a fixed list of $n$-ordinals $T_1,\ldots,T_k$, has the following explicit description.

An object of this fiber is given by an element   $\tau_{\rho_{\tau}}(T_1,\ldots,T_k;S).$  
A morphism in the fiber        $(T_1,\ldots,T_k)/p'$ from $\tau_{\rho_{\tau}}(T_1,\ldots,T_k;S)$ to $\tau'_{\rho'_{\tau'}}(T_1,\ldots,T_k;S')$ exists only if $\tau = \tau'$ and is determined by a quasibijection $\rho:S'\to S$ such that $\rho'_{\tau} = \rho_{\tau}\cdot |\rho|.$ 

It is obvious from this description  that the fiber $(T_1,\ldots,T_k)/p'$ splits as a coproduct of categories  of the form $$\coprod_{\tau} \tau(T_1,\ldots,T_k)/p',$$ 
where the objects of the category $\tau(T_1,\ldots,T_k)/p'$ are elements $\tau_{\rho_{\tau}}(T_1,\ldots,T_k,S)$ with fixed planar tree $\tau.$ 

Let $\mathcal SO^{\scriptstyle \mathcal{N}O^{(n)}}$ be the classifier for internal normal $n$-operads in a categorical symmetric operad (it was denoted $\mathbf{rh}^n$ in  \cite{SymBat}).
\begin{lemma}\label{initialobject} 

The category  $\tau(T_1,\ldots,T_k)/p'$ is isomorphic to the comma-category $\tau/(\JJ_n(k))^{op}$  from \cite[Lemma 4.2]{SymBat} if $\tau$ is considered as an object of $\mathcal SO^{\scriptstyle \mathcal{N}O^{(n)}}$. 
The categories  $\tau/(\JJ_n(k))^{op}$ and $\tau(T_1,\ldots,T_k)/p'$ are finite posets with an initial object.
\end{lemma}
\begin{proof}
Any element $\tau_{\rho_{\tau}}(T_1,\ldots,T_k,S)$ is identified with an object of $(\JJ_n(k))^{op}$ under $\tau$ because $\rho_{\tau}^{-1}$ provides a labeling of the elements of $|S|$ (that is, an object of $(\JJ_n(k))^{op}$) and under this relabeling there is a unique morphism in $\mathcal SO^{\scriptstyle \mathcal{N}O^{(n)}}$ from $\tau$ to this object. 
It is clear that both categories are finite posets and that the correspondence above establishes an isomorphisms of these posets.

Now we invoke Lemma 4.3 from \cite{SymBat} which claims that the category $\tau/(\JJ_n(k))^{op}$ is nonempty and connected. In fact the proof established more. Namely, that for each objects $S'$ and $S''$ in this category there exists an object $S$ and a span of morphisms $S'\leftarrow S\rightarrow S''.$ Since the category is a directed and finite poset an obvious induction implies that $\tau/(\JJ_n(k))^{op}$ has an initial object.
  \end{proof}

We can now finish the proof that the bimodules $d(NO^{(n)})$ and $d(O^{(n)})$ are constantly disconnected. We have to check that every $k$-tuple of quasibijections $f_1: T_1\to T'_1,\ldots,f_k:T_k\rightarrow T'_k$  induces  a weak equivalence  of nerves  of categories $(T_1,\ldots,T_k)/p'$ and $(T'_1,\ldots,T'_k)/p'.$  But these categories are coproducts of categories over the same indexed set and the $k$-tuple $f_1,\ldots,f_k$ sends a  summand indexed by $\tau$ to the summand indexed by $\tau.$ Since each summand has contractible nerve, we conclude that  $N((T_1,\ldots,T_k)/p') \to N((T_1,\ldots,T_k)/p')$  is a weak equivalence and hence the substitude $(NO^{(n)},\QQ_n^{op})$ is left localisable.

The proofs for the substitudes  $(CFO^{(n)}, \QQ_n^{op})$ and $(O^{(n)}, \QQ_n^{op})$ are similar once we know that the analogue of Lemma \ref{initialobject} holds for these case.  Let $\mathcal SO^{O^{(n)}}$  be the classifier for internal $n$-operads inside categorical symmetric operads (this is $\mathbf{ph}^n$ from \cite{SymBat}). It was shown     in \cite[Lemma 8.2]{SymBat}     that 
the natural operadic functor $p:\mathcal SO^{O^{(n)}} \to \mathcal SO^{NO^{(n)}}$ has a (nonoperadic) right adjoint which is the identity on the objects of  $(\JJ_n(k))^{op}.$  This induces an adjoint pair of functors between $\tau/(\JJ_n(k))^{op}$ and $p(\tau)/(\JJ_n(k))^{op},$ so the nerve of $\tau/(\JJ_n(k))^{op}$ is contractible again. The rest of the proof follows the same argument.  Exactly the same sort of argument works for constant free $n$-operad since \cite[Lemma 8.2]{SymBat}  was based on the possibility to drop-off the stamps and introduce a vertex to in the middle of an edge and then contract an additional  edge again. In constant-free case we don't need to worry about stamps.  

\subsection{$NO^{(n)}, CFO^{(n)}$ and $O^{(n)}$  are unary tame}\label{ut}

This time we check unary tameness of $(\ON, \QQ_n^{op})$ using the techniques from \cite[Section 12.22]{batanin-berger}. The proofs for the other substitudes are very similar and  simpler.

The classifier $(O^n)^{O^n + \mathcal{Q}_n^{op}}$ has as objects $n$-planar trees with additional decoration of each vertex by colours $X$ and $K$  
exactly like it was stated in \cite[Section 12.26]{batanin-berger}. Similarly, the morphisms are generated by contractions of $n$-planar subtrees all whose vertices have $X$-colour but we have additional generators which correspond to the action of quasibijections on vertices with colour $K.$ Observe that these are unary morphisms which do not change the underlying planar tree. 

 The classifier $(O^n)^{\mathcal{Q}_n^{op} + \mathcal{Q}_n^{op}}$ has the same objects again but morphisms are generated by action of quasibijections only. According to Lemma \ref{a+a} we need to find a subcategory of  $(O^n)^{\mathcal{Q}_n^{op} + \mathcal{Q}_n^{op}}$  which will be final in $(O^n)^{O^n + \mathcal{Q}_n^{op}}.$

Let us call an object  $a\in (O^n)^{O^n + \mathcal{Q}_n^{op}}$ {\it noncontractible}  if for any morphism $\phi:a\to b$ it is either a unary or a nullary morphism. If moreover, for any nullary $\phi:a\to b$ there is a morphism $\psi:b\to a$ such that $\psi\cdot\phi$ is the identity on $a$ and $X$ contains at least one $X$-vertex we call $a$ {\it a noncontractible retract}. Recall that nullary generators in a classifier for polynomial monads $\Tt^\Ss$ correspond to constants of the polynomial monad $\Ss,$ i.e., to the operations whose source is empty. Unary generators correspond to operations in $\Ss$ whose source is a singleton. Thus in $(O^n)^{O^n + \mathcal{Q}_n^{op}}$ the nullary generators  are exactly the morphisms of inserting an $X$-vertex of valency  two inside an edge of the $n$-planar tree. The last condition of retractability simply  means  that any insertion of a new $X$-vertex of valency two creates an internal edge connecting two $X$-vertices (we can then contract the subtree consisting of these two vertices) or, equivalently, there are no two $K$-vertices connected by an edge. It follows from this, in particular, that the root vertex of a noncontractible retract must an $X$-vertex.  

Let $\mathbf{nr}\subset  (O^n)^{\mathcal{Q}_n^{op} + \mathcal{Q}_n^{op}}$ be the full subcategory of noncontractible retracts. Let us prove that it is final in $ (O^n)^{O^n + \mathcal{Q}_n^{op}}.$ 

We first need to show that for each $a\in (O^n)^{O^n + \mathcal{Q}_n^{op}}$ there is at least one arrow to a noncontractible retract. Let $v(b)$ be the number of edges of the underlying planar tree of $a.$ Notice that an object $b$ is noncontractible  if and only if for any  morphism $\phi:b\to c$   the number $v(c)\ge v(b)$ and, hence, $b$ is noncontractible if and only if $l(b) = v(b) - \min_{b\to c}v(c) =0.$ We now can use obvious induction by $l(a)$ to prove that there is a morphism $f:a\to b$ to a noncontractible object. Now, if $b$ is noncontractible there is a unique way to map it to a noncontractible retract by inserting an $X$-vertex to each edge connecting two $K$-vertices and on the root edge if necessary.

A further property of noncontractible retracts is: for any unary $\phi:a\to b$ if $a$ is a noncontractible retract then $b$ is a noncontractible retract as well. Indeed, if there exists $\psi:b\to c$  with $v(c) < v(b)$ then the composite $\phi\cdot\psi$ provides a morphism with $v(a)<v(c).$ Hence, $b$ is noncontractible. It is also a noncontractible retract because a unary $\phi$ does can not force two $K$-vertices become connected by an edge in $b.$

The category $(O^n)^{O^n + \mathcal{Q}_n^{op}}$ has the following diamond property. Any span of  generators$$b\stackrel{\phi}{\leftarrow}a\stackrel{\psi}{\to}c$$ in $C$ can be completed to a commutative square by a cospan of generators (or identities)$$b\stackrel{\psi^*}{\to} d\stackrel{\phi^*}{\leftarrow}c.$$
The proof of this fact is exactly the same as in \cite[Section 12.24]{batanin-berger}.

Now, suppose we are given a span of generators  $b\stackrel{\phi}{\leftarrow}a\stackrel{\psi}{\to}c,$ where $c$ is a noncontractible retract. Then we claim that we can complete this span to a commutative square  in such a way that in the cospan $b\stackrel{\psi^*}{\to} d\stackrel{\phi^*}{\leftarrow}c$ the morphism $\phi^*$ is unary and $d$ is a noncontractible retract. Indeed, since $c$ is a noncontractible retract $\phi^*$ is either unary and everything is proved or nullary. If it is nullary there is a retraction $r:d\to c$ and postcomposing this cospan with $r$ we get another cospan in which the right morphism is the identity.

Finally given a span of morphisms  $b\stackrel{\phi}{\leftarrow}a\stackrel{\psi}{\to}c$ in which $b$ and $c$ both are noncontractible retracts we first factor $\phi$ and $\psi$ on generators and then complete the resulting diagram to a commutative grid using the diamond property. Then we see from the previous argument that we always can obtain the grid such that in the chain of morphisms on the boundary
$$b\to d_1\to \ldots \to d_i \leftarrow e_j \ldots \leftarrow e_1 \leftarrow c$$
all morphisms are unary and all objects are noncontractible retracts.
Thus we  proved that the category $a/ (O^n)^{O^n + \mathcal{Q}_n^{op}}$ is connected and so $\mathbf{nr}\subset (O^n)^{O^n + \mathcal{Q}_n^{op}}$ is final.
\begin{proof}[Proof of Theorem \ref{thm: operads for n-Ops are Reedy locally constant}] We finally use the criteria established in  Theorem \ref{left localisable criteria} and the results of Sections \ref{cd} and \ref{ut}.

 \end{proof}

\section{Stabilisation of operads and algebras} \label{sec:stablization}

In this Section we  obtain an extension of Theorem \ref{model=2}  to all intermediate cases  $2<n<\infty.$ It justifies our assertion that $\Ww_{\infty}$-locally constant $n$-operads are indeed a model of higher braided operads. We then apply it to prove a stabilisation theorem for algebras.

\subsection{Stabilisation of higher braided operads}

For an $n$-ordinal $R$ and $0\le p\le n$ we consider its  {\it $p$-suspension} $\mathbf{s}_p(R)$ which is an $(n+1)$-ordinal with the underlying set $R ,$ and the $m$-th order coincides with $<_m $ on $R$ until $m< p,$ the $p$-th order is empty and the $m$-th order for $m>p $ coincides with  $<_{m-1}$ on $R.$  

For example, the suspensions $\mathbf{s}_0(T)$ and $\mathbf{s}_2(T)$ of the $2$-ordinal on the left are  $3$-ordinals displayed on the right:

% Generated with LaTeXDraw 2.0.8
% Sat Jul 11 17:03:40 EST 2015
% \usepackage[usenames,dvipsnames]{pstricks}
% \usepackage{epsfig}
% \usepackage{pst-grad} % For gradients
% \usepackage{pst-plot} % For axes
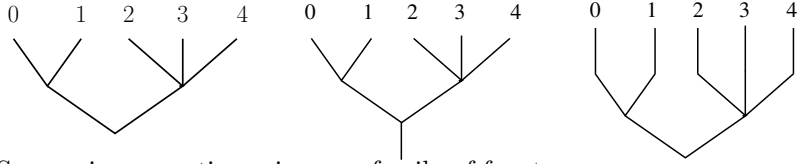
\[
\psscalebox{0.45} % Change this value to rescale the drawing.
{
\begin{pspicture}(2,-2)(6.81,1.7822068)
\psline[linecolor=black, linewidth=0.05](1.06,-0.35779327)(3.06,-1.7577933)(5.06,-0.35779327)(4.96,-0.35779327)
\psline[linecolor=black, linewidth=0.05](0.06,1.0422068)(1.06,-0.35779327)(2.06,1.0422068)(2.06,1.0422068)
\psline[linecolor=black, linewidth=0.05](3.46,1.0422068)(5.06,-0.35779327)(6.66,1.0422068)(6.66,1.0422068)
\psline[linecolor=black, linewidth=0.05](5.06,1.0422068)(5.06,-0.35779327)(5.06,-0.35779327)
\rput[b](0.06,1.5422068){\Huge 0}
\rput[b](2.06,1.5422068){\Huge 1}
\rput[b](3.46,1.5422068){\Huge 2}
\rput[b](5.06,1.5422068){\Huge 3}
\rput[bl](6.66,1.5422068){\Huge 4}
\end{pspicture}
}
\scalebox{0.4} % Change this value to rescale the drawing.
{
\begin{pspicture}(-2,-1.9)(7.237695,2.81)
\psline[linewidth=0.05](1.1882031,0.13)(3.188203,-1.2646877)(5.1882033,0.13)(5.088203,0.13)
\psline[linewidth=0.05](0.18820313,1.5353125)(1.1882031,0.13531242)(2.188203,1.5353125)
\psline[linewidth=0.05](3.5882032,1.5353125)(5.1882033,0.13531242)(6.7882032,1.5353125)
\psline[linewidth=0.05](5.18,1.63)(5.18,0.11)
\usefont{T1}{ptm}{m}{n}
\rput(0.1566211,2.4353125){\Huge 0}
\usefont{T1}{ptm}{m}{n}
\rput(2.0829296,2.4353125){\Huge 1}
\usefont{T1}{ptm}{m}{n}
\rput(3.5604882,2.4353125){\Huge 2}
\usefont{T1}{ptm}{m}{n}
\rput(5.1338477,2.4353125){\Huge 3}
\usefont{T1}{ptm}{m}{n}
\rput(6.9860744,2.4353125){\Huge 4}
\psline[linewidth=0.05](3.18,-1.25)(3.18,-2.5)
\end{pspicture} 
}
\scalebox{0.4}{
\begin{pspicture}(-2,-2.1)(6.9376955,2.6773438)
\psline[linewidth=0.05](1.1482031,-1.2426562)(3.1482031,-2.637344)(5.148203,-1.2426562)(5.048203,-1.2426562)
\psline[linewidth=0.05](0.14820312,0.1626563)(1.1482031,-1.2373438)(2.1482031,0.1626563)
\psline[linewidth=0.05](3.5482032,0.1626563)(5.148203,-1.2373438)(6.7482033,0.1626563)
\psline[linewidth=0.05](5.14,0.2573438)(5.14,-1.2626562)
\usefont{T1}{ptm}{m}{n}
\rput(0.1566211,2.2826562){\Huge 0}
\usefont{T1}{ptm}{m}{n}
\rput(2.0429296,2.2826562){\Huge 1}
\usefont{T1}{ptm}{m}{n}
\rput(3.5804882,2.3026564){\Huge 2}
\usefont{T1}{ptm}{m}{n}
\rput(5.1138477,2.3026564){\Huge 3}
\usefont{T1}{ptm}{m}{n}
\rput(6.6860743,2.3026564){\Huge 4}
\psline[linewidth=0.05](0.16,1.6773438)(0.16,0.1573438)
\psline[linewidth=0.05](2.14,1.6773438)(2.14,0.1573438)
\psline[linewidth=0.05](3.56,1.6773438)(3.56,0.1573438)
\psline[linewidth=0.05](5.14,1.7573438)(5.14,0.2373438)
\psline[linewidth=0.05](6.74,1.6773438)(6.74,0.1573438)
\end{pspicture} 
}
\]

Suspension operations   give us a family of  functors
$$\mathbf{s}_p:Ord(n)\rightarrow Ord(n+1), \ 0\le p \le n.$$ We also define an $\infty$-suspension functor $Ord(n)\rightarrow Ord(\infty)$ as follows. For an $n$-ordinal $T$ its $\infty$-suspension is an $\infty$-ordinal $\mathbf{s}_{\infty}T$ whose underlying set is the same as the underlying set of $T$ and $a<_p b$ in $\mathbf{s}^{\infty}T$ if
$a<_{n+p-1} b$ in $T .$ It is not hard to see that the sequence
$$ Ord(0)\stackrel{\mathbf{s}}{\longrightarrow} Ord(1) \stackrel{\mathbf{s}}{\longrightarrow} Ord(2) \longrightarrow \ldots \stackrel{\mathbf{s}}{\longrightarrow} Ord(n) \longrightarrow \ldots \stackrel{\mathbf{s}_{\infty}}{\longrightarrow} Ord(\infty),$$
exhibits $Ord(\infty)$ as a colimit of $Ord(n) .$%\vspace{1ex}

Any of the  suspension functors can be restricted to the category of quasibijections. We will denote such a restriction by
$$\mathbf{s}_p:\QQ_n\rightarrow \QQ_{n+1} .$$  
To simplify notation we will denote any suspension functor  just $\mathbf{s}$ since the results below are valid for any $0\le p \le n.$

The suspension functor can be extended to a substitude morphism (cf. \cite{batanin-baez-dolan-via-semi} for explanation) 
$$(\ON,\QQ_n^{op})\to (O^{(n+1)},\QQ_{n+1}^{op})$$
 and in the limiting case  to  a morphism
$$(\ON,\QQ_n^{op})\to (O^{(\infty)},\QQ_{\infty}^{op}).$$
Both these morphisms are over the substitude of symmetric operads $(SO,\Sm).$
Both morphisms will be denoted $(\Sigma,\mathbf{s}^{op})$ to shorten notations, moreover the same notations will be applied for any finite sequence of 
iterated suspensions:
 $$(\Sigma,\mathbf{s}^{op}): (\ON,\QQ_n^{op})\to (O^{(m)},\QQ_{m}^{op}) \ , \ n< m.$$
 
 \begin{proposition}\label{BCSigma} For any $n< m \le \infty$ 
 the morphism $(\Sigma,\mathbf{s}^{op})$ is Beck-Chevalley.  
 \end{proposition} 
\begin{proof} It follows straightaway from  Propositions \ref{operadicstabilization} and \ref{2of3}. 
\end{proof}

We finally arrive to the following Stabilisation Theorem for higher operads.

\begin{theorem}\label{koperadicstabilization} Let $\VV$ be a combinatorial symmetric monoidal model category with cofibrant unit.
Then for all $n\ge 3$ and $2\le k+1\le n$ 
the  symmetrisation functor
$$sym_{ n}^{}: Op^{\Ww_k}_{n}(\VV) \rightarrow SO(\VV)$$
and the suspension functor 
$$\Sigma_!:Op^{\Ww_k}_{n}(\VV)\rightarrow  Op^{\Ww_k}_{m}(\VV) \ , n<m\le \infty$$ are  left  Quillen
equivalences. 

For $n=2$ and any $1\le k\le \infty$ the functor
$$bsym_{2}^{}: Op^{\Ww_k}_{2}(\VV) \rightarrow BO(\VV)$$
is a left Quillen equivalence.

\end{theorem}

\begin{proof} The statements  follow from  points (4) and (5) of Theorem  \ref{Milgram},   Theorem \ref{thm: operads for n-Ops are Reedy locally constant}, Propositions  \ref{BCSigma} and \ref{operadicstabilization} and Corollaries  \ref{BC_for_W} and  \ref{lifting QE for transfer}.
\end{proof}

In the truncated case we obtain the following

\begin{corollary}\label{truncatedstabilization}  
Let $k\ge 0$ and let $\VV$ be a $k$-truncated combinatorial symmetric monoidal model category with cofibrant  unit.  
Then the symmetrisation functor
$$sym_{n}^{}: Op^{\Ww_\infty}_{n}(\VV) \rightarrow SO(\VV)$$
and the suspension functor 
$$\Sp_!:Op^{\Ww_\infty}_{n}(\VV)\rightarrow  Op^{\Ww_\infty}_{m}(\VV) \ , \ n< m \le \infty$$
are left Quillen equivalences for $3\le k+2\le n \le \infty.$  

For $n=2$  and $0\le k\le \infty$ the functor
$$bsym_{2}^{}: Op^{\Ww_\infty}_{2}(\VV) \rightarrow BO(\VV)$$
is a left Quillen equivalence.

\end{corollary} 

\begin{proof} It follows from Theorem \ref{kloc=lok}. 

\end{proof}

\subsection{Baez-Dolan stabilisation} \label{sec:Baez-Dolan}

We now show how Corollary \ref{truncatedstabilization}   implies a version of the Baez-Dolan stabilisation hypothesis.

Following \cite{batanin-baez-dolan-via-semi} we give the following definitions. Let $Ass_n\in Op_n(\VV)$ be the operad with constant values $(Ass_n)_T = I$, the monoidal unit, for every $T\in Ord(n).$  
Let $\mathcal{G}_{n}\in Op_n(\VV)$ be its cofibrant replacement.      We will denote  by $B_{n}(\VV)$   the category of $G_{n}$-algebras in $\VV.$  Let also $E_{\infty}(\VV)$ be the model category of $E_{\infty}$-algebras in $\VV$, that is, the category of algebras of a cofibrant replacement $E$ of the symmetric operad $Com.$

There is an isomorphism of categories of algebras of any $n$-operad $\mathcal{P}_n$ and the $(n+1)$-operad $\Sp_!(\mathcal{P}_{n})$ \cite[Lemma 2.5]{batanin-baez-dolan-via-semi}.   
Since $\Sp_!$ is a left Quillen functor,  the operad $\Sp_!(\mathcal{G}_{n})$ is cofibrant.
There is a  map of $(n+1)$-operads  $i: \Sp_!(\mathcal{G}_{n})\to \mathcal{G}_{n+1}.$  
Indeed, we have  $\Sp^*(Ass_{n+1}) = Ass_n$ and by adjunction we have  a map $\Sp_!(\mathcal{G}_{n})\to Ass_{n+1}.$ We  also have a trivial fibration $\mathcal{G}_{n+1}\to Ass_{n+1}.$ Since $\Sp_!(\mathcal{G}_{n})$ is cofibrant there is a lifting $i: \Sp_!(\mathcal{G}_{n})\to \mathcal{G}_{n+1}.$  

The morphism  $i$ induces a Quillen adjunction
between algebras of $\Sp_!(\mathcal{G}_{n})$ and algebras of $\mathcal{G}_{n+1}$ and so between algebras of $\mathcal{G}_{n}$ and $\mathcal{G}_{n+1}.$ Slightly abusing notation we will denote this adjunction $i^*\vdash i_!.$

We also have  maps $j_n:sym_n(\mathcal{G}_n)\to E$ for each $n\ge 2$ which induce Quillen adjunctions $j_n^*\vdash (j_n)_!$ between $B_n(\VV)$ and $E_{\infty}(\VV)$ which commute with $i^*$ and $i_!$ \cite[Section 3.5]{batanin-baez-dolan-via-semi}.  

\begin{theorem}\label{BD stabilisation} Let  $k\ge 0$ and  let $\VV$ be a $k$-truncated combinatorial symmetric monoidal model category with cofibrant unit. Then  $$i_!: B_n(\VV)\to B_{n+1}(\VV)$$
  and
  $$(j_n)_!: B_n(\VV)\to E_{\infty}(\VV)$$
  are left Quillen equivalences for $n\ge k+2.$ 
\end{theorem}

\begin{proof} Let  $n\ge 3.$ The operad $\mathcal{G}_n$ is cofibrant and  is equipped with a trivial fibration to $Ass_n$.
Applying $sym_n$ to this trivial fibration we obtain $sym_n(\mathcal{G}_n)\to sym_n (Ass_n) = Com$, and we will prove this is a weak equivalence. Let $Com\to Com'$ be a fibrant replacement in $SO(\VV)$. Then we have a map $sym_n(\mathcal{G}_n)\to Com'.$ 
This induces a mate $\mathcal{G}_n\to des_n(Com').$ The functor $des_n$ preserves weak equivalence between operads and so $des_n(Com')$ is weakly equivalent to $des_n(Com) = Ass_n.$ So $\mathcal{G}_n\to des_n(Com')$ is a weak equivalence because $\mathcal{G}_n$ is a cofibrant replacement of $Ass_n.$ Observe that $\mathcal{G}_n$ is also a cofibrant object in the localised category $Op^{\Ww_\infty}_{n}(\VV).$ By Corollary \ref{truncatedstabilization} $sym_n$ is a left Quillen equivalence after localisation, so the mate $sym_n(\mathcal{G}_n)\to Com'$ is a weak equivalence, and, hence, $sym_n(\mathcal{G}_n)$ is weakly equivalent to the $E_{\infty}$ operad $E.$ This proves the second statement.  

The statement for $i_!$ follows again from the two out of three property. 

The case $k=0$ and $n=2$ is somewhat special. In this case the contractible $2$-operad $\mathcal{G}_2$ gives a contractible braided operad $X = bsym_2 (\mathcal{G}_2)$ by Corollary \ref{truncatedstabilization}. This means, in particular, that for each $n\ge 0$ there is a weak equivalence of $\Br_n$-objects  $X_n\to I.$ The symmetrisation of $X$ is obtained as the quotient $q_n:X_n\to X_n/\mathrm{P}\Br_n , $ by the pure braid group on $n$-strands for each $n\ge 0.$ Then for any $Y\in \VV$ (choosing an appropriate simplicial resolution $Y_*$ of $Y$) we have an induced commutative square of simplicial sets:

% Note: we defined Map_V above, not Map_\VV.
\begin{equation*}\xymatrix{
 Map_V(X_n,Y) %\ar@<2.5pt>[r]^{\psi_!} 
 \ar@{<-}@<0pt>[d]_{}
\ar@<0pt>[r]^{}
& 
 \pi_0(Map_V(X_n,Y))%\ar@<0pt>[d]_{\psi}
 \ar@{<-}@<0pt>[d]^{} \\
 Map_V(X_n/\mathrm{P}\Br_n ,Y) \ar@<0pt>[r]^{}
%\ar@<2.5pt>[ur]^{\gamma} %\ar@{<-}@<0pt>[d]^{\alpha} %\ar@<2.5pt>[u]^{\beta}  
& 
 \pi_0(Map_V(X_n/\mathrm{P}\Br_n ,Y))}
\end{equation*} 
Since $\VV$ is $0$-truncated, the horizontal arrows in this diagram are weak equivalences. The discrete space in the right bottom corner  is the set of fixed points of the induced action of the pure braid group on the connected components of the corresponding mapping space. Since  $X_n$ is contractible, the action of the braid groups is homotopically trivial, so the right vertical arrow in the diagram is an isomorphism. This means that the left vertical arrow is a weak equivalence, by the two out of three property. Thus, the quotient morphism $q_n$ is a weak equivalence and the symmetrisation of $X$ is weakly equivalent to the symmetric operad $Com$ in $\VV.$ This finishes the proof.

\end{proof}

\begin{remark} Theorem  \ref{BD stabilisation} is an improvement of main result  \cite[Theorem 3.7]{batanin-baez-dolan-via-semi} where the existence of a standard system of simplices is required. On the other hand this Theorem from loc. cit. is proved for an arbitrary cofibrantly generated monoidal model category with cofibrant unit $\VV$ but Theorem \ref{BD stabilisation} asks $\VV$ to  be combinatorial.    
\end{remark}

Recall that Rezk's $(n+k,n)$-categories are fibrant objects in the model category $\Theta_n Sp_k, -2\le k\le \infty $, which is a truncation of the model category of Rezk's complete  $\Theta_n$-spaces $\Theta_n Sp_{\infty}$. Here $Sp_k$ is the $k$-truncation of the category of simplicial sets, and models homotopy $k$-types. The category $\Theta_n Sp_{\infty}$ is itself a certain Bousfield localisation of the category of simplicial presheaves $Sp^{\Theta_n^{op}}$ with its  injective model structure. 
This is a cartesian closed combinatorial model category that is $(n+k)$-truncated and satisfies all the hypotheses of Theorem \ref{BD stabilisation} (see \cite{Rezk}).
Recall also that the  category of  Rezk's $m$-tuply monoidal $(n+k,n)$-categories is the category  of fibrant objects in the (semi)model category $B_m(\Theta_n Sp_k).$ 

We immediately have

\begin{corol}[Stabilisation for Rezk's $(n+k,n)$-categories]\label{wkc} The  suspension functor induces a left Quillen equivalence 
$$ i_!: B_m(\Theta_n Sp_k) \to  B_{m+1}(\Theta_n Sp_k)$$   for   $m\ge n+k+2$ and, hence, an equivalence between homotopy categories of Rezk's $m$-tuply monoidal $(n+k,n)$-categories  and Rezk's $(m+1)$-tuply monoidal $(n+k,n)$-categories. 
\end{corol}

Recall that Tamsamani introduced a definition of weak $n$-categories, expanded by Pellissier and reframed by Simpson \cite{Simpson}. Simpson begins with a cartesian, left proper, and tractable (meaning: combinatorial and with domains of the generating (trivial) cofibrations cofibrant) model category $\M$, then uses an iterated procedure to produce a model category $PC^n(\M)$ of $n$-precategories, then left Bousfield localises to produce a model category (which we denote $Seg^n(\M)$) where the fibrant objects are the {\em Segal $n$-categories}, i.e., satisfy the full Segal condition \cite[Theorem 19.2.1, Proposition 19.4.1]{Simpson}.

Applied to the model category $Set$ (with weak equivalences the isomorphisms, and every map a cofibration and a fibration), this produces \textit{the model category of $n$-prenerves} (also known as $n$-precategories) \cite[Definition 20.2.1]{Simpson}, Simpson's reframing of Tamsamani's model. Applied to the Kan-Quillen model structure on simplicial sets, this produces {\em the model category of Segal $n$-precategories}, defined by Hirschowitz and Simpson. 

\begin{proposition}
Let $\M$ be a combinatorial, monoidal model category with cofibrant unit. Then the categories $PC^n(\M)$ are too, and hence satisfy our Stabilisation Theorem \ref{koperadicstabilization}.
\end{proposition}

\begin{proof}
This is essentially already proven in \cite{Simpson}, but with more hypotheses on $\M$. An object of $PC^n(\M)$ is a pair $(X,A)$ where $X$ is a set, $\Delta_X$ is the category of finite linearly ordered sets decorated by the elements of $X$, and $A:\Delta_X^{op} \to \M$ is a functor satisfying the unitality condition that $A(x_0) = \ast$ for any sequence of length zero \cite[Chapter 6]{Simpson}. 
% page 115 = p124 of the pdf
It is well-known that such a functor category is combinatorial (with the injective, projective, or Reedy model structure) when $\M$ is, and is monoidal (with respect to the levelwise monoidal structure for the injective case, and to the Day convolution structure for projective and Reedy). The unit is a representable functor built from the unit $I$ of $\M$, and is cofibrant if $I$ is.
\end{proof}

\begin{proposition}
Let $\M$ be a left proper, combinatorial, monoidal model category with cofibrant unit. Then $Seg^n(\M)$ is, too, and hence satisfies our Stabilisation Theorem \ref{koperadicstabilization}.
\end{proposition}

\begin{proof}
This is a standard application of left Bousfield localisation. The monoidal structure on $Seg^n(\M)$ is the same as that on $PC^n(\M)$ and hence the unit is still cofibrant. As proven in \cite{white-localization}, the pushout product axiom on a localised model structure is equivalent to a condition relating the morphisms one is inverting and the monoidal product. This condition is spelled out and verified in \cite[Theorem 1.4]{Rezk} for the localisation encoding the Segal condition. The condition is also verified in \cite[Theorem 12.1.1]{Simpson}.
\end{proof}

In order to apply our Theorem \ref{BD stabilisation}, to prove Baez-Dolan stabilisation, we need a $k$-truncated model category. There are two ways to proceed. We can either begin with a truncated model category, such as $\M = Sp_k$, or we can build $Seg^n(\M)$ for a general $\M$ and then truncate, by applying a left Bousfield localisation. We prefer the latter approach. We denote this construction $\tau_k Seg^{n+k}(\M)$.  While Simpson does not use $m$-operads to model $m$-tuply monoidal categories, we can do so.

\begin{corollary}\label{cor:tamsamani}
Let $\M$ be a left proper, combinatorial, monoidal model category with cofibrant unit. Then $\tau_k Seg^{n+k}(\M)$ satisfies the conditions of Corollary \ref{truncatedstabilization} and Theorem \ref{BD stabilisation}. In particular, the Baez-Dolan stabilisation hypothesis is true for Tamsamani weak $n$-categories $Seg^n(Set)$ and for higher Segal categories $Seg^n(sSet)$, just as in Corollary \ref{wkc}.
\end{corollary}

\begin{proof}
The model structure $\tau_k Seg^{n+k}(\M)$ is $k$-truncated by construction. That $k$-truncation is a {\em monoidal} left Bousfield localisation is an easy exercise, just as the case for truncation of spaces, spectra, and chain complexes considered in \cite{white-localization}. That the unit is cofibrant is again automatic since the class of cofibrant objects is unchanged by left Bousfield localisation.
\end{proof}

\begin{remark}
Simpson requires $\M$ to be a tractable, left proper, cartesian model category, and proves that $PC^n(\M)$ and $Seg^n(\M)$ are the same \cite[Theorem 19.3.2]{Simpson}. We note that cartesian model categories have cofibrant unit \cite[Definition 7.7.1]{Simpson}. While Simpson does not spell out that the localisation of a tractable model category (e.g., $Seg^n(\M)$) is tractable, this is clear, e.g., from \cite{white-localization}. The main reason to require left properness is for the existence of the localisation. We conjecture that this condition could be dropped, with Theorem \ref{thm:bous-loc-semi} used to produce a semimodel structure on $Seg^n(\M)$, where stabilisation could be proven. The main reason Simpson requires tractability is for better control over the localisation, but combinatoriality is sufficient for its existence. Simpson requires $\M$ to be cartesian so that it makes sense to enrich in $PC^n(\M)$, then establishes in \cite[Theorem 19.5.1, Section 20.5]{Simpson}, a connection to a $\M$-enriched $(n+1)$-category $n$CAT$(\M)$, the category of strict $\M$-enriched categories. We have seen above that `tractable' can be replaced with `combinatorial' and `cartesian' can be replaced with `monoidal with cofibrant unit,' and one still obtains model categories $PC^n(\M)$ and $Seg^n(\M)$ where the stabilisation hypothesis holds. However, when $\M$ is monoidal but not cartesian, the connection to enrichment is lost. 
\end{remark}

\begin{remark}
In \cite[Chapter 23]{Simpson}, Simpson proves a Baez-Dolan stabilisation result for $Seg^n(Set)$ and points out that his methods hold more generally (e.g., for $Seg^n(sSet)$ as well). Using a Whitehead construction \cite[23.1.1]{Simpson} (which is a left Bousfield localisation, analogous to our truncation), Simpson defines what it means for an $n$-category to be $m$-connected, then recognises that $(m-1)$-connected $(n+m)$-categories model $m$-tuply monoidal weak $n$-categories. We conjecture that our model for $m$-tuply monoidal weak $n$-categories, as algebras over $G_m$, is Quillen equivalent to Simpson's model category of $(m-1)$-connected $(n+m)$-categories. If this conjecture is true, then Corollary \ref{cor:tamsamani} recovers Simpson's Baez-Dolan result \cite[Theorem 23.0.3]{Simpson}. 
%define a $k$-truncated model category whose fibrant objects are $(k-1)$-connected objects of $Seg^{n+k}(\M)$. 
\end{remark}

We turn now to another model for weak $n$-categories: the $n$-quasi-categories of Ara \cite{Ara}. The idea is to iterate the construction of quasi-categories and end up with a model category of presheaves on $\Theta_n$, tightly connected (and Quillen equivalent to) Rezk's $\Theta_n$-spaces \cite[Theorem 8.4]{Ara}. One can left Bousfield localise to truncate Ara's model in exactly the same way we truncated Rezk's model, and so obtain a $k$-truncated model of $n$-quasi-categories, which we denote $\tau_k nQcat$.

\begin{corollary}
Ara's model category of $n$-quasi-categories satisfies the conditions of Theorem \ref{koperadicstabilization} and $\tau_k nQcat$ satisfies the conditions of Corollary \ref{truncatedstabilization} and Theorem \ref{BD stabilisation}. Hence, $n$-quasi-categories satisfy Baez-Dolan stabilisation.
\end{corollary} 

\begin{proof}
Ara's model category is cartesian, combinatorial, and has all objects cofibrant \cite[Theorem 2.2, Corollary 8.5]{Ara}. In particular, the unit is cofibrant and the model structure is left proper, hence the left Bousfield localisation $\tau_k nQcat$ exists and again has all objects cofibrant. Just as in \cite{Rezk} and \cite{white-localization}, the localised model structure is again cartesian.
\end{proof}

Several other models for weak $n$-categories have been proposed and compared in \cite{BergnerRezk}. For instance, Bergner introduced a model structure for Segal categories $SeCat_c$ that is cartesian, combinatorial, has all objects cofibrant, and whose fibrant objects are the Reedy fibrant Segal categories. It is Quillen equivalent to Rezk's complete Segal spaces, but cannot be iterated in the same way that Rezk's can (because the cartesian property will be lost). However, the Segal machinery (or analogous complete Segal machinery) can be applied to $\Theta_n Sp$ to obtain new models for $(n+1)$ categories, for every $n$. Bergner and Rezk introduce:
\begin{enumerate}
\item $\Theta_n Sp$-Segal categories, a combinatorial, cartesian model structure on functors $\Delta^{op}\to \Theta_n Sp$ whose fibrant objects satisfy a Segal condition \cite[Theorem 5.2]{BergnerRezk}. 
\item A combinatorial, cartesian model structure with all objects cofibrant, whose fibrant objects satisfy a subset of the conditions required of complete Segal spaces \cite[Proposition 5.9]{BergnerRezk}.
\end{enumerate}
Furthermore, they prove these two are equivalent to each other and to Rezk's $\Theta_n Sp$ \cite[Theorem 6.14, Corollary 7.1, Theorem 9.6]{BergnerRezk}. Lastly, both can be left Bousfield localised to make them $k$-truncated.

\begin{corollary}
Both of the Bergner-Rezk model structures above satisfy the conditions of Theorem \ref{koperadicstabilization}, and their truncations satisfy the conditions of Corollary \ref{truncatedstabilization} and Theorem \ref{BD stabilisation}. Hence, the Baez-Dolan stabilisation hypothesis is true for these models.
\end{corollary}

\begin{proof}
Both model structures are defined as left Bousfield localisations of the injective model structure, and hence are combinatorial and have all objects cofibrant (hence, are left proper). Hence, the truncation model structures exist. All four model structures are cartesian by \cite[Theorem 5.2, Proposition 5.9]{BergnerRezk}, and an argument analogous to the arguments above for truncations.
\end{proof}

\begin{remark}
We have proven the Baez-Dolan stabilisation hypothesis for several models of higher categories: all those we are aware of that possess a monoidal model structure. Possibly, other model of higher categories such as $n$-relative categories, $n$-fold Segal spaces, or simplicial categories, can be equipped with such a tensor product (e.g., higher analogues of the Gray tensor product of $2$-categories) but we are not aware about any complete work in this direction. If this structure appears in the future, our methods should immediately prove the Baez-Dolan stabilisation hypothesis for the corresponding model. Furthermore, any other model of weak $n$-categories that is homotopically equivalent to the models we have treated will automatically satisfy the Baez-Dolan stabilisation hypothesis on the homotopy category level.
\end{remark}

\noindent {\bf Acknowledgements.}    We   wish to express our  gratitude to C.Berger, R.Garner, E.Getzler, A.Joyal, S.Lack, M.Markl, R.Street,  M.Weber for many useful discussions, and to R.Griffiths and A.Campbell for encouraging us to include the application to Tamsamani weak $n$-categories and Ara's $n$-quasi-categories. We furthermore wish to thank the Oberwolfach Research Institute for Mathematics for hosting us for a week in September of 2021, where this helpful conversation with A.Campbell occurred. We also thank the referee for a very helpful report.

The first author is especially grateful to Denis-Charles Cisinski. Most ideas regarding locally constant presheaves belong to him. This paper would never be written without our long conversations and his illuminating explanations.  
The first author also  gratefully acknowledges  the financial
support and inspiring working atmosphere of  Max Plank
Institut f\"{u}r Mathematik, where a substantial part of the paper was written during his 2016 visit,  and Institut des Hautes \'Etude Scientifiques in Paris, where the paper was finished in January 2020.

The second author gratefully acknowledges the support of the National Science Foundation under Grant No. IIA-1414942, the Australian Academy of Science, and the Australian Category Theory Seminar. He is grateful to Macquarie University for hosting him on three occasions while we carried out this research.

\renewcommand{\refname}{Bibliography.}


\begin{thebibliography}{99}

\bibitem[Ara14]{Ara} Ara D., Higher quasi-categories vs higher Rezk spaces, 
\textit{Journal of K-theory} (2014), vol. 14, no 3, 701-749.

\bibitem[BD95]{BD} Baez J., Dolan J.,  Higher-dimensional algebra and topological quantum field theory, {\em Journal Math. Phys.} (1995), 36:6073-6105.

\bibitem[BD98]{BD3} Baez J., Dolan J.,  Higher-Dimensional Algebra III: n-Categories and the Algebra of Opetopes, {\em Adv. Math.} (1998),135:145-206.

\bibitem[Bar10]{Barwieck} Barwick C., 	\newblock On left and right model categories and left and right {B}ousfield localizations.
\newblock {\em Homology, Homotopy Appl.} (2010), 12(2):245--320.

\bibitem[BBPTY16a]{tillmann1} Basterra M., Bobkova I., Ponto K., Tillmann U., Yeakel S., \newblock Inverting operations in operads. \newblock {\it Topology and its Applications}, 235 (2017), 130-145. 



\bibitem[Bat98]{BatM} Batanin M.A.,  Monoidal globular categories as a natural environment for the theory of weak n-categories, {\it Adv. Math.} (1998), 136:39-103.

\bibitem[Bat07]{SymBat} Batanin M.A.,   The symmetrisation of $n$-operads and compactification of real configuration spaces, {\it Adv. Math.}  
(2007), 211:684-725.
 

\bibitem[Bat08]{EHBat} Batanin M.A.,   The Eckmann-Hilton argument and higher operads, {\it Adv. Math.} (2008), 217:334-385.

\bibitem[Bat10]{LocBat} Batanin M.A.,  Locally constant $n$-operads as higher braided operads, {\it J. of Noncommutative
Geometry} (2010), 4:237-265.

\bibitem[Bat17]{batanin-baez-dolan-via-semi} Batanin M.A., \newblock An operadic proof of the {B}aez-{D}olan stabilisation hypothesis,  {\em Proceedings of the AMS} (2017), 145, 2785-2798. 

\bibitem[BB17]{batanin-berger} Batanin M.A., Berger C., \newblock Homotopy theory for algebras over polynomial monads, \textit{Theory and Application of Categories} (2017), 32, No. 6, 148-253.

\bibitem[BB09]{bbl} Batanin M.A., Berger C., \newblock The Lattice Path Operad and  Hochschild cochains,  \textit{ Contemporary Math., AMS} (2009),  vol. 504, p. 23-52.

\bibitem[BBM11]{bbm} 
Batanin M.A., Berger C., Markl M.,
\newblock { Operads of natural operations I: Lattice paths, braces and Hochschild cochains,} \textit{ S\'{e}minaire and Congr\`{e}, } (2011), Collection SMF, 26, pp.1-33.

\bibitem[BL19]{batanin de leger} Batanin M.A., De Leger,\ F., Polynomial monads and delooping of mapping spaces, \newblock \textit{ J. Noncommut. Geom.} 13 (2019), 1521-1576. 
%Available as  {\tt arXiv:1712.00904}.

\bibitem[BKW18]{BKW} Batanin M.A., Kock J., Weber M., Regular patterns, substitudes, Feynman categories and operads, {\it Theory and Application of Categories} (2018), 33, 6-7, p.148-192. 

\bibitem[BW15]{batanin-white-CRM} Batanin M.A. and White D.,
\newblock Baez-Dolan Stabilization via (Semi-)Model Categories of Operads,  in ``Interactions between Representation Theory, Algebraic Topology, and Commutative Algebra,'' {\em Research Perspectives CRM Barcelona}, Volume 5 (2015), pages 175-179, ed. Dolors Herbera, Wolfgang Pitsch, and Santiago Zarzuela. Birkh{\"a}user.
%, DOI 10.1007/978-3-319-45441-2.


\bibitem[BW21]{batanin-white-eilenberg-moore} Batanin M.A. and White D., Left Bousfield localization and Eilenberg-Moore Categories, \textit{Homology, Homotopy and Applications} (2021), 23(2), pp.299-323. 
  
  \bibitem[BW20]{bous-loc-semi} Batanin M.A. and White D., Left Bousfield localization without left properness, available as arXiv:2001.03764, 2020.

\bibitem[BK17]{BergerK} Berger C., Kaufman R., Comprehensive factorisation systems, {\it Tbilisi Math. Journal} (2017), 10(3), p.255-277.  

\bibitem[BM03]{BergerMoerdijk} Berger C., Moerdijk I., Axiomatic homotopy theory of operads,  {\it Comment. Math. Helv.} (2003), 78(4):805--831.

\bibitem[BR14]{BergnerRezk} Bergner J. E. and Rezk C., Comparison of models for (1, n)-categories, II, to appear in \textit{Journal of Topology}, available as arXiv:1406.4182.


\bibitem[Cis06]{cis06} Cisinski D.C., \textit{Les pr\'{e}faisceaux comme mod\'{e}les des types d'homotopie} (French, with English and French summaries), Ast\'{e}risque 308, xxiv+390, 2006.

\bibitem[Cis09]{cis} Cisinski D.C., Locally constant functors. {Math. Proc. Camb. Phil. Soc.} (2009), 147:593-614.

\bibitem[DS03a]{DS1} Day B., Street R., 
Lax monoids, pseudo-operads, and convolution, in: ``Diagrammatic Morphisms and Applications", \textit{Contemporary Mathematics} (2003), 318:75--96.

\bibitem[DS03b]{DS2} Day B., Street R., Abstract substitution in enriched categories, \textit{J. Pure Appl. Algebra} 179 (2003) 49--63.

\bibitem[Dug03]{dag} Dugger D., Replacing model categories by simplicial ones.
\textit{Transactions of the American Mathematical Society} (2003), Volume 353, Number 12, Pages 5003-5027.

\bibitem[DK80]{DwyerKan}  Dwyer W.,  Kan D., {Simplicial localization of categories,} \textit{Journal of pure and applied algebra} (1980), 17:267-284


\bibitem[FGHW08]{species} Fiore M.,
Gambino N.,
Hyland M.,
Winskel G., {The cartesian closed bicategory of generalised species of structures}, {\it Journal of the London Math Society} (2008), Volume 77, Issue 1, Pages 203-220.

\bibitem[Fre09]{fresse-book}
Fresse B.,
\newblock {\em Modules over operads and functors}, volume 1967 of { Lecture Notes in Mathematics}.
\newblock Springer-Verlag, Berlin, 2009.

\bibitem[Fre12]{Fresse} Fresse B.,
\newblock {\em Homotopy of Operads \& Grothendieck-Teichm\"uller Groups}, Part 1 and 2, {Mathematical Surveys and Monographs},
\newblock AMS, volume 217, 2017.

\bibitem[GK13]{GK} Gambino N., Kock J., {Polynomial functors and polynomial monads}, {\it Math. Proc. Cambridge Philos. Soc.} {154} (2013), no. 1, 153--192.

\bibitem[Get10]{G} Getzler E., {Operads Revisited}, {\it Algebra, Arithmetic and Geometry, Progress in Mathematics} (2010), 269, Springer-Verlag, 675-698. 
 
 \bibitem[Gui80]{Guitart} Guitart R., {Relations et carr\'es exacts,} {\it Ann. Sc. Math. Qu\'ebec} {IV} (1980), 103--125.

 
 
\bibitem[Hir03]{hirschhorn} Hirschhorn P., 
\newblock {\em Model categories and their localizations}, volume~99 of {\em
  Mathematical Surveys and Monographs} (2003).
\newblock American Mathematical Society, Providence, RI.

\bibitem[Hov99]{hovey-book} Hovey M.,
\newblock {\em Model categories}, volume~63 of {\em Mathematical Surveys and
  Monographs} (1999).
\newblock American Mathematical Society, Providence, RI.



\bibitem[Kel74]{club} Kelly G. M., On clubs and doctrines. In Category Seminar (Proc. Sem., Sydney, 1972/1973), pages 181-256. \textit{Lecture Notes in Math.} (1974), Vol. 420. Springer, Berlin.

\bibitem[Koc11]{Kock} Kock J., Polynomial functors and trees, {\it  Int. Math. Res. Notices} {3} (2011), 609--673.


\bibitem[Lav97]{lavers} Lavers T.G., \newblock{The theory of vines}. 
\newblock{ {\it Communications in Algebra}} (1997), 25:1257--1284.


\bibitem[Mal12]{Malt} Malsiniotis G., {Carr\'e exacts homotopiques, et d\'erivateurs,} {\it Cahiers de Top. et G\'eom. Diff. Cat\'egoriques} (2012), LIII, (1), 3-63.

\bibitem[Mal05]{Maltsiniotis} Malsiniotis G., {\it La th\'eories de l'homotopie do Grothendieck,} Ast\'erisque, vol. 301, Soc. Math. France, 2005. 



\bibitem[MS04b]{mcclure-smith} McClure J. and Smith J., Cosimplicial Objects and little n-cubes I. {\em Amer. J. Math.} (2004), 126(5):1109-1153.


\bibitem[McD79]{McDuff} McDuff D.,  On the classifying spaces of discrete monoids. \textit{Topology} (1979), Volume 18, Issue 4, pp. 313-320.


\bibitem[Mur11]{Muro} Muro F., Homotopy theory of nonsymmetric operads, \textit{Algebr. Geom. Topol.} 11 (2011), no. 3, 1541-1599.

\bibitem[Pet12]{Pet} Petersen D., On the operad structure of admissible $G$-covers, \textit{Algebra Number Theory} (2013), 7 (8), pp. 1953-1975.

\bibitem[Sim12]{Simpson} Simpson C.,
\newblock {\em Homotopy theory of higher categories}, vol. 19 of {\em New Mathematical Monographs}.
\newblock Cambridge University Press, Cambridge, 2012.

\bibitem[Spi01]{spitzweck} Spitzweck M., {\it Operads, Algebras and Modules in Model Categories and Motives}, PhD thesis, Bonn, 2001. Available electronically from http://arxiv.org/abs/math/0101102.

\bibitem[SS00]{SS}  Schwede S.,  Shipley B.E., {\it Algebras and modules in monoidal model categories.}{\it  Proc. London Math. Soc.} (2000), 3:491-511.

\bibitem[SZ14]{Zav}Szawiel S., Zawadowski M.,Theories of analytic monads,  \emph{Math. Struct. Comp. Sci.} (2014), 24(6):e240604, 33.

\bibitem[Rez10]{Rezk} Rezk C., A Cartesian presentation of weak $n$-categories, \textit{Geom. Topol.} (2010), 14:521-571.


\bibitem[RSS01]{RSS} Rezk C., Schwede S., Shipley S., Simplicial structures on model categories and functors, \textit{American Journal of Mathematics} (2001), 123(3), 551-575.

\bibitem[Web16]{mark weber} Weber M., Algebraic Kan extensions along morphisms of internal algebra classifiers, \textit{Tbilisi Mathematical Journal}, 9(1), (2016), pp. 65-142. 

\bibitem[Whi21a]{white-localization}
White, D.,
\newblock Monoidal {B}ousfield localizations and algebras over operads, Equivariant Topology and Derived Algebra, Cambridge University Press (2021), 179-239. 

\bibitem[Whi21b]{white-oberwolfach} White D., Substitudes, Bousfield localization, higher braided operads, and Baez-Dolan stabilization, \textit{Mathematisches Forschungsinstitut Oberwolfach}, Number 46 (2021): Homotopical Algebra and Higher Structures.

\bibitem[WY18]{white-yau1} White D., Yau D.,
\newblock {B}ousfield localizations and algebras over colored operads, 
\newblock {\em Applied Categorical Structures} (2018), 26:153-203.

\bibitem[WY19]{white-yau3} White D., Yau D., Homotopical adjoint lifting theorem, {\em Applied Categorical Structures} (2019), 27:385-426.



\end{thebibliography}
\end{document}